\documentclass[10pt]{amsart}
\usepackage{2025macros}

\usepackage{todonotes}

\usepackage{parskip}
\usepackage{eucal}
\newcommand{\edit}[1]{\todo[noline,color=violet!35,linecolor=blue!40!black,size=\small]{EDITS HERE}}

\begin{document}

\title{Chern--Simons factorization algebras and knot polynomials}

\author{
  Kevin Costello
  \and
  John Francis
  \and
  Owen Gwilliam
}

\thanks{J.F. was supported by the National Science Foundation under awards 0902974, 1207758, and 1812057. Part of this work was done while J.F. and O.G. were in residence at the Mathematical Sciences Research Institute in Berkeley, California, during the Spring 2020 semester, which was supported by the National Science Foundation under Grant No. DMS-1440140. The National Science Foundation supported O.G. through DMS Grant 2042052 during the last stages of writing of this paper.}

\begin{abstract}
This work identifies the Reshetikhin--Turaev invariant of links in terms of a trace map on factorization homology. 
In particular, to recover the knot invariants associated to Chern--Simons theories,
we construct a filtered $\cE_3$-algebra $\cA^\lambda$ by BV quantization of Chern--Simons theory for a semi-simple Lie algebra ${\frak g}$ with invariant pairing~$\lambda$, 
and we prove that a finite-dimensional representation $V$ of the Drinfeld--Jimbo quantum group $\U_\hbar{\frak g}$ defines a perfect $\cA^\lambda$ module~$V$. 
For any framed link $K$ in $\RR^3$, we then prove that there is an equality
\[
\int_{K\subset\RR^3}{\rm tr}(V) = Z_V(K\subset\RR^3)
\]
between the factorization homology trace for $V$ and the Reshetikhin--Turaev link invariant determined by~$V$.
\end{abstract}



\maketitle

\tableofcontents

\section{Introduction}

In the seminal paper \cite{WitCS}, Witten showed that the Jones polynomial of a knot can be realized in a Chern--Simons quantum field theory 
as the expected value of a Wilson operator on the knot. 
To be more precise, 
Witten argued that if $K \subset S^3$ is a knot, 
then the Jones polynomial is encoded in the following functional integral expression:
$$
J(K, e^{2 \pi i/ k})= \int_{A \in {\rm Conn}(M) / {\rm Gauge} } e^{i k S_{CS}(A)} {\rm Tr}_{V}({\rm Hol}_{K}(A)). 
$$
Here the Jones polynomial $J(K,q)$ depends on a parameter $q$, 
which we set as $q=e^{2 \pi i / k}$, to relate to the level appearing in the Chern--Simons theory. 
This putative integral is over a quotient space, namely the infinite-dimensional space of $SU(2)$-bundles with connection on $S^3$ modulo gauge. 
The integrand involves the Chern--Simons functional $S_{CS}(A)$ of a connection, labeled $A$; 
the integer $k$ is the level; $V$ is the fundamental representation of ${\rm SU}(2)$;
and the final term is the trace over $V$ of the holonomy of $A$ around the knot~$K$.
In physical language, that last term is the Wilson loop operator,
and the overall integral computes the expected value of the Wilson loop operator.

This functional integral expression is heuristic, because the functional measure is not well-defined.  
But instead of trying to define this functional integral, 
Witten derived the Jones polynomial from his analysis of formal properties of Chern--Simons theories.

Around the same time, Reshetikhin and Turaev \cite{RT} gave a general construction of knot invariants starting with a quantum group.    
Reshetikhin and Turaev's invariants are more general than the Jones polynomial, and 
they are defined whenever one has a compact Lie group $G$ with a finite-dimensional representation~$V$.  

Witten's approach also yields knot invariants depending on the same data: 
the group $G$ is the gauge group for Chern--Simons theory, and 
the Wilson operator on the knot uses the trace of the holonomy in the representation~$V$.
(The Jones polynomial is a special case.) 

It has long been believed that Witten's and Reshetikhin--Turaev's knot invariants are the same.  
It has been difficult to verify this belief, however, 
because Witten's constructions rely on non-rigorous arguments about infinite-dimensional functional integrals. 

In this paper we prove, in perturbation theory, a version of this equivalence of invariants. 
The novel mechanism in our work is the theory of factorization homology,
which captures the behavior of observables for perturbative topological field theories.
To motivate the statement of our central result,
we sketch the dictionary that guides our work.

\begin{rmk}
After \cite{WitCS}, a steady stream of work in physics and mathematics has explored Chern--Simons theory,
so that there is now an ocean of literature.
We will not attempt to map it, 
focusing instead on work with which we are familiar and which has connections with this paper.\footnote{Undoubtedly we have overlooked important prior work, despite an effort to engage the literature,
and we welcome feedback on this front.}
For a relatively recent and rather magisterial survey of Chern--Simons theory from the perspective of mathematics,
see~\cite{FreAMS}.
\end{rmk}

\subsection{A guiding dictionary and our central result}

Before delving into intricacies,
we wish to explain the basic dictionary underlying our work.
The coarsest version is
\begin{center}
\begin{minipage}{.3\textwidth}
{\fontfamily{qag}\selectfont
path integral of a TFT}
\end{minipage}
$\longleftrightarrow\quad$ 
\begin{minipage}{.4\textwidth}
{\fontfamily{qag}\selectfont
factorization homology}
\end{minipage}
\end{center}
and hence we should be able to convert constructions with Chern--Simons theory
on 3-manifolds into factorization homology computations of an $\cE_3$-algebra,
the local observables of the Chern--Simons theory.
(Later in this introduction, we refine this correspondence considerably.)
But the Jones polynomial involves an elaboration of that dictionary,
in two steps.

First,
given a connection $\nabla$ on a 3-manifold $M$ and a knot $K \subset M$,
the trace of holonomy of $\nabla$ around $K$ is itself the path integral for a field theory on the knot
(i.e., a mechanical system with periodic time) that depends on~$\nabla$.
On the physics side we have a TFT coupled to a 1-dimensional defect TFT,
and it corresponds to a factorization algebra on a stratified manifold ${K \subset M}$.
Under our dictionary that means
\begin{center}
\begin{minipage}{.35\textwidth}
{\fontfamily{qag}\selectfont
path integral for the TFT\\coupled to 1d TFT }
\end{minipage}
$\longleftrightarrow\quad$ 
\begin{minipage}{.4\textwidth}
{\fontfamily{qag}\selectfont
factorization homology on \\a stratified manifold $K \subset M$}
\end{minipage}
\end{center}
where this factorization homology involves the $\cE_3$-algebra of Chern--Simons theory on $M \setminus K$ and an $\cE_1$-algebra along~$K$.

Second, we can describe the 1-dimensional QFT using quantum mechanics with a finite-dimensional Hilbert space,
so that the partition function has an alternate description as the trace of an operator.
This admits a factorization version as an $\cE_3$-algebra (for Chern--Simons theory) acting on an $\cE_1$-algebra (for the 1d theory) acting on a vector space (the Hilbert space).
A {\em key result} of this paper is that, because the Hilbert space is finite-dimensional, there is a natural isomorphism
\begin{center}
\begin{minipage}{.4\textwidth}
{\fontfamily{qag}\selectfont
factorization homology on \\a stratified manifold $K \subset M$}
\end{minipage}
$\xto{\cong}\quad$ 
\begin{minipage}{.4\textwidth}
{\fontfamily{qag}\selectfont
factorization homology on $M$}
\end{minipage}
\end{center}
essentially because factorization homology is Morita-invariant.

The Jones polynomial is the expected value of a particular operator in this coupled TFT,
so it is computed by a path integral of a 1d TFT coupled to Chern--Simons theory. 
By our dictionary that means it is in the image of our isomorphism:  
\begin{center}
\begin{minipage}{.25\textwidth}
{\fontfamily{qag}\selectfont
Jones polynomial\\ for a knot $K \subset M$}
\end{minipage}
$\longleftrightarrow\quad$ 
\begin{minipage}{.4\textwidth}
{\fontfamily{qag}\selectfont an element $J_K$ in factorization \\homology on $M$}
\end{minipage}
\end{center}
so our goal is to characterize that operator and its image in factorization homology.

To put things on a solid footing, 
we introduce an invariant of framed links $K\subset M$ in terms of factorization homology
\[
\int_{K\subset M} \tr(V)
\]
given in terms of an $\cE_3$-algebra $A$ and a perfect $A$-module~$V$.
We use the perturbative quantization of Chern--Simons theory to produce such a pair $(A,V)$ of inputs.
Our main theorem is then the following.

\begin{thm}
\label{thm: central}
For each invariant pairing on a semisimple Lie algebra $\g$, 
BV quantization of Chern--Simons theory gives a filtered $\cE_3$-algebra $\cA^\lambda$. 
A choice of finite-dimensional representation $V$ of the Drinfeld--Jimbo quantum group $\U_\hbar\g$ defines a perfect $\cA^\lambda$ module.
In terms of this data, 
for any framed link $K\subset \RR^3$, there is an equality
\[
\int_{K\subset\RR^3}\tr(V) = Z_V(K\subset \RR^3)
\]
between the factorization homology trace and the Reshetikhin--Turaev invariant of the link.
\end{thm}

In the remainder of this introduction we explain the constituent parts of this result, 
and how the factorization homology invariant $\int_{K\subset M} \tr(V)$ can be seen as a formalization of Witten's invariant.

\subsection{Sharpening the question}

To clarify what we accomplish, 
it is helpful to review what is already known in this direction.

Soon after Witten introduced the Chern--Simons theory,
its perturbative quantization was explored by many physicists and mathematicians
(see, for instance, \cite{GMM,GuaBook,AxeSinI,AxeSinII,KonECM,BarCS}).
The challenge of renormalization is very mild here, 
allowing important structural work to be done. 
Moreover, a careful treatment of the Wilson loop observables
led to a new perspective on Vassiliev invariants and surprising connections to many facets of geometric topology (see, e.g., \cite{BarVI,AltFreVCS,BotTau,VolSurv}).

We review this work in Section \ref{sec: CS},
where we explain how it fits into the Batalin--Vilkovisky (BV) formalism,
as articulated in~\cite{CosBook}.
An aspect important to our discussion now is that one can construct a perturbative quantization of a classical field theory by obstruction theory.  
That is, the problem of constructing a field theory at the quantum level becomes a cohomological one: 
one needs to calculate the obstruction and deformation groups, 
which are respectively $\H^1$ and $\H^0$ of a certain cochain complex,
and one needs to identify the obstruction cocycle. 

For Chern--Simons theory on $\RR^3$, the obstruction and deformation groups that appear are, respectively, $\H^4(\g)$ and $\H^3(\g)$. 
For any simple Lie algebra $\g$, $\H^4(\g)$ vanishes, so there is never any obstruction. 
Since $\H^3(\g)$ is 1-dimensional, at each order in $\hbar$ there is a single parameter governing deformations. 
This parameter can be realized by having a $\hbar$-dependent {\em level} 
\[
\lambda = \lambda^\cl + \hbar \lambda_1 + \hbar^2 \lambda_2 + \cdots,
\]
where $\lambda^\cl$ is a nondegenerate invariant symmetric bilinear form on $\g$
and the perturbative choices are invariant symmetric bilinear forms.  Since we assume $\g$ is simple, the space of symmetric invariant bilinear forms on $\g$ is 1-dimensional. Note here an important difference with Witten's work: 
Witten requires the level to be integral (or more accurately, to live in $\H^4({\rm B} G,\ZZ)$), whereas the perturbative setting requires no such condition.

Thus, one obtains the following result,
which was implicit in the work of Axelrod--Singer and Kontsevich and articulated fully in~\cite{CosCS,CosBook}.

\begin{thm}[Thm. 14.2.1, Chapter 5, \cite{CosBook}]
\label{thm: levels for CS}
The classical BV theory for Chern--Simons theory on $\RR^3$ with Lie algebra $\g$ and pairing $\lambda^\cl$ admits BV quantizations.
The set of BV quantizations (to all powers of $\hbar$) modulo equivalences is a torsor over $\hbar\CC[\![\hbar]\!] \otimes \H^3(\g)$.
\end{thm}
The proof of this result is rather straightforward once the formalism of \cite{CosBook} is established: no explicit analysis of Feynman diagrams is needed. 

There is also a ``perturbative'' facet to the quantum group approach.
Reshetikhin and Turaev need only a rigid braided monoidal category\footnote{Recall that a rigid braided monoidal category is a braided monoidal category in which every object has a dual.  Reshetikhin and Turaev construct a knot invariant given a choice of dualizable object in braided monoidal category;
the dualizable objects in a braided monoidal category provide such a rigid category.}  to recover functorial invariants of framed tangles \cite{RT},
and such are provided by the quantum groups with formal parameter $\hbar$,
also known as quantum enveloping algebras:
take the finite-rank topologically free representations of~$\U_\hbar \g$.\footnote{The later work of Reshetikhin and Turaev \cite{RT3man} requires a much stronger input---namely, a modular tensor category---to recover functorial invariants of oriented 3-manifolds.
These 3-manifold invariants are not amenable to the perturbative treatment of this paper.} 
This example is not, however, the only ``quantization'' of~$\U\g$. 

Drinfeld classified the deformations of $\Rep_\fin(\g)$, the symmetric monoidal ribbon\footnote{The ribbon structure is a natural automorphism on the identity functor, which corresponds to working with framed tangles.} category of finite-dimensional representations of $\g$, in his influential work \cite{DriQHA,DriACHA,DriGal}. 
(For a general discussion of deformations of such categories, see Section 7.22 of~\cite{EGNO}. 
For nice expositions of Drinfeld's result, see chapter 16 both of \cite{ChaPre} and of~\cite{EtiSch}.)

\begin{thm}[Drinfeld]
The symmetric monoidal ribbon category $\Rep_\fin(\g)$ admits nontrivial deformations as a ribbon category.
The set of deformations over $\CC[\![\hbar]\!]$ modulo ribbon equivalences is a torsor over $\hbar\CC[\![\hbar]\!] \otimes \H^3(\g)$.
\end{thm}

Hence, we arrive at a central question addressed in this paper: 
\begin{quotation}
Is there a natural isomorphism between these two spaces of deformations?
\end{quotation}
More specifically, we ask {\em how} does each choice of BV quantization of Chern--Simons theory determine a deformation of $\Rep_\fin(\g)$ as a rigid braided monoidal category?

The answer is that such a natural isomorphism exists, 
and it will be realized by methods arising from a general story about how quantum field theory relates to higher algebra and categories.
(One compelling aspect of our approach is that it is not tailor-made for Chern--Simons theory,
but has applications to other theories, including non-topological theories.
See \cite{CosYang,CosSM,CWY} for closely related work.)

\subsection{The appearance of higher algebra}

Our approach does not proceed by direct computation or an involved analysis of the Feynman diagrammatics.
Instead, it leverages  the rich formal properties of quantum field theory,
much as Witten did in relating the Jones polynomial to Chern--Simons theory,
but here formalized via {\em factorization algebras}.
This notion is the novel ingredient in this work,
and it provides a bridge from perturbative quantum field theory to higher algebra.

In \cite{CG1,CG2}, it is shown how a perturbative quantum field theory is encoded in a factorization algebra over $\CC[\![\hbar]\!]$  called the {\it factorization algebra of observables} of the theory.  
The factorization algebra encodes the structures present on observables, such as the operator product and (in good cases) the correlation functions.
Such a factorization algebra assigns to each open set $U$ of the space-time manifold 
(e.g., for Chern--Simons theory, an oriented 3-manifold),
a cochain complex $\Obs(U)$ consisting of the \emph{observables} or \emph{operators} of the quantum field theory on $U$, i.e., the measurements one can make on $U$ in the field theory.
(It is a cochain complex because we are working in the BV formalism,
which is homological in nature.)

The basic structure of a factorization algebra is provided by the factorization product map. If $U_1,\dots,U_k$ are disjoint opens all contained in an open $V$, one has a map
\begin{equation} 
	 m_V^{U_1,\ldots,U_k}: \Obs(U_1)\otimes\cdots\otimes\Obs(U_k)\to\Obs(V)
\end{equation}	
as illustracted in figure \ref{fig:factorization}.  These maps must satisfy some straightforward axioms detailed in \cite{CG1}.
\begin{figure}
\begin{center}
 \begin{tikzpicture}[scale=0.8]
 \draw[dotted,semithick] (0,0) circle (2.5);
 \draw (-0.5,1) circle(0.5) node {$U_1$};
 \draw (-1.2,-0.5) circle (0.5) node {$U_2$};
 \draw (-0.3, -1) node {\dots};
 \draw (1.1,-1) circle (0.8) node {$U_k$};
 \draw (1.3, 1.5) node {$V$};
\node at (8,0){$\rightsquigarrow \quad m_V^{U_1,\ldots,U_k}: \Obs(U_1)\otimes\cdots\otimes\Obs(U_k)\to\Obs(V)$};
 \end{tikzpicture}
\end{center}
	\caption{ The factorization product map. \label{fig:factorization}}
\end{figure}
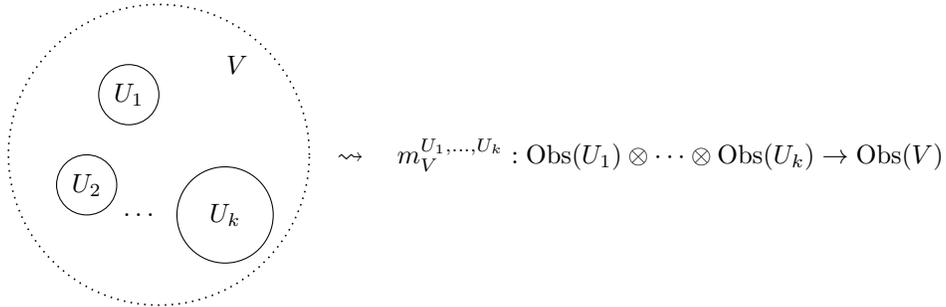
For a topological field theory, such as Chern--Simons theory,
this factorization algebra is {\em locally constant}:
if $D \subset D'$ is a pair of open balls,
then the map of extending observables from smaller to larger open sets
\[
\Obs(D) \to \Obs(D')
\]
is a quasi-isomorphism.
This property captures the fact that the geometry is irrelevant and only the topology matters.

A result of Lurie \cite{LurHA} says that locally constant factorization algebras on $\RR^3$ are the same as $\cE_3$-algebras,
i.e., algebras over the operad of little 3-disks.  (See Chapter 5, Section 4 of \cite{LurHA}, particularly Theorem 5.4.5.9.)    For instance, there is a 2-sphere's worth of ways of forming the product of observables, because the space of pairs of disjoint balls in a given 3-ball is homotopy equivalent to $S^2$.  

Therefore, by combining these existing results, we find the following. 
(See Section~\ref{subsec: nontrivial def}, notably Proposition~\ref{prp: CS quant exists}.)

\begin{prp*}
Fix a quantization $\lambda$ of Chern--Simons theory for a Lie algebra $\g$.  
Then the operator products between observables of this quantum theory are described by an $\cE_3$-algebra $\cA^\lambda$ over~$\CC[\![\hbar]\!]$. 
\end{prp*}

A finer analysis shows that observables form a filtered $\cE_3$-algebra  over the filtered commutative algebra $\CC[\![\hbar]\!]$, 
which is filtered by saying that $F^i \CC[\![\hbar]\!] = \hbar^{2i} \CC[\![\hbar]\!]$.  
(Since we are only sketching our arguments in this introduction, 
we will not go into details here about filtrations.)

This result allows us to analyze the observables of Chern--Simons theory using the powerful theory of $\cE_3$-algebras and their topological factorization homology.  
Topologists introduced $\cE_n$-algebras in the 1960's and 1970's to describe the structure on $n$-fold loop spaces: see \cite{BoaVogAMS,BoaVogLNM,MayLoop}.  
Topological forms of factorization homology were developed by Segal, McDuff, Salvatore, Lurie, Ayala--Francis \cite{Segal,McDuff,SalCS,LurHA,AF}.  
Early work on what we call topological factorization homology focused on its applications to the study of mapping spaces \cite{LurHA}.  
More recent work exhibits topological factorization homology as an important tool in the study of topological quantum field theories~\cite{LurTFT,AF,Sch,BBJ1,BBJ2}.

As we will see shortly, this theory will allow us to relate the $\cE_3$-algebra of observables of Chern--Simons theory to the quantum group.  

One can ask what is the filtered $\cE_3$-algebra of the proposition.  
The cochain complex underlying this $\cE_3$-algebra is the space of observables in Chern--Simons theory on a $3$-ball.  
In any field theory, classical observables are defined to be the space of functions on a certain derived version of the equations of motion. 
Because we are working in perturbation theory, we look at the derived space of solutions to the equations of motion near a given (trivial) solution. 
Since, for Chern--Simons theory, solutions to the equations of motion are flat bundles, 
we find that classical observables on a ball are functions on the formal derived stack of flat bundles on a ball.  
Since all flat bundles on a ball are trivial, this formal derived stack is the formal neighborhood of a point in the classifying stack of the gauge group $G$.  
This formal neighborhood can be thought of as the classifying space of the Lie algebra $\g$.  
Functions on ${\rm B}\g$ are $\g$-invariant functions on a point, 
which can be modelled homologically by the Chevalley--Eilenberg cochain complex $\C^*(\g)$ of $\g$ with coefficients in the trivial module. 
Thus, we have sketched the following. See Section~\ref{sec: classical observables}, notably Lemma~\ref{poin}, for a fuller discussion.

\begin{lmm*}
The classical observables of Chern--Simons theory are described by the filtered commutative algebra~$\C^*(\g)$, the Chevalley--Eilenberg cochains of the Lie algebra~$\g$.
\end{lmm*}

The filtration is defined by setting $F^i \C^*(\g) = \Sym^{\ge i} (\g^\vee[-1])$.

The filtered $\cE_3$-algebra describing quantum observables is a filtered deformation of $\C^*(\g)$. 
It will be given by an $\cE_3$ structure on the filtered cochain complex $\C^*(\g)[\![\hbar]\!]$, 
which modulo $\hbar$ is just the commutative algebra structure on~$\C^*(\g)$.

\subsection{Why are there local observables in Chern--Simons theory?}

We address here an aspect that would bother a physicist.
In a physics treatment of Chern--Simons theory, one would normally say that there are no local observables and hence a physicist would ask why we find an interesting $\cE_3$-algebra.

The observables we are finding here---which form a copy of the exterior algebra on $\g^\ast$---can be described, in physical terminology, as being built from the ghosts for constant gauge transformations.  

In physics, the standard procedure (with a compact gauge group) does {\em not} introduce ghosts for constant gauge transformations.  
From a mathematics point of view, this choice is reasonable because the functor of taking invariants for the action of a compact Lie group is an exact functor, 
and so it does not need to be replaced by a derived functor. 
By contrast, the functor of taking invariants for the action of a Lie algebra is not exact, 
and hence has a derived functor, namely the Chevalley--Eilenberg cochains.  
Ghosts are introduced in the process of taking the derived functor of invariants for the action of the infinite-dimensional Lie algebra of gauge transformations.  
We do not need to introduce ghosts for constant gauge transformations because they arise from the action of a compact Lie group.

Our treatment of Chern--Simons theory is purely perturbative, however, and hence only sees the Lie algebra. 
In fact, we only need the complex form of the Lie algebra: 
a choice of real form plays a role in the contour of integration for the path integral, 
but this contour is irrelevant in perturbation theory.  
In a perturbative treatment like ours, it is thus necessary to introduce ghosts for constant gauge transformations.  

\subsection{Connecting with categories}
\label{sec: connect with cats}

Our discussion so far has shown how the problem of analyzing the structure of operator products in Chern--Simons theory (which, at first sight, involves calculation of many Feynman diagrams) can be turned into an algebraic problem: 
understanding deformations of $\C^*(\g)$ as a filtered $\cE_3$-algebra. 

It turns out that there is an intimate relationship between such $\cE_3$ deformations and the quantum group.  

It is a result of \cite{LurHA} that for any $\cE_n$-algebra $\cA$, 
the $\infty$-category of left modules $\LMod_\cA$---obtained 
by ``forgetting'' $\cA$ down to its underlying $\cE_1$ (or homotopy associative) algebra and taking its left 
modules---inherits a natural $\cE_{n-1}$-monoidal structure from~$\cA$.
(Simply combine Theorem 4.8.5.5 and Corollary 5.1.2.6 from \cite{LurHA}.)
Since the $\cE_3$-algebra $\cA^\lambda$ of quantum observables of Chern--Simons deforms the $\cE_\infty$ (or commutative) algebra $\C^*(\g)$, 
we find that the braided monoidal category of left modules will deform the symmetric monoidal category of (filtered) $\C^*(\g)$-modules.  

This category $\LMod_{\C^*(\g)}$ of left modules is closely related to a category of more obvious importance to us: representations of $\g$.
Indeed, every representation $M$ of $\g$ determines the $\C^*(\g)$-module $\C^*(\g,M)$.
In \cite{CosYang}, a filtered version of Koszul duality is developed for just this kind of situation,
and it is shown that there is an equivalence of symmetric monoidal $\infty$-categories 
\begin{equation}
\label{filtKD}
\Perf_{\C^*(\g)} \simeq \Rep_\fin(\g)^{\dg}
\end{equation}
by a filtered version of Koszul duality (see Corollary~\ref{filt KD}).
In brief, if $M$ is of finite-type, then $\C^*(\g,M)$ is perfect.
Note that on the left side we use tensor product over $\C^*(\g)$, 
whereas on the right side we use tensor product over the base field.
(See Section~\ref{sec: filtered KD} for more exposition.)

Thus, a nontrivial deformation of $\C^*(\g)$ as an $\cE_3$-algebra
determines a nontrivial deformation of $\Perf_{\C^*(\g))}$ as an $\cE_2$-monoidal $\infty$-category,
and hence a nontrivial $\cE_2$-monoidal deformation of~$\Rep_\fin(\g)^{\dg}$.

The usual category $\Rep_\fin(\g)$ sits as a full subcategory inside $\Rep_\fin(\g)^{\dg}$, 
and it is manifest from the construction that this subcategory is preserved by the tensor product and its braided monoidal deformations.
Hence, we recover nontrivial braided monoidal deformations of the underlying category, as well.

To summarize, thanks to higher abstract nonsense developed over the last few decades,
there is a direct pipeline from the setting of perturbative quantum field theory to the setting of higher categories.
This pipeline explains, in structural terms, why the perturbative quantization of Chern--Simons theory produces a braided monoidal deformation of the category of representations.
(We emphasize that this pipeline applies also to a {\em generic} topological field theory arising from an action functional.)

\subsection{The main identification}

According to Drinfeld \cite{DriACHA,DriQHA}, another algebraic structure also produces braided monoidal deformations of the category of finite-dimensional $\g$-modules. 
If $\U_{\hbar}(\g)$ is a quasi-triangular quasi-Hopf algebra over $\CC[\![\hbar]\!]$ that quantizes $\U(\g)$, then the category of these $\U_{\hbar}(\g)$-modules is a braided monoidal (in fact, ribbon) deformation of the category of $\g$-modules.  
If we twist this quasi-triangular quasi-Hopf algebra, in the sense of Drinfeld,
then the braided monoidal category does not change.

Assembling our discussion thus far, we obtain the first main result of this paper.

\begin{thm}
\label{first main theorem}
Let $\g$ be a semisimple Lie algebra.
There is a canonical bijection between the following spaces:
\begin{enumerate} 
 \item Perturbative quantizations of Chern--Simons theory in the sense of~\cite{CosBook}. 
 \item Filtered $\cE_3$-algebras quantizing $\C^*(\g)$ with fixed semi-classical behavior. 
 (That is, we fix the shifted Poisson bracket on $\C^*(\g)$ that encodes the failure of the quantized $\cE_3$-algebra to be commutative modulo~$\hbar$.) 
 \item Braided monoidal deformations of the symmetric monoidal category of finite-dimensional $\g$-modules, where the braiding has fixed semi-classical behavior. 
 \item Quasi-triangular quasi-Hopf algebras quantizing $\U(\g)$ with fixed semi-classical behavior, up to Drinfeld twist. (The semi-classical behavior is described by a quasi-triangular quasi-Lie bialgebra structure on $\g$; this data also fixes the semi-classical behavior of the $\cE_3$-algebra quantizing~$\C^*(\g)$.) 
\end{enumerate}
All of these spaces are non-canonically isomorphic to~$\hbar \H^3(\g)[\![\hbar]\!]$. 
\end{thm} 

The relationship between (3) and (4) is due to Drinfeld.
Note that for (4), one must choose an associator to fix the isomorphism of deformations with $\hbar \H^3(\g)[\![\hbar]\!]$. 
Similarly, for (1) one must make a choice about how to do regularization (e.g., to use the configuration space method of Axelrod--Singer--Kontsevich, or a renormalization scheme in the sense of \cite{CosBook}) to fix the isomorphism with~$\hbar \H^3(\g)[\![\hbar]\!]$.

Our discussion above has explained that (1) leads to (2): 
a perturbative quantization of Chern--Simons theory leads to a filtered $\cE_3$-algebra quantizing $\C^*(\g)$, and its category of perfect left modules is a braided monoidal category quantizing the category of $\g$-modules.  
A deformation theory calculation shows that all of these maps are equivalences; 
the point is that at each order in $\hbar$, the obstruction-deformation complex for each quantization problem is controlled by~$\H^\ast(\g)[3]$.   

The connection from (2) to (3) was explained in Section~\ref{sec: connect with cats}: it combines Lurie's results on left modules for  $\cE_3$-algebras with the filtered Koszul duality of Corollary~\ref{filt KD}.

\begin{rmk}
Note that a Drinfeld associator lifts the semiclassical data to the quantum level,
so that once an associator is chosen, there is a canonical equivalence between (3) and~(4).
\end{rmk}

\subsection{Wilson loops and modules for an $\cE_3$-algebra}
\label{sec: wilson loop and fermions}

Our discussion so far has explained how to produce an $\cE_3$-algebra from Chern--Simons quantum field theory, 
and then how the ribbon category of perfect left modules for this $\cE_3$-algebra recovers the category of representations of the quantum group.

The work of Reshetikhin--Turaev goes further: 
it takes a representation of the quantum group---equivalently, a perfect module for the dual $\cE_3$-algebra---and produces a knot invariant. 
In parallel, Witten's construction takes a Wilson line in Chern--Simons theory and produces a knot invariant.  
We would therefore expect a direct connection between perfect left modules for the $\cE_3$-algebra and Wilson lines.

Such a connection is provided by considering ``constructible'' factorization algebras associated to defects for the quantum field theory. 
Let us sketch how to do this, 
focusing on the perspective from physics.
In the next subsection we give a precise mathematical articulation.

Suppose we are studying a classical field theory on a manifold $M$, and we choose a submanifold $N \subset M$.  
One can introduce a {\it defect} on $N$ by introducing extra fields that live only on $N$, 
but whose Lagrangian also depends on the fields of the original theory on $M$. 
In other words, we have a kind of field theory on $N$ that is coupled to the ambient theory on $M$.
The observables of this theory with a defect will form a constructible factorization algebra: 
away from $N$ the observables will agree with the factorization algebra of the original, defect-free theory, 
while on open sets of $M$ that intersect $N$, 
the observables will be enhanced by extra observables built from the fields living on~$N$.   

In order to explain how Wilson lines are related to $\cE_3$ modules, 
we will have to introduce defects for Chern--Simons theory that lead to Wilson line observables 
and then will study the constructible factorization algebras coming from these defects.

It is a well-established idea that a Wilson line observable can be realized as a defect consisting of a quantum mechanical system, living on the line, that couples to the ambient gauge theory. 
But an arbitrary quantum mechanical system will not work.
For a quantum mechanical system to lead to a Wilson loop, 
the partition function of this system on a circle, as a function of the gauge field $A$, 
must be the trace of the holonomy of $A$ in a chosen representation.  

Witten \cite{WitCS} suggested a way to do this for an arbitrary representation of a compact Lie group: 
consider a particle moving on a coadjoint orbit of the group $G$ and couple the gauge field via the action of $G$ on the orbit. 
This theory can be treated using geometric quantization, 
and hence the Borel--Weil--Bott theorem allows one to realize any desired representation.
(A lucid explanation of this idea can be found in~\cite{BeaCS}.)\footnote{One could also construct this quantum mechanical system using the methods of \cite{GG,GLL},
which would produce a 1-dimensional BV theory and hence a factorization algebra on the knot.
It would be interesting for someone to do this.}

In Section~\ref{sec: ff} we will study a simpler, fermionic system, which will yield a special class of Wilson lines.  
We will choose a representation $V$ of $G$, and introduce on a line $K: \RR \hookrightarrow \RR^3$, a pair of fermions
\[
	\psi \in \Omega^0(\RR, V) \quad\text{and}\quad
	\psib \in \Omega^0(\RR,V^\ast)
\]
with action functional
\[
S_{coup} = \int_K \langle \psib, (\d+K^*A) \psi \rangle_V
\]
where $A$ is the gauge field of Chern--Simons theory.\footnote{Hence we now work with $\ZZ/2$-graded cochain complexes, or super cochain complexes, throughout the paper. This extra grading is quite inert in our constructions, so there are no super subtleties to worry about.}

The algebra of operators of this free fermion system is $\Cl(V \oplus V^*)$, 
the Clifford algebra determined by $V \oplus V^*$ using the evaluation pairing.
Note that we remember the natural $\ZZ/2$-grading.

The partition function of this fermionic system on $S^1$, in the presence of a background gauge field $A$ on $S^1$, is the trace of the holonomy of $A$ in the spinor representation $\SS_V = \Lambda^\ast V$.  
Therefore this defect realizes the Wilson loop in this representation.

We now discuss an important refinement by allowing the 1-manifold to have boundary. 
Our defects are now not just lines or circles, but half-lines and line segments.

Before coupling to the ambient gauge theory,
note that this fermionic system has a natural boundary condition, where on the boundary we set $\psi = 0$.  
The observables of the fermionic system in the presence of this boundary condition give a constructible factorization algebra on the half-line $\RR_{\ge 0}$.  
On an interval $(a,b)$ the value of the factorization algebra is the Clifford algebra $\Cl(V \oplus V^*)$. 
But on an interval $[0,\epsilon)$ containing the boundary, this factorization algebra has value $\Lambda^\ast V$,
since our boundary conditions implies that we are only allowed observables on the $\psib$ field.  
The structure maps of a locally constant factorization algebra on the half-line encode an associative algebra (in this case $\Cl(V \oplus V^*)$) together with a module (in this case~$\Lambda^\ast V$).   

Now suppose we couple this fermionic system on the half-line to the ambient gauge theory. 
If we quantize only the fermion and leave the ambient gauge theory as a classical theory---what we 
will call the half-quantized coupled system---then we obtain a constructible factorization algebra that is given by $\C^*(\g)$ on any ball disjoint from the defect, is given by 
\begin{align*}
\C^*(\g,\Cl(V \oplus V^\vee)) &\cong \C^*(\g,\End_\CC(\SS_V)) 
\end{align*}
on any open ball intersecting the half-line but not its boundary, and is given by
\begin{align*}
\C^*(\g,\SS_V) 
\end{align*}
on any ball containing the endpoint of the half-line.
(These claims are the content of Proposition~\ref{prp: main ff}.)

When we quantize the gauge theory as well, 
the observables again form a constructible factorization algebra.
In this case, in the bulk away from the half-line, the observables agree with $\cA^\lambda$.
Around the boundary of the half-line, we have a deformation of $\C^*(\g,\SS_V) $ to a $\cA^\lambda$-module,
and in a ball that intersects the half-line but not its endpoint, 
we have the algebra of endomorphisms over $\cA^\lambda$ of this module.
(These claims are the content of Proposition~\ref{prp: defect main}.)

The partition function of this fully-quantized coupled system thus determines a quantization of the Wilson loop observable.
In this way we have identified the expected value of a Wilson loop operator with the knot invariant arising from a quantum group.
Indeed, this identification by fermions exhibits the canonical bijection between quantizations of Chern--Simons theories and ribbon deformations of~$\Rep_\fin(\g)$.

\begin{rmk}
There is another reason we take this approach to constructing the Wilson line observables.
A different approach would be simply to write down a quantum observable whose associated classical observable is the trace of holonomy.
That is, trace of holonomy is a cocycle in the cochain complex of classical BV observables,
and one picks a cocycle in the cochain complex of quantum observables whose reduction modulo $\hbar$ is trace of holonomy.
In fact, one can interpret \cite{AltFreVCS} as providing such a cocycle.
However, this approach has a drawback:
any classical cocycle $J^\cl$ admits many such lifts, 
since one lift $J$ can be modified to $J + \hbar J'$, 
where $J'$ is some other quantum cocycle.
Our approach using the coupled theory evades this drawback,
since we show in Section \ref{subsec: def complex for defects} 
that any quantization is unique if $\g$ is semisimple.
In other words, a quantization of the coupled system determines a canonical lift.
\end{rmk}

\begin{rmk}
We remark here on a nice feature of this system.
The Feynman diagrams for the coupled system can be identified with the theory of configuration space integrals developed by Bar-Natan, Bott--Taubes, Thurston, Altsch\"uler--Freidel, and many others \cite{BarVI,BotTau,Thu,AltFreVCS}.
Our results thus provide a setting wherein such computations organize coherently to relate to quantum groups.
\end{rmk}

\subsection{Modules for an $\cE_3$-algebra and knot invariants}

We just described how constructible factorization algebras connect with the physical notion of a defect, 
in the sense of a field theory living on a submanifold and coupled to a field theory living on the ambient manifold.
This relationship can be pushed further:
for a topological field theory on $\RR^n$ with factorization algebra $\cA$ (i.e., an $\cE_n$-algebra),
we identify a category of constructible factorization algebras, supported on 1-manifolds with boundary, that captures the category of topological line defects as articulated by physicists (see~\cite{KapICM, FucSch3d, FucSchTFT, KapWT} as a starting place).
Moreover, our formalism provides a systematic explanation for how perturbative topological field theories yield functorial knot invariants,
by directly connecting these factorization algebras (encoding algebras of operators) to functorial field theories (notably Reshetikhin--Turaev tangle invariants).

Let us outline our axiomatic notion of line defects and how they produce tangle invariants.

There is a powerful framework that axiomatizes the properties that topological line defects should possess,
known as the Tangle Hypothesis \cite{BaeDol, LurTFT}.
Roughly speaking it goes as follows. 
Suppose one wanted a field theory in the style of Atiyah--Segal for a particle moving in $n$-dimensional space. 
Thus, there is a bordism category, consisting here of 0-dimensional manifolds embedded in $\RR^{n-1}$ (playing the role of a spacelike hypersurface) as the objects and of 1-dimensional manifolds embedded in $\RR^{n-1} \times [0,1]$ (playing the role of a slab of spacetime) as the morphisms.
We call this the category of {\em tangles}.
There is a natural monoidal structure given by disjoint union,
but it is not symmetric monoidal---in contrast to the usual Atiyah--Segal setting---but $\cE_{n-1}$-monoidal,
precisely because we are interested in configurations of points in $\RR^{n-1}$.
A tangle field theory is then an $\cE_{n-1}$-monoidal functor into some other $\cE_{n-1}$-monoidal category~$(\cC,\otimes)$.

The tangle category captures vividly the visual idea of line defects,
while the functor encodes a kind of quantum mechanical system living along these 1-manifolds.

To the reader familiar with the Atiyah--Segal axioms, 
it likely comes as no surprise that a tangle field theory is fully determined by what it assigns a point (or a point with the relevant geometric structure, such as an orientation).
If the target is $(\cC,\otimes)$, then a tangle field theory assigns to the point a dualizable object in $\cC$.
Conversely, every dualizable object in $\cC$ determines a tangle field theory.
This characterization of $\cC$-valued tangle field theories is known as the Tangle Hypothesis,
and it is proven for such 1-dimensional defects in~\cite{AF.tangle}.

In the case $n = 3$ and $(\cC,\otimes)$ is a braided monoidal category,
this construction is manifestly that developed by Reshetikhin and Turaev \cite{RT},
and so we denote the tangle TFT of a dualizable object $V$ by~$\RT_V$.

Let us now connect this idea to our situation.
Recall that a perturbative topological field theory on $\RR^3$ determines an $\cE_3$-algebra $\cA$,
and that left $\cA$-modules $\LMod_\cA$ form a $\cE_{2}$-monoidal category.
By the Tangle Hypothesis, we thus recognize that the tangle field theories with values in $\LMod_\cA$ are given by the dualizable modules, also known as the {\em perfect} modules. 
In short, $\Perf_{\cA}$ captures precisely the functorial topological line defects for the perturbative 3-dimensional field theory,
and for each $V \in \Perf_{\cA}$, the functor $\RT_V$ determines a class of functorial tangle invariants.

A motivating idea for this paper is that every such tangle TFT $\RT_V$ admits an explicit construction by factorization homology.
The basic idea is simple. 
For such $V$, its endomorphisms $\End_\cA(V)$ is an $\cE_1$-algebra in $\cA$-modules, and hence it determines a factorization algebra $\cA_V$ on any framed $3$-manifold decorated with embedded framed $1$-manifolds. 
In the complement of a tangle $T \subset [0,1] \times \RR^{n-1}$, 
$\cA_V$ assigns $\cA$ to each open ball;
on an open ball intersecting the tangle, $\cA_V$ assigns~$\End_\cA(V)$.
This assignment (and the natural structure maps) determines a factorization algebra.
To a connected 1-manifold with boundary, it assigns $V$ to any ball containing just the positively-oriented boundary,
and it assigns the dual $V^\vee$ to any ball containing just the negatively-oriented boundary.
A careful development of this idea, 
using the beta factorization homology appropriate for coefficients in categories, is in~\cite{AF.tangle.fact}.

Our second main result is then the following,
which bridges functorial field theory and factorization homology.

\begin{thm}
\label{second main theorem}
Let $\cA$ be an $\cE_3$-algebra and let $V$ be a perfect left $\cA$-module.
The Reshetikhin--Turaev link invariants of the tangle TFT associated to $V$ can be computed using factorization homology.
\end{thm}

This result allows us to focus on the $\cE_3$-algebras that we produce from Chern--Simons theory.
This theorem says that for a finite-dimensional representation $V$ of $\g$,
each filtered $\cE_3$-algebra deformation $\cA^\lambda$ of $\C^*(\g)$ determines
a Reshetikhin--Turaev tangle TFT $\cA^\lambda_V$.
Similarly, for a finite-dimensional representation $V$ of $\g$, 
each ribbon deformation $\lambda'$ of $\Rep_\fin(\g)$ determines such a tangle TFT.
Hence to identify such deformations, we can identify the tangle TFTs,
i.e., look for choices of $\lambda$ and $\lambda'$ that give the same tangle invariants.
More specifically, we compare link invariants.
In this way we can match the usual quantum link invariants with the perturbative invariants.

In particular, this result implies that we should try to produce such a constructible factorization algebra for each quantization of Chern--Simons theory in some natural way.
Thankfully, as discussed already in Section~\ref{sec: wilson loop and fermions}, this problem is amenable to an approach from physics---the construction of line defects by field theories---that produces the desired constructible factorization algebra by the machinery of~\cite{CG1,CG2}.

We emphasize that this second main theorem of ours gives a mechanism for explicitly matching parameters between the BV quantizations and the quantum group deformations.
Thanks to the first main theorem of the paper, we know that it will be possible to produce an explicit bijection between the deformation problems,
but this second theorem shows how:
one picks a tangle and a representation $V$ and then computes the tangle invariants order by order in $\hbar$.
The first and second theorems are independent of one another, 
but together they show how Chern--Simons theory leads to functorial tangle invariants,
with factorization homology as the pathway.

\begin{rmk}
A few physical comments may clarify our assertions.
First, note that we allow our line defects to terminate, i.e., to be snipped at a point.
By allowing an endpoint, we must specify a 0-dimensional defect between the physics living along the line and the physics of the ambient space.
For this reason, we must specify a bimodule for the ambient $\cE_3$-algebra and for the $\cE_1$-algebra living along the line.
Second, the dualizability of this bimodule arises because we want to be able to twist and bend this line defect freely.
This freedom to manipulate the defect, without changing its dynamics, is encoded in the standard, categorical arguments that assure the existence of a dual.
Finally, we note that in our setting, there is not a clear distinction between ``order'' and ``disorder'' defects.
In this example of Chern--Simons theory, it is clear how to interpret our constructions as Wilson-type operators.
But in the explicit fermionic construction, we will see---even at the classical level---that the line defect acts as a source for charge and changes the equations of motion by introducing a singularity along the line.
\end{rmk}

\subsection{Back to knot polynomials}

We now return to the opening of the paper:
in terms of our discussion so far,
how does a rigorous formulation of perturbative Chern--Simons theory encode, for instance, the Jones polynomial?

Let $\g$ be a semisimple Lie algebra.
The case $\g = \mathfrak{sl}_2$ leads to the Jones polynomial proper,
while the HOMFLYPT (or Jones--Conway) polynomial appears by organizing the series $\mathfrak{sl}_n$ of Lie algebras intelligently.

Our first step is by Theorem~\ref{first main theorem}.
Each quantized Chern--Simons theory determines a quantum group.
Hence, among the quantizations appears the Drinfeld--Jimbo quantum group $\U_q \g$
when $q$ is set to $\exp(\hbar)$ in a suitable manner.
We will denote this special case by $\U_\hbar \g$.
Up to a nonlinear reparametrization of the formal disk in which $\hbar$ lives,
every quantum group is equivalent to~$\U_\hbar \g$.

Each quantization produces a ribbon category that is equivalent to $\Rep(\U_\hbar \g)$,
the finite-rank and topological-free representations of the Drinfeld--Jimbo quantum group.
Hence the line operators, particularly the Wilson loops, of quantized Chern--Simons theory
are labeled by such representations.

Our second step is by Theorem~\ref{second main theorem}.
Each representation $V \in \Rep(\U_\hbar \g)$ determines framed link invariants .
The theorem says these invariants can be computed in two distinct ways:
\begin{itemize}
\item by using a functor $\RT_V$ from the category of framed tangles in $\RR^3$ to $\Rep(\U_\hbar \g)$, in the Reshetikhin--Turaev-style \cite{RT,KRT}, or
\item by using factorization homology, as developed in Section~\ref{sec: tangles}.
\end{itemize}
As we show in Sections~\ref{sec: ff} and \ref{sec: coupled}, these factorization homology computations agree with physical computations of the Wilson operator when $V$ can be realized explicitly as a defect theory.

In this sense we have shown that the quantum field theoretic approach is mirrored by the quantum group approach.
Theorem~\ref{second main theorem} assures us that there is agreement between two mathematical formalizations of Witten's manipulations of the path integral (by cutting up 3-space with an embedded link into components).
Theorem~\ref{first main theorem} tells us how to match the particular examples of these machines ($\cE_3$-deformation of $\C^*(\g)$ versus a quantum group deformation of $\U\g$) 
arising from the case of Chern--Simons theory.

\begin{rmk}
The $\cE_3$-algebra needed for the factorization computation can be obtained by taking the (derived internal) endomorphisms of the unit $\mathbf{1}$ for the braided monoidal category $\Rep(\U_\hbar \g)$.
The underlying $\cE_1$-algebra is simply the Lie algebra cochains $\C^*(\g[\![\hbar]\!])$ since the underlying category of $\Rep(\U_\hbar \g)$ is just finite-rank and topological-free modules over $\U(\g[\![\hbar]\!])$,
but there is a nontrivial $\cE_3$-algebra structure arising from the braided monoidal product.
(There is an isomorphism of algebras between $\U(\g[\![\hbar]\!])$ and $\U_\hbar \g$,
but they are not isomorphic as quasi-triangular quasi-Hopf algebras.)
With this purely algebraic presentation of the $\cE_3$-algebra,
one could use factorization homology to compute the usual functorial framed tangle invariants
\end{rmk}

Now that we have recovered the standard functorial invariants of framed tangles,
we can invoke classic results that extract the knot polynomials from the functorial situation. 
These are not totally straightforward: 
there are some subtleties in connecting the $\U_q\mathfrak{sl}_2$ invariants to the Jones polynomial.
Let $K(L) \in \ZZ[t^{1/2},t^{-1/2}]$ denote the Kauffman bracket of a framed link $L$ in $\RR^3$.
The Jones polynomial $J([L]) \in \ZZ[t^{1/2},t^{-1/2}]$ is an invariant of the underlying oriented link $[L]$ given by
\[
J([L]) = (-t^{3/2})^{w(L)}, 
\]
where $w(L)$ is the writhe of $L$.
(This is Kauffman's approach to the Jones polynomial, not the original construction.)
There is also an invariant of this framed link $L$, arising from the fundamental representation of $\U_q\mathfrak{sl}_2$, 
which we will denote $F(L)$.
It is an element of $\ZZ[q^{1/2},q^{-1/2}]$.
If we identify $t$ and $q$,
then $F(L)$ agrees with the Kauffman bracket $K(L)$ up to a sign determined by the writhe of the link and its number of components.
(A clear and thorough discussion can found in~\cite{Tin},
who cites~\cite{Oht}.)

It might seem surprising that the perturbative quantization of Chern--Simons theory,
which only takes the path integral in a formal neighborhood of the trivial flat connection,
can produce rich invariants such as the Jones polynomial.
On the other hand, we have just learned that the Jones polynomial of an oriented link $V_L$ is an element of the ring $\ZZ[q^{1/2},q^{-1/2}]$, i.e., a Laurent polynomial in the variable~$\sqrt{q}$.
Analytic continuation thus determines this function from its Taylor expansion around the point $q=1$.
Hence, substituting $q = e^\hbar$, one obtains a formal power series in $\hbar$.
Each coefficient can be identified as a Vassiliev invariant \cite{KRT, KasTurChord}, 
and the theory of the Kontsevich integral assures us that it can be computed perturbatively.

\begin{rmk}
Expositions of Chern--Simons theory via quantum groups often emphasize the necessity of working at a root of unity,
which is quite far from what we do here.
Here we work with quantum groups with a formal parameter $\hbar$,
and we claim it encodes a part of Chern--Simons theory: the perturbative part.
These assertions are not in tension but are, in fact, compatible. To get a 3-dimensional topological field theory, in the sense of Atiyah and Segal,
but extended down to 1-manifolds,
one needs to start with a {\em modular} tensor category \cite{BakKir, BDSPV}.
This kind of category is a very special type of ribbon category,
exhibiting much stronger finiteness properties.
These arise from quantum groups by a rather sophisticated method,
as a subquotient of the category of finite-dimensional representations of $\U_q \g$ at a root of unity.
(One must also be thoughtful about how one constructs the quantum group at a root of unity!)

On the other hand, the knot invariants arise from the Drinfeld--Jimbo group $\U_q\g$ where $q$ is rational parameter.
This $q$ is a variable in the knot polynomial.
Specializing to roots of unity evaluates the polynomials and hence produces a number,
which is less informative as a knot invariant.

The connection between these situations is that every closed, oriented 3-manifold can be presented by surgery of $S^3$ along a framed link.
The same manifold admits many different presentations.
There is thus an equivalence relations on framed links in $S^3$ whose equivalence classes label oriented homeomorphism types of oriented closed 3-manifolds \cite{Kirby, FenRou}.
Hence, if we want to use the functorial framed link invariants to produce invariants of oriented closed 3-manifolds,
we need the invariants to be constant on each equivalence class.
Roughly speaking, the modular tensor categories arising from quantum groups arise by implementing these relations.

The perturbative theory allows one to construct the knot invariants by studying line defects perturbatively.
These defects yield the $\hbar$-expansion of the knot polynomials,
which thus determine the polynomials in terms of $q$.
To obtain a path integral invariant of a 3-manifold, however, one must integrate over the full space of connections,
not just near the trivial connection,
and this requires the level to be integral, 
which translates into $q$ becoming a root of unity.
\end{rmk}

\subsection{Section by section outline}

We begin by developing the language necessary to realize the argument outlined in the introduction 
and then turn to constructions in field theory itself.
Hence Section \ref{sec: tangles} provides the necessary formalism of constructible factorization algebras and exhibits the relationship with Reshetikhin--Turaev tangle TFTs.
It ends by reviewing the example that we will deform, i.e., tangle TFTs from~$\Rep_\fin(\g)$.
Section \ref{sec: CS} reviews the perturbative quantization of Chern--Simons theory (without defects) and examines its factorization algebra.
Section \ref{sec: ff} constructs the 1-dimensional field theory, the charged fermion, whose factorization algebra will lead to the tangle TFT that we explicitly exhibit.
Section \ref{sec: coupled} finally builds the desired constructible factorization algebra by quantizing Chern--Simons theory coupled to a fermion.
In the appendix we review results about the Feynman diagrammatics of Chern--Simons theory.
Each section begins with an overview, so more detailed information about what is accomplished in a given section can be found there.

\subsection{Acknowledgements}

This project has gestated for over a decade and hence has received support and encouragement from many people and institutions.

J.F. thanks David Ayala and Hiro Lee Tanaka for their collaboration in~\cite{AFT2} and~\cite{AF.tangle}, on which this work relies.
O.G. benefitted from feedback on this work from the audiences of a lectures series at the Hausdorff Institute for Mathematics in fall 2017 and at Northeastern University in fall 2018.
He thanks Sachin Gautam, David Jordan, Pavel Safronov, Valerio Toledano-Laredo, Peter Teichner, and Brian Williams for illuminating conversations and probing questions as these ideas developed.
The Max Planck Institute for Mathematics provided a wonderful working environment and financial support for O.G. during his primary work on this project in 2017--18.
J.F. and O.G. had a chance to re-engage with this project during the MSRI program ``Higher Categories and Categorification'' in spring 2020, and at UNAM in Cuernavaca, Mexico in summer 2022.

\section{Links and factorization homology}
\label{sec: tangles}

We recall the notion of a framed $(d\subset n)$-manifold from Example~5.2.12 of~\cite{AFT1}, and studied in \S4.3 of~\cite{AFT2}. 
A {\it framing} of a properly-embedded submanifold $K^d\subset M^n$ consists of a framing of $M$ 
(i.e., a bundle isomorphism $T_M\cong \epsilon^n_M$ between the tangent bundle and trivial rank $n$ bundle)
together with one of the following equivalent structures:
\begin{enumerate}
\item a framing of $K$ and a framed embedding $K\times\RR^{n-d}\hookrightarrow M$ extending the given embedding of~$K$;

\item a framing of $K$ and a framing of the normal bundle ${\rm Norm}_{K\subset M}$ compatibly with the framing of~$M$;

\item a nullhomotopy of the Gauss map $K\ra {\rm Gr}_d(\RR^n)$ defined by the embedding $K\subset M$ and the framing of~$M$.

\end{enumerate}
By a {\it framed link} $K\subset M$ we will mean denote a framed $(1\subset n)$-manifold in the above sense, where~$d=1$.

\begin{rmk}
Note that being properly-embedded includes submanifolds like $\RR\times\{0\}\subset \RR \times\RR^{n-1}$ but excludes submanifolds such as $(0,1)\times\{0\}\subset\RR\times\RR^{n-1}$.
\end{rmk}

We form $\oo$-categories of framed $n$-manifolds (see~\cite{AF}) and, later, framed links (see~\cite{AFT1} and\cite{AFT2}).

\begin{dfn}
The $\infty$-category $\mfld_n^{\fr}$ is the limit
\[
\xymatrix{
\mfld_n^{\fr}\ar[r]\ar[d]&\ar[d]\Spaces_{/{\rm V}_n(\RR^\oo)}\\
\Mfd_n\ar[r]&\Spaces_{/{\rm Gr}_n(\RR^\oo)}~.}
\]
\end{dfn}
In particular, an object of $\Mfld_n^\fr$ is an $n$-manifold together with a trivialization of its tangent bundle, which is equivalent to a lift of the tangent classifier $M\to {\rm Gr}_n(\RR^\oo)$ to the infinite Stiefel manifold ${\rm V}_n(\RR^\oo)$. A framed embedding $M\hookrightarrow N$ consists of an embedding together with a homotopy between the two trivializations of $TM$, from the framing of $M$ and the pulledback framing from $N$.

For the remainder of this section, we fix the following:
\begin{itemize}

\item $A$ is an $\cE_n$-algebra in an symmetric monoidal $\oo$-category $\cV$, with $n>1$.

\item $V$ is a perfect $A$-module, i.e., an $A$-module for which which there exists an $A$-linear perfect pairing
\[
V \otimes_A V^\vee \longrightarrow A
\]
where $V^\vee = \Hom_A(V,A)$ is the $A$-linear dual of~$V$.

\item $M$ is a framed $n$-manifold, and $K\subset M$ is a framed link.

\end{itemize}
We introduce the following fundamental invariant using factorization homology
    \[
    \int A: \Mfld_n^{\fr}\longrightarrow \cV
    \]
with coefficients in~$A$.

\begin{dfn}
Let $A$, $V$, and $K\subset M$ be as above. 
The {\em trace of $V$ over $K$}
\[
\displaystyle\int_{K\subset M}\tr(V) \in \displaystyle\int_M A
\]
is the image of the element
\[
\tr(V)\in {\rm HH}_\ast(A)^{\otimes \pi_0 K} \simeq \displaystyle\int_{K\times\RR^{n-1}}A
\]
under the map
\[
\displaystyle\int_{K\times\RR^{n-1}}A \longrightarrow \displaystyle\int_{M}A    
\]
given by applying factorization homology $\int_{-} A$ to the framed embedding $K\times\RR^{n-1}\hookrightarrow M$.
\end{dfn}

The main theorem of this section identifies a special case of this trace invariant with the Reshetihkin--Turaev invariant. 
Our proof uses aspects of a beta form of factorization homology, 
in development in work with David Ayala~\cite{AF.tangle.fact}. 
This theory generalizes factorization homology of $\cE_n$-algebras to rigid $\cE_{n-1}$-monoidal $\oo$-categories, 
in the same way that the Hochschild homology of algebras can be generalized to categories. 
In particular, for a $\cV$-linear $\cE_{n-1}$-monoidal $\oo$-category~$\cC$, 
there is a functor
\[
\int\cC: \Mfld_n^{\fr}\longrightarrow \cV
\]
assigning to a framed $n$-manifold $M$ the beta factorization homology $\int_M \cC$, an object of~$\cV$.
Facts about beta factorization homology used in our arguments are collected in Theorem~\ref{betaclaims}. 

We now turn to the main theorem of this section.

\begin{thm}
    Given a choice of invariant pairing on a semisimple Lie algebra $\g$, there exists a filtered $\cE_3$-algebra $\C_\hbar^\ast(\g)$ deformation of Lie algebra cochains $\C^\ast(\g)$. A finite-dimensional representation $V$ of the Drinfeld--Jimbo quantum group $\U_\hbar\g$ defines a perfect module for $\C_\hbar^\ast(\g)$, and the resulting factorization homology link invariant is equal to the Reshetikhin--Turaev invariant
    \[
    \int_{K\subset \RR^3} \tr(V)= Z_V(K\subset\RR^3)
    \]
    under the identification $\H^0_\hbar(\g)\cong\CC[\![\hbar]\!]$.
\end{thm}
\begin{proof}
Let $A$ be an $\cE_3$-algebra in~$\cV$. 
The 1-dimensional Tangle Hypothesis~\cite{AF.tangle} provides an essentially unique functor
\[
Z_V: {\rm Bord}_1^{\fr}(\RR^2) \longrightarrow \Perf_A
\]
given the choice of a perfect $A$-module in~$\cV$. 
By the essential uniqueness of this functor, it agrees with the Reshetikhin--Turaev functor at the level of homotopy categories.

By the description of the functor $Z_V$ in terms of beta factorization homology of rigid $\cE_2$-monoidal $\oo$-categories in~\cite{AF.tangle.fact}, the value on a framed link $Z_V(K\subset\RR^3)$ is given by the composite
\[
\xymatrix{
{\rm obj}(\Perf_A)\ar[r]
&\displaystyle\int_{K\times\RR^{n-1}}{\Perf_A}\ar[r]&\displaystyle\int_{\RR^n}{\Perf_A}\ar[r]^-\sim& A\\
}
\]
evaluated on the element $V\in \obj(\Perf_A)$.
We have the following commutative diagram:
\[
\xymatrix{
{\rm obj}(\Perf_A)\ar[dr]\ar[r]
&\displaystyle\int_{K\times\RR^{n-1}}{\Perf_A}\ar[r]&\displaystyle\int_{\RR^n}{\Perf_A}\ar@{=}[r]& A\\
&\displaystyle\int_{K\times\RR^{n-1}}A\ar[r]\ar[u]&\displaystyle\int_{\RR^n}A\ar[u]_\simeq\ar@{=}[r]& A\ar[u]}
\]
The commutativity in the middle square is the compatibility between enriched alpha and beta factorization homology: beta factorization homology for a category with a single object (namely $A$) is equivalent to alpha factorization homology. 
The commutativity of the left triangle is by definition of the trace.
\end{proof}

In the proof we used the following assertion about the relation of factorization homology of $\cE_n$-algebras and the beta factorization homology of rigid $\cE_{n-1}$-monoidal $\oo$-categories from~\cite{AF.tangle.fact}.

\begin{thm}\label{betaclaims}[Beta claims \cite{AF.tangle.fact}]
For $A$ an $\cE_n$-algebra in $\cV$, there is a natural transformation of functors
\[
\displaystyle\int A \longrightarrow \int \Perf_A
\]
on $\Mfld_n^{\fr}$, the $\oo$-category of framed $n$-manifolds and framed embeddings. For $M = S^1\times\RR^{n-1}$, there is a natural equivalence
\[
\int_{S^1\times\RR^{n-1}} A \simeq \int_{S^1\times\RR^{n-1}}\Perf_A \simeq \hh_\ast(\Perf_A)
\]
between the factorization homology of $S^1\times\RR^{n-1}$ with coefficients in $A$, the beta factorization homology of $S^1\times\RR^{n-1}$ with coefficients in $\Perf_A$, and the $\cV$-enriched Hochschild homology of the $\oo$-category $\Perf_A$. For each $V\in \Perf_A$, the following diagram commutes
\[
\xymatrix{
\obj(\Perf_A)\ar[r]&\displaystyle\int_{S^1\times\RR^{n-1}}\Perf_A \\
\ast\ar[r]^-{\rm unit}\ar[u]^-{\{V\}}& \displaystyle\int_{S^1}\End_A(V)\ar[u]}
\]
where $\obj(\Perf_A)\ra\int_{S^1\times\RR^{n-1}}\Perf_A$ is the canonical unit map in beta factorization homology.
\end{thm}

Our analysis of the link invariant $\int_{K\subset M}\tr(V)$ will make use of a stratified form factorization homology for the pair $K\subset M$. The coefficient systems for this form of factorization homology are given by $\cE_{1\subset n}$-algebras, which we now define. 

For the definition of the $\infty$-category $\mfld^{\fr}_{1\subset n}$, see Example 5.2.12~\cite{AFT1}. 
Heuristically, this $\oo$-category has as objects framed links $K\subset M$, 
where $M$ is a framed $n$-manifold, 
and has as the space of morphisms $\Emb^{\fr}(K\subset M, L\subset N)$, 
the limit
\[\xymatrix{
\Emb^{\fr}(K\subset M, L\subset N)\ar[d]\ar[r]&\Emb^{\fr}(M,N)\ar[d]\\
\Emb^{\fr}(K,L)\ar[r]&\Emb^{\fr}(K\times\RR^{n-1},N)\\}\] where $\Emb^{\fr}$ consists of framed open embeddings with the compact-open topology.

\begin{dfn} 
Let $\disk^{\fr}_{1\subset n}$ denote the full $\oo$-subcategory of $\mfld^{\fr}_{1\subset n}$ consisting of finite disjoint unions of copies of $\emptyset\subset \RR^n$ and $\RR\subset \RR^n$. 
An {\em $\cE_{1\subset n}$-algebra in $\cV$} is a symmetric monoidal functor
\[
\disk^{\fr}_{1\subset n}\longrightarrow \cV
\]
denoted as a pair $(A,R)$, where $A$ is the underlying $\cE_n$-algebra and $R$ is the underlying $\cE_1$-algebra.
\end{dfn}

The factorization homology
\[
\int_{K\subset M}(A,R) = \underset{\disk_{1\subset n}/K\subset M}{\rm colim} (A,R)
\]
of framed links $K\subset M$ with coefficients in $\cE_{1\subset n}$-algebras is developed in detail in \cite{AFT2}. In particular, it satisfied a version of $\otimes$-excision, for excision sites which are transverse to the link.

\begin{rmk} 
$\disk^{\fr}_{1\subset n}$ is the smallest subcategory of $\mfld^{\fr}_{1\subset n}$ with the property that any finite subset of a manifold $M$ in $\mfld^{\fr}_{1\subset n}$ has an open neighborhood that is homeomorphic to an object of $\disk^{\fr}_{1\subset n}$.
\end{rmk}

We will make essential use of the following form of Morita theory for $\cE_{1\subset n}$-algebras.

\begin{dfn}
Let $(A,R)$ and $(A,S)$ be $\cE_{1\subset n}$-algebras. 
An {\em $(A; R,S)$-bimodule $V$} is an $(R,S)$-bimodule in $\int_{S^{n-2}\times\RR^2}A$-modules.    
\end{dfn}

Such bimodules are related to functors that have the following structure.

\begin{dfn}
Let $(A,R)$ and $(A,S)$ be $\cE_{1\subset n}$-algebras. 
A {\em $(n,1)$-Morita morphism under $A$}
\[
(A,R) \to (A,S)
\]
is a functor
\[
\Perf_R \longrightarrow \Perf_S
\]
that is linear with respect to the action of $\int_{S^{n-2}\times \RR^2}A$, i.e., preserves the enrichment in $\int_{S^{n-2}\times \RR^2}A$-modules.
\end{dfn}

The following result relates the previous two definitions.

\begin{lmm}
\label{lmm: perfect is n1 morita}
If an $(A; R, S)$-bimodule $V$ is perfect as an $S$-module, 
then the relative tensor product $-\otimes_R V$ defines an $(n,1)$-Morita morphism $\Perf_R\to \Perf_S$.
\end{lmm}

\begin{proof}
An $(A; R, S)$-bimodule $V$ defines a functor $\Mod_R\ra \Mod_S$ that is manifestly $\int_{S^{n-2}\times\RR^2}A$-linear. 
If $V$ is perfect as an $S$-module, then for every perfect $R$-module $W$ the relative tensor product $W\otimes_A V$ will be perfect as an $S$-module.    
\end{proof}

This notion has the follow interplay with factorization homology.

\begin{lmm}\label{lemma.bimmodule}
Let $(A,R)$ and $(A,S)$ be $\cE_{1\subset n}$-algebras, 
and let $V$ be an $(A;R,S)$-bimodule that is perfect as an $S$-module. There is a natural map
\[
\int_{K\subset M}(A,R) \to \int_{K\subset M}(A,S)
\]
of factorization homology theories on $\Mfld_{1\subset n}^{\fr}$ after restriction to {\bf closed} 1-manifolds~$K$.
\end{lmm}

\begin{proof}
By linearity of Hochschild homology, for a $B$-linear functor $\cC\ra \cD$ the resulting map $\hh_\ast(\cC)\ra \hh_\ast(\cD)$ is a map of $B$-modules. Applying this in the case $B= \int_{S^{n-2}\times\RR^2}A$, we obtain that the induced map
\[
\int_{S^1}R \simeq {\rm HH}_\ast(R)\to {\rm HH}_\ast(S)\simeq \int_{S^1} S
\]
is a map of $\int_{S^{n-2}\times\RR^2}A$-modules.

The result follows  a composition that uses $\otimes$-excision for factorization homology of $(1\subset n)$-manifolds from~\cite{AFT2}:
\[
\xymatrix{
\displaystyle \int_{K\subset M} (A,R)
\simeq
\displaystyle \int_{K} R\underset{\displaystyle \int_{K\times\RR^{n-1}\smallsetminus  K\times\{0\}}A}\otimes \displaystyle \int_{M\smallsetminus K} A\ar[d]\\
\displaystyle {\rm HH}_\ast(\Perf_R)^{\otimes \pi_0 K}\underset{\displaystyle \int_{K\times\RR^{n-1}\smallsetminus  K\times\{0\}}A}\otimes \displaystyle \int_{M\smallsetminus K} A\ar[d]\\
\displaystyle {\rm HH}_\ast (\Perf_S)^{\otimes \pi_0K}\underset{\displaystyle \int_{K\times\RR^{n-1}\smallsetminus  K\times\{0\}}A}\otimes \displaystyle \int_{M\smallsetminus K} A\ar[d]\\
\displaystyle \int_{K} S\underset{\displaystyle \int_{K\times\RR^{n-1}\smallsetminus  K\times\{0\}}A}\otimes \displaystyle \int_{M\smallsetminus K} A\simeq \displaystyle\int_{K\subset M}(A,S)~.
}
\]
Here the diffeomorphism 
\[
K\times\RR^{n-1}\smallsetminus  K\times\{0\}
\cong
\coprod_{\pi_0 K} S^{n-2}\times\RR^2
\]
lets us invoke the first paragraph to modify the left hand side of each tensor product in the composition.
\end{proof}

\begin{thm}  
Let $M$ be a framed $n$-manifold with $K\subset M$ a framed link. 
Let $(A,R)$ be an $\cE_{1\subset n}$-algebra with $V$ an $(A;R, A)$-bimodule that is perfect as an $A$-module. 
There is a natural map ${\sf m}_V:\int_{K\subset M}(A,R)\to \int_MA$ such that the diagram 
\[
\xymatrix{
&\displaystyle\int_K R\ar[rr]\ar[dd]&&\displaystyle\int_{K\subset M}(A,R)\ar[dd]^-{{\sf m}_V}\\
\Bbbk\ar[rd]_{\tr(V)}\ar[ru]^-{\rm unit}\\
&\displaystyle\int_{K}A\ar[rr]&& \displaystyle\int_{M}A\\}
\]
commutes up to equivalence.
\end{thm}

\begin{proof} 
Applying Lemma~\ref{lemma.bimmodule} to the case of $S=A$ gives the commutative square above. It remains to establish the commutative triangle on the left. By definition the map $\int_K R \ra \int_K A$ is given, for each connected component of $K$, by the value of Hochschild homology on the map
\[
\Perf_R \xra{-\ot_R V} \Perf_A~.
\]
This functor sends the unit $R\in\Perf_R$ to $V$. Applying Hochschild homology, this sends the unit element of $\int_{S^1} R$ to the element $\tr(V)\in \int_{S^1}A$.
\end{proof}

\section{Recollections on filtered Koszul duality}
\label{sec: filtered KD}

This section unpacks the consequences of the preceding work in the context relevant to this paper.
In short, it combines the above with results from \cite{CosYang} on filtered Koszul duality,
demonstrating how to recover the tangle TFTs for finite-dimensional representations of $\g$.
The rest of the paper is devoted to quantizing the situation explained here.

\subsection{Recollections}

Recall that a {\em filtration} on a cochain complex $V$ is a sequence 
\[
V=F^0V\la F^1 V \la \ldots.
\]
The sequence is {\em complete} if $\varprojlim F^i V$ is zero. 
Equivalently, by taking cokernels $V_i := V/F^i V$, a filtration of $V$ is a functor $\NN^{\rm op} \to \Mod_\CC$ whose limit is equivalent to~$V$.

\begin{dfn}
Let $\cV$ denote the $\oo$-category of {\em complete filtered cochain complexes over $\CC$}, namely
\[
\cV := \Fun(\NN^{\rm op}, \Mod_\CC)~.
\]
\end{dfn}

This category $\cV$ is symmetric monoidal, with tensor product given by Day convolution: see~\cite{CosYang} and~\cite{GwiPav}. The tensor product distributes over colimits separately in each variable. 
A {\it filtered algebra} is an algebra in $\cV$.

We now turn to our starring example.

Consider the commutative algebra $A = \C^*\g$
equipped with the decreasing filtration
\[
F^k A = \Sym^{\geq k}(\g^\vee[-1]),
\]
with respect to which it is complete.
If we work in the context of complete filtered cochain complexes, 
as developed in Sections 7.2 and 8 of \cite{CosYang},
then the Koszul dual coalgebra is
\[
A^! =\CC \otimes_A^\LL \CC
\]
the derived tensor product in filtered complexes.
Note here the role of the filtration,
which also equips $A^!$ with a natural filtration.

This coalgebra is naturally quasi-isomorphic to a familiar coalgebra:
the completed dual $(\U\g)^\vee$.
On $\U\g$ we have  the natural increasing filtration 
\[
F_k \U\g = \bigoplus_{j \leq k} \Sym^j(\g),
\]
and it induces on the completed linear dual $(\U\g)^\vee$ a decreasing filtration 
\[
F^k (\U\g)^\vee = (\U\g/F_k \U\g)^\vee.
\]
If $\g$ is finite-dimensional,
then $(\U\g)^\vee$ is in fact a Hopf algebra,
namely the functions on the formal group associated to~$\g$.

\begin{lmm}[\cite{CosYang}, Lemma 7.2.2]
Let $\g$ be a Lie algebra. Then 
\[
(\C^*\g)^! \simeq (\U\g)^\vee
\]
as complete filtered Hopf algebras.
\end{lmm}

One can interpret this lemma as saying that the Koszul dual {\em algebra} to $\C^*\g$ is $\U\g$,
so long as one takes care with filtrations and completions.

\subsection{Koszul duality for $\g$}

\def\RComod{{\rm RComod}}

We now consider left modules over $A$ in the dg category of complete filtered cochain complexes
and, dually, right comodules over $A^!$ in the same dg category.
We denote these by $\LMod_A$ and $\RComod_{A^!}$, respectively.

We also care about certain subcategories.
Let $\Perf_A$ denote the full subcategory of $V\in \LMod_A$ consisting of modules that are dualizable, i.e., for which there exists a perfect pairing
\[
V\otimes_A V^\vee \longrightarrow A
\]
This property of dualizability can be checked at the level of the associated graded,
which recovers the definition in Section 8.1 of \cite{CosYang}.
Let $\Fin_{A^!}$ denote the full subcategory of $\RComod_{A^!}$ consisting of comodules whose cohomology has finite total dimension
(i.e., its cohomology is bounded and is finite-dimensional over $\CC$ in every degree).

A central result of Section 8 of \cite{CosYang} is the following.
(Here we have fixed $A$ but the proof there applies to a broader class of complete filtered algebras.)

\begin{prp}
The functor $\CC \otimes^\LL_A - \colon \LMod_A \to \RComod_{A^!}$ is a quasi-equivalence of dg categories.
Moreover, it restricts to a quasi-equivalence from $\Perf_A$ to $\Fin_{A^!}$.
\end{prp}

Just as we saw above that $(A^!)^\vee \simeq \U\g$---with a proper understanding of filtrations---we can extend this result to talking about modules over $\U\g$, suitably understood.
See Section 8.2 of \cite{CosYang} for a discussion.
Here we only care about dg $\U\g$-modules whose cohomology has total finite dimension,
and which we thus call {\em finite}.

\def\op{{\rm op}}

\begin{prp}[\cite{CosYang}, Lemma 8.2.5]
There is a quasi-equivalence of dg categories
\[
\Fin_{A^!} \simeq \Fin_{(A^!)^\vee}^\op.
\]
\end{prp}

Putting all these results together, we have the following.

\begin{crl}
\label{filt KD}
There is a quasi-equivalence 
\[
\Perf_{\C^*\g} \simeq \Fin_{\U\g}^\op
\]
of dg categories.
\end{crl}

In fact, since $\C^*\g$ is commutative, 
the left hand side is a symmetric monoidal dg category via $\otimes^\LL_{\C^*\g}$.
The right hand side is also symmetric monoidal, but via $\otimes^\LL_\CC$.
One can trace through the construction of the quasi-equivalence to see that it is a quasi-equivalence of {\em symmetric monoidal} dg categories.

Earlier we used $\Rep_\fin(\g)^\dg$ to denote the $\infty$-category of finite-type dg representations of $\g$.
The dg category $\Fin_{\U\g}$ is presents this $\infty$-category,
just as $\Perf_{\C^*\g}$ presents the $\infty$-category of perfect filtered modules over $\C^*\g$.
Hence we have obtained the equivalence (\ref{filtKD}) mentioned in the introduction.

\begin{rmk}
The equivalence (\ref{filt KD}) is certainly well-known to experts in this wing of homotopical algebra.
An essentially equivalent result is proved in \cite{CPTVV} (see Section 3.6.2),
where mixed-graded structures encode the filtrations discussed here.
\end{rmk}

\subsection{Extending to quantizations}

We now deform the preceding equivalence over the formal variable $\hbar$.
By Drinfeld's theorem, 
there exists an $\cE_2$-monoidal deformation of $\Fin_{\U\g} = \Rep^{\dg}_{\fin}(\g)$ for each choice of level~$\lambda$. 
This deformation determines a deformation of~$\C^\ast(\g)$. 
Immediately below, we use the term {\it moduli functor} to mean a space-valued functor on derived Artin commutative algebras -- in particular, we do not require that is infinitesimally cohesive or has a tangent complex.

\begin{dfn}
Given an $\cE_2$-monoidal $\cV$-linear $\oo$-category $\cC\in \Alg(\Mod_\cV)$, 
let ${\rm tDef}(\cC)$ denote the moduli functor whose $R$-points are the space of $R$-linear $\cE_2$-monoidal deformations of $\cC$ that are {\em trivialized} as $\cE_1$-monoidal deformations. 
That is, ${\rm tDef}(\cC)(R)$ is the fiber
\[
\xymatrix{
{\rm tDef}(\cC)(R)\ar[r]\ar[d]&\Alg_{\cE_2}(\Mod_R(\Mod_\cV))^\sim\ar[d]\\
\{R\otimes \cC\}\ar[r]&\Alg(\Mod_R(\Mod_\cV))^\sim}
\]
or, equivalently, ${\rm tDef}(\cC)(R)$ is the space of $\cE_2$-monoidal refinements of $R\otimes \cC$ as a monoidal $\cV$-linear $\oo$-category.
\end{dfn}

That is, we are deforming just the braiding: in the notation ${\rm tDef}$ the prefix t emphasizes this trivialization of the deformation of the underlying monoidal structure.

There is an analogous moduli space for~$\cE_2$-algebras.

\begin{dfn}
For an $\cE_2$-algebra $A\in \Alg_{\cE_2}(\cV)$ in $\cV$, 
let ${\rm tDef}(A)$ denote the moduli functor whose $R$-points are the space of $R$-linear $\cE_3$-algebra deformations of $A$ that are trivialized as $\cE_2$-algebra deformations. 
That is, ${\rm tDef}(A)(R)$ is the fiber
\[
\xymatrix{
{\rm tDef}(A)(R)\ar[r]\ar[d]&\Alg_{\cE_3}(\Mod_R(\cV))^\sim\ar[d]\\
\{R\otimes A\}\ar[r]&\Alg_{\cE_2}(\Mod_R(\cV))^\sim}
\]
or, equivalently, ${\rm tDef}(A)(R)$ is the space of $\cE_3$-algebra refinements of $R\otimes A$ as an $\cE_2$-algebra. 
\end{dfn}

\begin{lmm}
Let $A$ be an $\cE_2$-algebra in $\cV$. 
There is a natural equivalence of moduli functors
\[
{\rm tDef}(A) \simeq {\rm tDef}\bigl(\Perf_A(\cV)\bigr)
\]
between $\cE_2$-trivialized $\cE_3$ deformations of $A$ and monoidally-trivialized $\cE_2$-deformations of~$\Perf_A(\cV)$.
\end{lmm}

\begin{proof}
We establish a pair of maps
\[
\Perf: {\rm tDef}(A) \leftrightarrows {\rm tDef}\bigl(\Perf_A(\cV)\bigr): \End(1)~,
\]
which are quickly seen to be mutually inverse.
Given $B\in \Alg_{\cE_3}(\Mod_R(\cV))$ an $\cE_2$-trivialized $\cE_3$-deformations of $A$ over $R$, 
then $\Perf_B(\cV)$ is an $R$-linear $\cE_2$-monoidal $\oo$-category. 
Since an $\cE_2$-algebra equivalence $B\simeq R\otimes A$ is part of the data of a point in ${\rm tDef}(A)(R)$, 
therefore $\Perf_B(\cV)$ is equivalent to $\Perf_{R\otimes A}(\cV) \simeq R\otimes\Perf_A(\cV)$ as a monoidal $\oo$-category.

Given $\cD \in {\rm tDef}(\Perf_A(\cV))(R)$, 
the endomorphisms of the monoidal unit $\End_\cD(1)$, that is the object in $\cV$ given by the enrichment of $\cD$ over $\cV$ obtains the structure of an $\cE_3$-algebra in $R\otimes \cV$ by the Deligne conjecture: it has the structure of an $\cE_1$-algebra (given by composition of endomorphisms) in $\cE_2$-algebras, which by Lurie--Dunne additivity gives an $\cE_3$-algebra structure. 
From the equivalence of monoidal $R$-linear $\oo$-categories $\cD\simeq \Perf_{R\otimes A}(\cV)$, 
we obtain an equivalence $\End_\cD(1)\simeq R\otimes A$ as $R$-linear $\cE_2$-algebras in $\cV$, 
and hence a point of~${\rm tDef}(A)(R)$.
\end{proof}

\begin{crl}\label{cor.filt.E3.q.gp}
Given a choice of level $\lambda$, consider the $\cE_2$-monoidal deformation $\Rep^{\dg}_{\fin}(\U_\hbar\g)$ of $\Rep^{\dg}_{\fin}(\g)$ given by the Drinfeld--Jimbo quantum group. 
The derived endomorphisms of the unit $\C^*_\hbar(\fg)$ obtains the structure of a filtered $\cE_3$-algebra, 
and there is a natural equivalence 
\[
\Perf_{\C^*_\hbar(\fg)}(\cV) \simeq \Rep^{\dg}_{\fin}(\U_\hbar\g)
\]
of $\cE_2$-monoidal $\oo$-categories.
\end{crl}

\section{Chern--Simons theory and its factorization algebra}
\label{sec: CS}

The main result of this section is the following.

\begin{prp}
A choice of level $\lambda \in \hbar \H^3(\g)[[\hbar]]$ determines a perturbative quantization of Chern--Simons theory,
which in turn produces a nontrivial deformation $\cA^\lambda$ of $\cA^\cl = \C^*(\g)$ as a filtered $\cE_3$-algebra.
\end{prp}

This filtered $\cE_3$-algebra $\cA^\lambda$ provides the framed 3-disk algebra 
appearing in our construction of Reshetikhin--Turaev invariants.

The proof of this proposition has three steps:
\begin{itemize}
\item perturbatively quantize Chern--Simons theory, following~\cite{CosBook} in Section~\ref{subsec: def cplx for CS} and the appendix;
\item obtain a factorization algebra by the machinery of \cite{CG1,CG2} in Sections~\ref{sec: classical observables} and ~\ref{subsec: quantum observables}; and
\item verify it is a nontrivial deformation in Section~\ref{subsec: nontrivial def}.
\end{itemize}
Note that our discussion of the first step follows Costello's approach \cite{CosCS, CosBook},
which carefully treats these constructions within the BV formalism and is deformation-theoretic in nature.
A review of the Feynman diagrammatics and associated issues appears in Appendix~\ref{pure CS quantization}.
It is purely expository, 
since the first systematic mathematical treatment of the perturbative quantization of Chern--Simons theory was by Axelrod--Singer \cite{AxeSinI,AxeSinII};
contemporaneous work by Kontsevich is outlined in~\cite{KonECM}.
(Many others developed this perturbative expansion, particularly physicists \cite{GMM},
for which \cite{GuaBook} is a systematic reference. 
In \cite{IacCS, CatMne} one can find other helpful treatments of Chern--Simons theory within the BV formalism.) 

Before embarking on the proof, we describe classical Chern--Simons theory in the form relevant to us
and then analyze its factorization algebra of observables,
explaining why it is equivalent to $\C^*(\g)$.
The quantum observables exist automatically by the main theorem of \cite{CG2},
since we exhibit a nontrivial BV quantization.
Unfortunately, they are not amenable to direct understanding,
as the factorization structure involves the full complexity of the Feynman diagrams.\footnote{The case of {\em abelian} Chern--Simons theory is, however,
amenable to understanding---since it is a free theory---and 
it is analyzed in Section 4.5 of~\cite{CG1}.}
Instead, in Section~\ref{subsec: nontrivial def} 
we verify it is a nontrivial deformation by somewhat indirect methods.

\subsection{Classical Chern--Simons theory}

We describe the perspective relevant for our purposes,
for which Section 4, Chapter 5 of \cite{CosBook} or \cite{CosCS} are references.
(For a discussion of much more, particularly many aspects we ignore, see \cite{FreCCS}.)
Here we take a slightly expository approach as this perspective may be slightly unusual;
we explain the connection with the standard physics terminology in Section~\ref{subsec:CSfields}.

\subsubsection{}

Any classical field theory has {\em equations of motion}, 
the system of partial differential equations that govern the physical system under consideration.
For Chern--Simons theory on a manifold $M$ with principal $G$-bundle $P \to M$ 
and $G$ a finite-dimensional Lie group,
the fields are $\g$-valued connections $\nabla$ for $P$,
and the equation of motion is the zero curvature equation $F_\nabla = 0$.
Suppose for simplicity that $P$ is trivialized 
and thus we have a preferred connection $\d$
so that every other connection can be written as $\nabla = \d + [A, -]$,
where $A \in \Omega^1(M) \otimes \g$.
Then the zero curvature equation is the Maurer--Cartan equation 
\[
\d A + \frac{1}{2} [A,A] = 0.
\]
This geometric situation will suffice for us as we will take $M = \RR^3$ throughout.

The reader may have noticed that this equation of motion does not depend on $M$ having dimension three.
What is special about dimension three is that the Maurer--Cartan equation becomes variational:
it is the  Euler--Lagrange equation for the action functional
\[
CS(A) = \frac{1}{2} \int_M \langle A, \d A \rangle_\g + \frac{1}{6} \int_M \langle A, [A \wedge A] \rangle_\g,
\]
where $\langle-,-\rangle_\g$ denotes a nondegenerate invariant inner product on~$\g$.
(The existence of this pairing is a constraint on the kind of Lie algebras we consider.)

The formula for $CS$ does not make sense for arbitrary 1-form $A$ if $M$ is not compact, 
so it is not always a well-behaved function on the whole space of $\fg$-valued 1-forms.
From the point of view of variational calculus and classical field theory, however, this issue is irrelevant.
What we really care about are variations of $CS$ with respect to compactly-supported fields.
In other words, there is a foliation of the space of fields $\Omega^1(M,\fg)$ by the compactly-supported fields,
namely 
\[
\Omega^1_c(M,\fg) \subset T_{A} \Omega^1(M,\fg) \cong \Omega^1(M,\fg)
\]
at each field $A$. 
There is a well-defined section of the linear dual bundle to this foliation, 
which we might denote $\delta CS$ as it is a kind of 1-form on the space of fields,
and variational calculus says we want the zero locus of this section.
In explicit terms, 
\[
\delta CS(A) = \d A + \frac{1}{2} [A,A], 
\]
the curvature of the connection $\d + A$,
so the zero locus cuts out the flat connections.

\subsubsection{}

So far we have ignored an important aspect of this gauge theory:
we want to consider connections up to gauge equivalence,
i.e., up to the natural automorphisms of the bundle $P$ over~$M$.
In short, the correct space of fields forms a stack, namely connections modulo gauge transformations.
The equations of motion are then used to pick out the locus of flat connections modulo gauge equivalences.
Here one should also be careful to take the correct derived locus.

Instead of expressing the theory in an as-yet undeveloped theory of derived differential geometry,
we will simply state the object with which we work, 
which can be interpreted as a formal derived stack,
a well-defined object in mathematics~\cite{LurSAG,ToeVez}.

\begin{dfn}
\label{dfn:g^M}
Let $M$ be a smooth manifold and
let $\fg$ be a finite-dimensional Lie algebra.
Let $\g^M$ denote the dg Lie algebra $\Omega^*(M,\fg)$ where the bracket satisfies
\[
[\alpha \otimes x, \beta \otimes y] = \alpha \wedge \beta \otimes [x,y]_{\fg},
\]
with $\alpha,\beta \in \Omega^*(M)$ and $x,y \in \fg$, and the differential is $\d \otimes \id_\fg$.
From hereon we will simply denote the differential by $\d$.
This dg Lie algebra has Maurer--Cartan equation
\begin{equation}\label{MCeqn}
\d A + \frac{1}{2} [A,A]= 0
\end{equation}
for $A \in \Omega^1(M,\fg)$, and 
given a solution $A$, the operator $\d + A$ is a flat connection on the trivial bundle $\underline{\fg} \to M$.
\end{dfn}

From the perspective of derived geometry, 
this dg Lie algebra $\g^M$ describes the formal neighborhood of the trivial connection $\d$ inside a derived moduli space of flat connections on the trivial $\fg$-bundle.
We denote this formal moduli space by~${\rm B}\g^M$.

\subsubsection{}
\label{subsec:CSfields}

Let us mention the standard physical terminology for this situation.
Thus, the usual {\em fields} are connection 1-forms $A \in \Omega^1(M) \otimes \g$,
and these are given cohomological (or {\em ghost}) degree zero.
Note that in physics the space of fields is a shift down by one of the underlying graded vector space of~${\rm B}\g^M$.

These fields are precisely the physically relevant (or ``naive'') fields with which we started.
The degree -1 fields---called {\em ghosts}---are $\Omega^0(M) \otimes \g$.
These describe the Lie algebra of gauge symmetries.
The degree 1 fields---called {\em antifields}---are $\Omega^2(M) \otimes \g$,
and the degree 2 fields---called {\em antighosts}---are $\Omega^3(M) \otimes \g$.
These provide the conjugate variables to the fields and ghosts, respectively,
with respect to the shifted symplectic pairing 
\[
\langle \alpha \otimes x, \beta \otimes y \rangle = \pm \int \alpha \wedge \beta \langle x,y\rangle_\g.
\]
The action functional is then the standard BV action arising from the ``naive'' action we started with:
abusively we denote it
\[
CS(A) = \frac{1}{2} \langle A, d_{dR} A \rangle + \frac{1}{6} \langle A, [A \wedge A] \rangle,
\] 
where $A$ now denotes an arbitrary element of the fields.

\begin{rmk}
The physicist's convention is natural in the following sense:
the fields in a BV theory are always the underlying tangent complex of a solution to the variational PDE picked out by an action functional.
It is a nontrivial result of derived geometry that the shifted tangent complex admits a natural Lie algebra structure encoding the full formal neighborhood of that solution, rather than just the tangent space.
\end{rmk}

Due to this shifted symplectic structure, the commutative dg algebra of classical observables 
(discussed in more detail below) should have a shifted Poisson structure.
The differential on this algebra is a Hamiltonian vector field $\d^\cl = \{CS,-\}$,
and the square-zero property $(\d^\cl)^2 = 0$ is equivalent to the {\em classical master equation} (CME),
which says $\{CS,CS\}$ is Poisson-central (i.e., is a constant).
In the BV formalism, the CME is thus a condition for a putative action functional to define a well-posed classical BV theory.

\begin{rmk}
This dg Lie algebra makes sense for any smooth manifold $M$, independent of dimension.
In dimension 3 and when $\g$ is equipped with a nondegenerate invariant inner product, however, one finds that the formal moduli space ${\rm B}\g^M$ has a 1-shifted Poisson structure.
This structure is a general feature of variational PDE.
The {\em derived} critical locus of a function possesses a natural shifted Poisson structure,
and the formal neighborhood of a solution to a variational PDE is precisely the formal neighborhood of a point in a derived critical locus of the action functional.
For extensive discussion, see the chapters on classical field theory in~\cite{CG2}.
\end{rmk}

\subsection{The factorization algebra of classical observables}
\label{sec: classical observables}

The observables of classical Chern--Simons theory are easy to describe and understand.
We begin by describing everything without field-theoretic language and 
then explain the interpretation via field theory.
At first we consider global observables on $M$ and then turn to the local-to-global structure.

\subsubsection{}

The ``classical observables'' should be the dg commutative algebra of functions on the formal moduli space ${\rm B}\g^M$.
In formal derived geometry, if a formal moduli space corresponds to a dg Lie algebra $\mathfrak{h}$, 
then its algebra of functions corresponds to $\C^*(\mathfrak{h})$.
Hence, for classical Chern--Simons theory,
the algebra of classical observables is
\[
\C^*(\g^M) = \C^*(\Omega^*(M) \otimes \g).
\]
It is best here to remember that de Rham forms come equipped with a natural topology,
especially when we turn to constructing a quantization,
so we need to take into account some issues of functional analysis in defining the Chevalley--Eilenberg cochains of this dg Lie algebra. 

These issues are mild and dealt with systematically in \cite{CG1},
which describes a well-behaved category ${\rm DVS}$ of {\em differentiable vector spaces}.
As a brief sketch of the approach there, 
we mention the following.
For $\cE$ the smooth sections $\cE$ of a vector bundle $E \to M$,
the appropriate linear dual $\cE^\vee$ is given by the compactly supported distributional sections of $E^\vee \otimes \Dens \to M$,
where $E^\vee$ denotes the vector bundle dual and $\Dens \to M$ denotes the density line.
For $\cE = C^\infty(M,E)$ and $\cF = C^\infty(M,F)$,
the appropriate tensor product $\cE \otimes \cF$ is given by 
smooth sections of the vector bundle $E \boxtimes F \to M \times N$,
where $E \boxtimes F = \pi_M ^*E \otimes \pi_N^*F$.
We define tensor products for compactly supported or distributional sections similarly
as such sections of the appropriate vector bundle over the product space $M~\times~N$.

Using these notions of dual and tensor product, 
we construct $\C^*(\g^M)$ in the usual way.
Explicitly, we have the underlying vector space
\[
\C^\#(\g^M) = \prod_{n \geq 0} \Sym^n((\Omega^*(M)^\vee \otimes \fg^\vee)[-1])
\]
equipped with a differential 
\[
\d_{dR} \otimes \id_{\fg^\vee} + \d_{b}
\]
where the de Rham differential denotes the natural differential on the compactly supported de Rham currents
(which are the distributional dual to the de Rham forms) 
and $\d_b$ denotes the operator determined by the Lie bracket by the usual formula for Lie algebra cohomology.

Note that $\C^*(\g^M)$ is equipped with the usual decreasing filtration for Chevalley--Eilenbrg cochains:
\[
F^k \C^*(\g^M) = \prod_{n \geq k} \Sym^n((\Omega^*(M)^\vee \otimes \fg^\vee)[-1]),
\]
which manifestly respects the differential.

\subsubsection{}
\label{subsec:obscl special values}

To clarify the kind of information contained in these observables,
let us describe them for some simple choices of manifold~$M$.
\begin{itemize}
\item For $M = \RR^3$, the Poincar\'e lemma implies
\[
\C^*(\g^{\RR^3}) \simeq \C^*(\g).
\]
In other words, the classical local observables are just functions of the ghosts, 
and the naive, physically relevant observables of degree zero cohomology are just constant functions.
\item For $M = S^1 \times \RR^2$, the K\"unneth isomorphism and Poincar\'e lemma implies
\[
\C^*(\g^{S^1 \times \RR^2}) \simeq \C^*(\g[\epsilon]) = \C^*(\g, \csym(\g^\vee)),
\]
where $\H^*(S^1) = \CC[\epsilon]$ with $|\epsilon| = 1$. 
In other words, the naive, physically relevant observables of degree zero cohomology are $\csym(\g^\vee)^\g$, which are trace-class functions on the Lie algebra~$\g$.
This result matches our expectation that the relevant observables for $\g$-local systems on a circle ought to be functions of monodromy, 
since the adjoint quotient $\g/\g$ parametrizes possible monodromy.
\item For $M = \Sigma_g \times \RR$, where $\Sigma_g$ is the closed Riemann surface of genus $g$, we find
\[
\C^*(\g^{\Sigma_g \times \RR}) \simeq \C^*(\g\otimes \H^*(\Sigma_g)).
\]
In particular, the degree zero cohomology has a component 
\[
\csym(\g^\vee \otimes H_1(\Sigma_g))^\g,
\]
which encodes the naive algebra of functions on a formal neighborhood of the trivial connection inside the character variety.\footnote{Note that $\g^\vee \otimes H_1(\Sigma_g,\RR)$ is the linear dual to $\g \otimes \H^1(\Sigma_g,\RR)$, which is the tangent space to the character variety at the trivial connection, if one ignores gauge transformations. Hence the $\g$-invariant formal power series on that tangent space are functions near the trivial connection.}
\end{itemize}
In general, it is easy to understand the classical observables on a manifold $M$ 
so long as one understands the cohomology of~$M$.

\subsubsection{}

Those explicit examples bring to the foreground an important feature of our constructions:
it is functorial in the choice of manifold~$M$.
We now wish to articulate the kind of structure possessed by the observables as a function of the manifold.

It may help to know that we have a sheaf of dg Lie algebras $\Omega^* \otimes \fg$ on $M$,
which we view as a sheaf of formal moduli spaces.
This sheaf assigns to an open set $U \subset M$ the deformations of the trivial flat $\fg$-bundle on $U$.
In other words, we are thinking about the sheaf of solutions to the Maurer--Cartan equation,
which forms a sheaf because partial differential equations are local in nature.

This sheaf is locally constant because the de Rham complex is locally constant (up to quasi-isomorphism).
This property should be clear since the sheaf of local systems is manifestly locally constant.

Since we have a sheaf of formal moduli spaces, 
we expect to obtain a cosheaf of dg commutative algebras by taking the algebra of functions open set by open set.
Thus, they also should satisfy a weaker set of axioms that characterize a factorization algebra of cochain complexes.
We now show that they do. 

\begin{dfn}
The {\em prefactorization algebra of classical observables for Chern--Simons theory} is the precosheaf
\[
\Obs^{\cl} \colon {\rm Opens}(M) \to \Ch({\rm DVS})
\]
assigns to the open set $U \subset M$, the dg commutative algebra $\C^*(\g^U)$.
In fact, all the structure maps respect the filtration and so we can refine the target to complete filtered cochain complexes in~${\rm DVS}$.
\end{dfn}

This construction is well-behaved:
Theorem 6.6.1 of \cite{CG1} immediately specializes to the following.

\begin{lmm}
The precosheaf $\Obs^\cl$ satisfies the following:
\begin{itemize}
\item The structure maps are maps of commutative dg algebras, and so $\Obs^\cl$ is precosheaf of commutative dg algebras.
\item It satisfies codescent for the standard topology, and so $\Obs^\cl$ is cosheaf of commutative dg algebras.
\item It thus satisfies codescent for the Weiss topology and so $\Obs^\cl$ is a {\em commutative} factorization algebra.
\end{itemize}
\end{lmm}

Moreover, as the originating sheaf of dg Lie algebras is locally constant
this factorization algebra $\Obs^\cl$ is locally constant.
By Theorem 5.4.5.9 of \cite{LurHA}, it thus defines an $\cE_3$-algebra~$\cA^\cl_\fg$.
We therefore obtain the following.

\begin{lmm}
\label{poin}
The Poincar\'e quasi-isomorphism $P$ determines a filtered quasi-isomorphism of filtered $\cE_3$-algebras
$\cA_\fg^\cl \simeq \C^*(\fg)$,
where $\C^*(\fg)$ is a filtered $\cE_3$-algebra by forgetting down from filtered $\cE_\infty$ algebras.
\end{lmm}

\subsection{The deformation complex}
\label{subsec: def cplx for CS}

The formalism developed in~\cite{CosBook} allows one to give an obstruction-theoretic proof that BV quantizations exist without computing any Feynman diagrams.
A key ingredient is the {\em deformation complex} of a classical BV theory,
developed in section 4, chapter 5 of \cite{CosBook}.
We foreground this cohomological machinery as it will play a parallel role for the coupled theory,
but we will exhibit an explicit quantization of Chern--Simons theory in the next section.

As a gloss, the deformation complex of a classical BV theory consists of the cochain complex of {\em local functionals} on the fields,
i.e., Lagrangian densities.
(Note that we allow local functionals of arbitrary cohomological degree.)
The BV bracket defines a shifted Lie bracket on the local functionals.
The differential is $\{S^{\cl},-\}$, where $S^{\cl}$ is the classical action.
Hence, the deformation complex is a shifted dg Lie algebra.
Any degree zero element $I$ such that
\[
0 = \{S^{\cl} + I, S^{\cl} + I\} = 2 \{S^{\cl}, I\} + \{I,I\}
\]
is a shifted Maurer--Cartan element and hence determines a new classical BV theory 
whose action functional is $S^{\cl} + I$.
In particular, degree 0 cocycles determine first-order deformations of the classical BV theory.
Cocycles in degree -1 encode local symmetries of the classical theory;
and obstructions to satisfying the quantum master equation end up being degree 1 cocycles.

In Section 14, Chapter 5 of \cite{CosBook},
it is shown that the deformation complex of Chern--Simons theory on $\RR^3$ has a simple form:
\[
\Def_{CS} \simeq \C^{\geq 1}_{\Lie}(\g)[3].
\]
In other words, the deformation complex is encoded by the {\em reduced} Lie algebra cohomology.
For $\g$ semisimple, we thus see that 
\begin{itemize}
\item $\H^1(\Def_{CS}) \cong \H^4(\g) = 0$, 
and hence there are no obstructions to BV quantization, and
\item $\H^0(\Def_{CS}) \cong \H^3(\g)$, 
so that each power in $\hbar$, we are free to modify the classical action by changing the pairing,
as we just described above.
\end{itemize}
Moreover, different choices of level determine distinct quantizations: 
there is no local symmetry that identifies them.

\subsubsection{${\rm SO}(3)$-equivariant quantization}

The following brief discussion is not needed for the central results of this paper,
but it may illuminate some facets of the situation.

It is natural to ask whether we can quantize preserving symmetries of the classical theory.
We examine here the question of whether the rotational symmetry under ${\rm SO}(3)$ lifts.
Note that we are not just asking whether, say, the algebra of observables is isomorphic under a rotation,
but how to specify that isomorphism.
From the perspective of higher algebra, we are asking whether the quantum observables can be lifted to an algebra over the operad of oriented (but not framed) $3$-disks, equivalently, a framed $\cE_3$-algebra.
A refinement of the deformation complex allows us to conclude that it can be lifted.
Under the pipeline described in the introduction,
this lift determines a ribbon structure on the category of line defects.

Let us note first an important issue: the distinction between a discrete and a smooth action,
to use the terminology of \cite{CG1}.\footnote{See Section 3.7 therein.  
This distinction is an avatar of a familiar distinction in topology, where there is a big difference between the classifying space ${\rm B}G$ of a Lie group $G$ and $B(G^\delta)$ of its underlying discrete group~$G^\delta$. Our terminology here is connected with functional analytic issues in our constructions.}
It is straightforward to construct quantum Chern--Simons theory on $\RR^3$ in a way that is manifestly equivariant with respect to ${\rm SO}(3)$ as a discrete group.
Simply use the gauge-fixing arising from the Euclidean metric on $\RR^3$,
which ensures that the (regularized) BV bracket and BV Laplacian are strictly ${\rm SO}(3)$-equivariant 
and so is RG flow.

On the classical theory, 
there is an infinitesimal action of $\mathfrak{so}(3)$,
whose classical Noether current~is
\[
J_\xi(A) = \int \langle A, \cL_\xi A \rangle
\]
where we identify an element $\xi \in \mathfrak{so}(3)$ with a vector field on $\RR^3$ 
and which hence has a canonical action by the Lie derivative $\cL_\xi$ on differential forms.
This action makes $\Obs^\cl$ a {\em smoothly} ${\rm SO}(3)$-equivariant factorization algebra 
(one obtains the desired derivation of the factorization algebra by bracketing with the local functional $J_\xi$).
What is less obvious is that one can quantize this smooth action.

Working equivariantly with respect to $\mathfrak{so}(3)$ corresponds working over the base ring $\C^*(\mathfrak{so}(3))$, as a filtered dg commutative algebra.
The equivariant classical observables are the factorization algebra $\C^*(\mathfrak{so}(3), \Obs^\cl)$ with values in filtered $\C^*(\mathfrak{so}(3)$-modules.
To quantize equivariantly is to solve the QME over the base ring $\C^*(\mathfrak{so}(3))$.\footnote{This process is well-defined and extensively developed in Chapter 13 of \cite{CG2},
although we do not use anything subtle here.}
One can examine the $\mathfrak{so}(3)$-equivariant deformation complex $\Def^{\mathfrak{so}(3)-eq}_{CS}$,
in the style of our discussion above.

\begin{lmm}
The $\mathfrak{so}(3)$-equivariant deformation complex of Chern--Simons theory on $\RR^3$ satisfies
\[
\Def^{\mathfrak{so}(3)\text{-eq}}_{CS} \simeq \C^{\geq 1}(\mathfrak{so}(3)) \otimes \C^{\geq 1}(\g)[3].
\]
\end{lmm}

The key observation is that the quasi-isomorphism $\CC \hookrightarrow \Omega^*(\RR^3)$ of the Poincar\'e lemma is $\mathfrak{so}(3)$-equivariant.
This identification means that the argument from Section 14, Chapter 5 of \cite{CosBook}, carries over.

\begin{crl}
When $\g$ is semisimple, there are no obstructions to $\mathfrak{so}(3)$-equivariant quantization as $\H^*(\mathfrak{so}(3)) = \CC \oplus \CC[-3]$.
\end{crl}

See \cite{EllGwi} for a treatment of such results for all topological AKSZ theories, including Chern--Simons theories, notably Proposition~1.4, which implies the lemma and its corollary.

\subsection{The factorization algebra of quantum observables}
\label{subsec: quantum observables}

We have described a family of BV quantizations of classical Chern--Simons theory,
and now we turn to understanding the associated factorization algebras.
Applying the main result of \cite{CG2}, Theorem 9.6.0.1,
we deduce the following result.

\begin{prp}
\label{prp: CS quant exists}
Each choice of level $\lambda \in \hbar \H^3(\g)[[\hbar]]$ determines a BV quantization of classical Chern--Simons theory on $\RR^3$, 
and hence it determines a factorization algebra $\Obs^\lambda$ with values in dg modules over $\CC((\hbar))$.
This factorization algebra reduces modulo $\hbar$ to a factorization algebra quasi-isomorphic to $\Obs^\cl$.
\end{prp}

As mentioned in the introduction, 
quantization produces a deformation as a {\em filtered} $\cE_3$-algebra.
Let us identify the relevant filtration at the level of the factorization algebra first.

Consider the underlying graded vector space of the quantum observables $\Obs^\lambda(\RR^3)$:
\[
\prod_{m,n \geq 0}  \hbar^m \Sym^n((\Omega^*(M)^\vee \otimes \fg^\vee)[-1]).
\]
There are three apparent gradings here: by cohomological degree, by power of $\hbar$, and by symmetric power.
The differential (for any choice of parametrix or scale) interacts with these gradings by preserving or increasing the power of $\hbar$ but possibly lowering the symmetric power by at most two.
This decrement is due solely to the BV Laplacian.

\begin{dfn}
\label{dfn: obs filtration}
Equip the quantum observables with a filtration where
\[
F^k \Obs^\lambda = \prod_{2m + n \geq k} \CC \hbar^m \otimes \Sym^n((\Omega^*(M)^\vee \otimes \fg^\vee)[-1]),
\]
which is preserved by the differential.
\end{dfn}

This filtration is compatible with support conditions and with structure maps, and hence it extends to the whole factorization algebra.

One compelling feature of this filtration is that the associated graded factorization algebra 
is the quantum observables for {\em abelian} Chern--Simons theory.
More accurately, it encodes Chern--Simons theory for the abelian Lie algebra $|\g|$,
by which we mean the underlying vector space of $\g$ equipped with the trivial bracket.
To see this, note that the differential on the associated graded is $\d_{dR} + \hbar \Delta$,
since the filtration kills contributions that increase the symmetric power 
(e.g., the term in the differential arising from the Lie bracket of~$\g$).

To sum up, we have the following.

\begin{lmm}
The quantum observables $\Obs^\lambda$ form a locally constant factorization algebra with values in complete filtered dg $\CC[[\hbar]]$-modules in ${\rm DVS}$.
Hence, they determine a filtered $\cE_3$-algebra.
\end{lmm}

\begin{proof}
It suffices to show that the associated graded---i.e., the the quantum observables for abelian Chern--Simons theory---is locally constant.
But here we can use a spectral sequence argument. 
For any inclusion $i: D \hookrightarrow D'$ of a small disk into a bigger disk,
there is a structure map associated to $i$ for the associated graded of observables.
Consider now the filtration by powers of $\hbar$.
Then there is a map of spectral sequences arising from the structure map for $i$,
and it is an isomorphism on the first page, as the {\em classical} observables for abelian Chern--Simons theory is locally constant.
\end{proof}

\subsection{Quantization yields a nontrivial deformation}
\label{subsec: nontrivial def}

It remains to show that this $\cE_3$ deformation is nontrivial.
We offer two arguments that exemplify some different features of our methods.
Both rely, however, on the fact that the associated graded of $\Obs^\lambda$ corresponds to abelian Chern--Simons theory,
which is a free field theory and hence can be analyzed in depth.
We will refer freely to the analysis provided in Section 4.5 of \cite{CG1} of abelian Chern--Simons theory.

\subsubsection{Reduction along a torus}

One approach works by shifting the problem to the more familiar domain of associative algebras,
using the pushforward of factorization algebras, 
which is a version of dimensional reduction.
Basically, we will show that the dimensional reduction of quantized Chern--Simons theory 
along a closed, oriented 2-manifold $\Sigma$ 
produces a nontrivial deformation quantization of the dimensionally-reduced classical observables.
We will focus here on the torus $T = S^1 \times S^1$ for simplicity.

Recall that for a map $\pi: X \to Y$, 
a factorization algebra $\cF$ on $X$ pushes forward to a factorization algebra $\pi_* \cF$
by the formula 
\[
\pi_* \cF(U) = \cF(\pi^{-1}U).
\]
Hence, for a locally constant factorization algebra $\cF$ on $\RR \times \Sigma$,
the pushforward $\pi_* \cF$ along the projection $\pi: \RR \times \Sigma \to \RR$
yields a locally constant factorization algebra on $\RR$,
which is equivalent to some $\cE_1$-algebra.

If $\cF$ is commutative, then the pushforward is also commutative.
If $\cF$ is not commutative, then the image need not be commutative.
Thus, the following result will verify that $\cA^\lambda$ is a nontrivial $\cE_3$ deformation of~$\C^*(\g)$.

\begin{lmm}
\label{lmm: push to torus}
The pushforward $\pi_* \Obs^\lambda$ along the projection $\pi: T \times \RR \to \RR$ yields a locally constant factorization algebra
whose associated $\cE_1$-algebra is a noncommutative deformation of~$\pi_*\Obs^\cl$.
\end{lmm}

At the level of cohomology, this relationship is already clear, as our arguments will show. 
Let $\alpha$ and $\beta$ be cycles that generate $\H_1(T,\RR)$.
The graded commutative algebra $\H^*(\pi_*\Obs^\cl)$ is then the Lie algebra cohomology
\[
\H^*(\g[\alpha,\beta]) \cong \H^*(\g) \otimes \csym( \g^\vee \oplus \g^\vee)^\g \otimes \Sym(\g^\vee[1])^\g,
\]
which has a nontrivial Poisson bracket of degree 0.
This algebra is a graded extension of the more familiar Poisson algebra $\csym( \g^\vee \oplus \g^\vee)^\g$,
which is the completed algebra of functions around the trivial connection on the character stack of the torus.
A BV quantization will yield a deformation quantization of this Poisson algebra.

\begin{proof}
For abelian Chern--Simons theory and on an arbitrary closed Riemann surface,
this lemma is Proposition 4.5.2 of \cite{CG1}.
Now consider $\g$ nonabelian.
Due to the filtration at hand (of Definition~\ref{dfn: obs filtration}),
$\pi_* \Obs^\lambda$ encodes a filtered $\cE_1$-algebra,
and its associated graded is the $\cE_1$-algebra of the abelian Chern--Simons theory with Lie algebra $|\g$.
This associated graded algebra is a nontrivial deformation of classical observables for this abelian Chern--Simons theory. 
Hence $\pi_* \Obs^\lambda$ is not filtered-equivalent to $\pi_* \Obs^\cl$ for nonabelian~$\g$.
\end{proof}

\subsubsection{Computing the $\cP_3$ bracket}

Recall that the cohomology of an $\cE_3$-algebra $\cA$ is a shifted Poisson algebra 
whose bracket has degree~2. 
Hence another approach to showing the BV quantization produces a nontrivial $\cE_3$ deformation of $\C^*(\g)$ is to directly determine the shifted Poisson bracket on $\H^* \cA^\q_\lambda$ and verify it is nontrivial. 

\begin{lmm}
\label{lmm: nontrivial P2 bracket}
To first order in $\hbar$,
the $-2$-shifted Poisson bracket on $\C^*(\g)[[\hbar]]$ induced by
the $\cE_3$ algebra structure on $\cA^\q_\lambda$ is determined by the pairing $\lambda$ modulo~$\hbar$ on linear observables, extended as a biderivation.
\end{lmm}

To show this, we will use a computation that is a version of the physical notion of {\em descent} 
as used in topological field theory. 
See \cite{BBBDN} for pertinent discussion of this construction in physical language;
see \cite{EllSaf}, notably Section~2.4, for a careful treatment for how $\cE_n$ algebras (in operadic style) are encoded in the observables of TFTs like Chern--Simons theory.
In brief, descent gives a way of constructing observables (or operators).
Let us sketch how it works in our situation.

We use the term {\em local observable} to mean an observable with support on a disk:
the local observables for the classical theory are thus $\Obs^\cl(D)$ for some disk $D$.
Note that local constancy assures us these are quasi-isomorphic for any choice of disk.
Indeed, we have seen the classical local observables are quasi-isomorphic to~$\C^*(\g)$.
On the other hand, given a submanifold $S \subset \RR^3$ and an element $\alpha$ of $\C^*(\g)$,
it is possible to produce an observable 
\[
\cO_{S,\alpha} \in \Obs^\cl(U)
\]
for any open $U \supset S$, as follows.

The first step is to promote such an $\alpha$ to a functional $\widetilde{\alpha}$ on the dg Lie algebra $\Omega^*(\RR^3) \otimes \g$ {\em with values in $\Omega^*(\RR^3)$}.
There are two ways to do this, one smart and the other not;
we describe both, starting with the approach that doesn't work well.

Consider the map of commutative dg algebras $\CC \to R = \Omega^*(\RR^3)$ embedding numbers as constant functions on $\RR^3$.
The dg Lie algebra $\Omega^*(\RR^3) \otimes \g= R \otimes \g$ is the Lie algebra obtained by base change along this map.
This relationship leads to an interesting property of the Lie algebra cochains.
First, note that viewing $R \otimes \g$ as a dg Lie algebra over the base ring $R$,
there is a dg commutative $R$-algebra $\C^*_{R}(R \otimes \g)$ of Lie algebra cochains over $R$.
On the other hand, there is also the dg commutative $R$-algebra $R \otimes_\CC \C^*(\g)$.
These are quite similar but not equal.
There is an isomorphism of underlying graded algebras
\begin{equation}
\label{basechange}
\C^\#_{R}(R \otimes \g) \cong R^\# \otimes_\CC \C^\#(\g)
\end{equation}
but the differentials differ.

We wish to produce a map
\[
\widetilde{}: \C^*(\g) \to \C^*_{R}(R \otimes \g)
\]
that promotes $\alpha$ to~$\widetilde{\alpha}$.
There is an obvious guess: simply send $\alpha$ to $1_R \otimes \alpha$ using the isomorphism~\eqref{basechange}.
This map has a serious drawback: it does not respect the differential on $\C^*_{R}(R \otimes \g)$.
It is possible, however, to fix this problem, which is the essence of descent.

Consider the endomorphism of $\Omega^*(\RR^3)$ given by $\iota_{\partial_i}$,
namely contraction with the vector field $\partial_i$.
Cartan's formula tells us that $L_{\partial_i} = [\d, \iota_{\partial_i}]$.\footnote{In the notation of \cite{BBBDN}, we have $Q = \d$, the de Rham derivative, $P_\mu = \partial_\mu$, and $Q_\mu = \iota_{\partial_\mu}$.
For a systematic comparison with $\cE_n$ algebras, see the proof of Theorem~2.23 in~\cite{EllSaf}.} 
Consider the ``number operator'' $N = \sum_{i=1}^3 \d x_i \,\iota_{\partial_i}$, which is a kind of number operator
because $N(\omega) = k \omega$ for any $k$-form $\omega$.
This operator $N$ extends to $\Omega^*(\RR^3) \otimes \g$ by acting as the identity on the Lie algebra,
and hence it also extends to an operator on functions on that Lie algebra.

\begin{dfn}
For $\alpha \in \C^*(\g)$, its {\em total descendent}~is
\[
\widetilde{\alpha} = e^{N}(\alpha) = \sum_{j=0}^\infty \frac{1}{j!}N^j(\alpha),
\]
an element of $\C^*_{\Lie, \Omega}(\Omega^*(\RR^3) \otimes \g)$.
\end{dfn}

Direct computation shows that this map $\alpha \mapsto \widetilde{\alpha}$ is a cochain map.

This construction does not yet produce an observable:
$\widetilde{\alpha}$ takes as input a field $A$ from $\Omega^*(\RR^3) \otimes \g$ and returns as output a differential form $\widetilde{\alpha}(A)$.
An observable would return, by contrast, a number.
To obtain an observable, we pick a linear functional on differential forms, i.e., a de Rham current.
Concretely, view a closed submanifold $S \subset \RR^3$ as the de Rham current sending a form $\omega$ to~$\int_S \omega$.

\begin{dfn}
For $\alpha \in \C^*(\g)$ and $S \subset \RR^3$, set
\[
\cO_{S,\alpha} = \int_S \widetilde{\alpha},
\]
an element of $\Obs^\cl(U)$ for any open set $U$ containing~$S$.
\end{dfn}

We call $\cO_{S,\alpha}$ an observable {\em supported on $S$ and descended from~$\alpha$}.

\begin{rmk}
We only spelled out the case where $\alpha \in \C^*(\g)$, but one can more generally take $\alpha \in \Obs^\cl(D)$ and obtain an element $\widetilde{\alpha} \in \Omega^*(\RR^3, \Obs^\cl(D))$ and hence an observable $\cO_{S,\alpha}$.
In other words, one can descend observables supported in some disk to an observable along an interesting submanifold,
by the formula using the Euler operator.
This process continues to work for quantum observables, so long as one has explicitly lifted operations like $\iota_{\partial_i}$ to the quantum level.
\end{rmk}

Let us now use this construction to witness a nontrivial shifted Poisson bracket on the quantum observables.

The $\cE_3$ operad has a space of binary operations homotopic to~$S^2$,
arising from configurations of two disks in larger disk.
In particular, the dg operad $\C_*(\cE_3)$ has a nontrivial class in degree~$-2$, corresponding to the fundamental class of~$S^2$.
Thus, the Poisson bracket arises by descending an operator along a representative of this~$S^2$ and taking its product with an operator inside the~$S^2$.

In explicit terms, we do the following.
Fix two small disks $D \subset D'$ and a sphere $S^2 \subset D' \setminus D$, winding around $D$ but contained in $D'$.
Fix an open neighborhood $U \supset S^2$ disjoint from $D$ but contained in $D'$.
Given two observables $\alpha, \beta$ in the disk $D$, we can descend $\alpha$ to $S^2$ and then use the factorization product to obtain an observable on the disk~$D'$:
take the composition
\[
\begin{array}{ccccc}
\Obs(D) \otimes \Obs(D) & \to & \Obs(U) \otimes \Obs(D) & \to & \Obs(D')\\
\alpha \otimes \beta & \mapsto & \cO_{S,\alpha} \otimes \beta & \mapsto & \cO_{S,\alpha} \beta
\end{array}.
\]
At the level of cohomology, this composite induces the shifted Poisson bracket:
\[
\{[\alpha], [\beta] \} = [\cO_{S,\alpha} \beta].
\]
It is trivial for $\Obs^\cl$ (since it is truly commutative) but we will see that it is nontrivial for~$\Obs^\q$.

\begin{proof}[Proof of Lemma~\ref{lmm: nontrivial P2 bracket}]
We describe this computation only for abelian Chern--Simons theory,
since the nonabelian quantum observables have a filtration whose $E_1$ page is the abelian gauge theory, 
replacing $\g$ by its underlying vector space, viewed as an abelian Lie algebra.
These abelianized observables know the first order Poisson bracket by Lemma~10.3.3 of~\cite{CG2}.

Let us focus on the case $\g = \g_a$, the 1-dimensional Lie algebra,
which is abelian.
Let $x$ denote a generator, so $\g_a = \CC x$,
and consider the cocycle $\alpha = x^\vee$ in $\C^*(\g_a) = \CC[x^\vee]$.
Then $\widetilde{\alpha}$ sends any field $\omega \otimes x$ to $\omega$.
For any choice of closed submanifold $S \subset \RR^3$,
we have
\[
\cO_{S,\alpha}(\omega \otimes x) = \int_S \omega.
\]
This observable is a linear functional of the fields, and it is a cocycle because $S$ has no boundary.

We note that $\alpha$ is also a cocycle for the quantum differential,
since it is linear and hence must be annihilated by the BV Laplacian.

A particularly simple case is $S = \{0\}$, the point at origin. 
Let us compute the putative bracket of $\cO_{\{0\}, \alpha}$ with itself.
To do this, consider the disks 
\[
D = \{ |v| < 1/2\} \subset D' = \{|v| < 1\},
\]
the sphere
\[
S = \{ |v| = 3/4\}, 
\]
and the open neighborhood
\[
U = \{ 5/8 < |v| < 7/8\}.
\]
To compute the bracket, we wish to identify the cohomology class of $\cO_{S,\alpha} \cO_{\{0\},\alpha}$ in~$\Obs^\q(D')$.

Note that $S$ is the boundary of the closed disk $\overline{D} = \{|v| \leq 3/4\}$,
and so there is a 3-current $\mu = \int_{\overline{D}}$ such that $\d_{dR} \mu = \int_S$.
We thus find that
\begin{align*}
\d^\q(\cO_{\overline{D},\alpha} \cO_{\{0\},\alpha}) 
&= (\d_{dR} + \hbar \Delta)(\cO_{\overline{D},\alpha} \cO_{\{0\},\alpha}) \\
&= \cO_{S,\alpha} \cO_{\{0\},\alpha} + \hbar,
\end{align*}
where the juxtaposition $\cO_{S,\alpha} \cO_{\{0\},\alpha}$ denotes the quadratic element (namely the classical product).
Hence the cohomology of  $\cO_{S,\alpha} \cO_{\{0\},\alpha}$ in $\Obs^\q(D')$ is~$\hbar$.
In particular, it is nonzero, so that we obtain a nontrivial bracket.
\end{proof}

\section{Interlude on the charged 1-dimensional free fermion}
\label{sec: ff}

\def\ff{{\rm ff}}
\def\Ber{{\rm Ber}}

We wish to produce a Reshetikhin--Turaev theory whose $\cE_3$-algebra is $\cA^\lambda$,
which we just constructed using Chern--Simons theory.
We need to produce a $\cE_1$-algebra to label the tangles;
that is, we need the other piece of a framed $\cE_{1 \subset 3}$-algebra.
As a first step, this section exhibits a $\cE_1$-algebra of the form $\End_{\cA^\lambda}(\cV)$
for some perfect module $\cV$.
Our approach is to construct a 1-dimensional field theory whose factorization algebra has the appropriate form.
Explicitly, we work with a fermion living in a $\g$-representation,
i.e., a fermion with nontrivial $\g$-charge.

The main result of this section is the following.
It assembles Lemma~\ref{lmm: g-eq clifford alg}, Lemma~\ref{lmm: iso to hoch}, and Proposition~\ref{prp: character} into a single statement.

\begin{prp}
\label{prp: main ff}
Let $\rho: \g \to \End(V)$ be a finite-dimensional representation of a semisimple Lie algebra~$\g$.
The 1-dimensional fermion with values in $V$ admits a BV quantization
whose factorization algebra of quantum observables is equivalent to the super algebra given by the Lie algebra cohomology
\[
\C^*(\g,\Cl(V \oplus V^*)) \cong \C^*(\g, \End(\SS_\rho)),
\]
where $\SS_\rho$ denotes $\Sym (\Pi V)$ as the spinor representation of $\Cl(V \oplus V^*)$ and $\g$ acts through the canonical map 
\[
\g \to \mathfrak{so}(V \oplus V^*) \hookrightarrow \Cl(V \oplus V^*).
\]
Thus it yields an $\cE_{0 \subset 1}$ algebra from its observables.

Furthermore, the global sections on a circle have a distinguished element
given by the partition function of this $\rho$-charged fermion over a circle,
which manifestly encodes the holonomy of a flat $\fg$-connection on the spinor bundle.
That is, the partition function is
\[
\hbar^{-2\dim(V)}{\rm ch}_{\SS_\rho} \in \hh^0(\C^*(\g))[\hbar^{\pm 1}],
\] 
which is, up to a power of $\hbar$, the character ${\rm ch}_{\SS_\rho}$ of the spinor representation $\SS_\rho$ of~$\g$.
\end{prp}

Below we explain the players (such as the fermion and the Clifford algebra) in detail, 
but the key point is that we have a 1-dimensional theory whose observables encode the endomorphism algebra of a perfect $\C^*(\g)$-module, 
namely
\[
\End_{\C^*(\g)}(\C^*_{\Lie}(\g, \SS_\rho)),
\]
which is an equivalent expression to that above.
(We remark that we have set $\hbar = 1$ here for convenience, 
but we will track it in the work below, as it plays an important role in the Reshetikhin--Turaev theory.)

\subsubsection{Notations}

As we are working with fermions, we will need to work with super vector spaces, super cochain complexes, and so on.
If $V$ is an ordinary finite-dimensional vector space, 
let $\Pi V$ denote the associated purely odd vector space.
We use $\Sym(\Pi V)$ to denote the free commutative super algebra on $\Pi V$;
it looks like the exterior algebra $\Lambda V$,
except we view all odd wedge powers as being odd vectors.
(We reserve the notation $\Lambda V$ for the associative super algebra that is purely even.) 

Given a vector space $W$ with bilinear form $B$, 
we use $\Cl(W,B)$ to denote the Clifford algebra viewed as a super algebra.
That is, $\Cl(W,B)$ is the deformation of $\Sym(\Pi W)$ such that
\[
w w' + w' w = B(w,w')
\] 
for any $w,w' \in \Pi W$.
We use $\Cl(V \oplus V^*)$ to denote the case of $W = V \oplus V^*$ with $B$ the evaluation pairing.

There is a natural filtration on $\Cl(W,B)$ lifting that on $\Sym(\Pi W)$.
(Note that the grading does not lift.)
Let $\Rees(\Cl(W,B))$ denote the Rees algebra with respect to this filtration,
which is an algebra over the polynomials $\CC[\hbar]$.
Let $\Cl_\hbar(W,B)$ denote the localization by inverting $\hbar$,
which is an algebra over the Laurent polynomials~$\CC[\hbar,\hbar^{-1}]$.

\subsection{The 1-dimensional free fermion}

This section treats a fermion that is uncoupled to a gauge theory.
Its role is to develop the methods used in proving Proposition~\ref{prp: main ff} in the simplest relevant setting.

\subsubsection{The theory and its observables}

Let $V$ be a finite-dimensional vector space of dimension $d$.
Let $L$ be a 1-dimensional manifold without boundary.
We introduce now the fermionic theories of interest to us.

\begin{dfn}
The {\em free fermion} on $L$ associated to $V$ has fields
\[
\psi \in \cinf(L) \otimes \Pi V \quad\text{and}\quad \psib \in \cinf(L) \otimes \Pi V^*
\]
and action functional
\[
S_{\rm ff}(\psi,\psib) = \int_L (\psib, \d \psi)
\]
where $(-,-)$ denotes the evaluation pairing between $\Pi V$ and $\Pi V^*$.
The equations of motion are thus $\d \psi = 0 = \d \psib$.
\end{dfn}

This description is the usual one, but we will always work within the Batalin--Vilkovisky formalism and from hereon we include the antifields (i.e., 1-forms on $L$ with values in $\Pi V \oplus \Pi V^*$) as part of the data of the theory.

This theory is a free theory, and the methods of \cite{CG1} apply straightforwardly to obtain the following.
(Here by methods, we mean the shifted Heisenberg construction of the BV quantization combined with the arguments showing that the factorization algebra for the 1-dimensional free scalar field recovers the Weyl algebra, notably Section 4.3 of~\cite{CG1}.)

\begin{lmm}
\label{lmm: free fermion observables}
For $L = \RR$, the factorization algebra of classical observables is weakly equivalent to the factorization algebra associated to the super algebra $\Sym\, \Pi (V \oplus V^*)$.
The factorization algebra of quantum observables $\Obs^\q_{\rm ff}$ is weakly equivalent to the factorization algebra associated to the super algebra $\Cl_\hbar(V \oplus V^*)$.
\end{lmm}

\begin{proof}
This claim follows by recognizing that the Rees algebra of the Clifford algebra is the enveloping algebra of the super Lie algebra $Heis(\Pi (V \oplus V^*)$ obtained by centrally extending $\Pi (V \oplus V^*)$ via the evaluation pairing.
We can then apply the results that recover such enveloping algebras from factorization algebra constructions.
In detail, the quantum observables $\Obs^\q_{\rm ff}$ is equivalent to the twisted enveloping factorization algebra of the sheaf of dg Lie algebras $\Omega^*_\RR \otimes \Pi (V \oplus V^*)$, 
where the twist is determined by the Lie algebra cocycle 
\[
(\omega \otimes w, \omega' \otimes w') \mapsto \int_\RR \omega \wedge \omega' (w,w'),
\]
where $w,w' \in \Pi (V \oplus V^*)$ and the bracket means the evaluation pairing.
Now apply Proposition 3.4.1 of \cite{CG1}.
(See Sections 4.2 and 4.3 in \cite{CG1} for an extended discussion of this perspective.)
\end{proof}

This identification has important consequences when $L = S^1$.
When a factorization algebra on $\RR$ is locally constant, and hence can be identified with an $E_1$ algebra, 
the factorization homology on the circle is identified with the Hochschild chains of that algebra.
Hence, there is a natural but abstract equivalence
\[
\H^* \Obs^\q_{\rm ff}(S^1) \cong {\rm {\rm HH}}_*(\Cl_\hbar(V \oplus V^*)).
\]
Thanks to the explicit nature of our construction, however, we can give an explicit map
after choosing a Riemannian metric on~$S^1$. 

\subsubsection{Global observables on a circle}

Our goal in this subsection is to prove the following.

\begin{lmm}
\label{lem free obs to hoch}
A choice of Riemannian metric $g$ on $S^1$ determines a natural quasi-isomorphism
\[
\Obs^\q_{\rm ff}(S^1) \xto{\simeq} {\rm Hoch}_*(\Cl_\hbar(V \oplus V^*)))
\]
in such a way that we have an explicit cochain homotopy for any choice of path between two metrics.
\end{lmm}

One can give a very short proof by invoking the literature, 
but we will describe in detail the key steps as we will use the same techniques in constructing our  main result.

\begin{proof}
Let $g$ denote a Riemannian metric on $S^1$, and
let $\triangle$ denote the associated Laplacian.
By spectral theory, we get a decomposition $\cinf(S^1) = \cH \oplus \cH^\perp$ into the harmonic fields (i.e., $f$ such that $\triangle f = 0$) and a continuous complement on which $\triangle$ is an isomorphism.
Let $\triangle^{-1}: \cinf(S^1) \to \cinf(S^1)$ denote the extension to all smooth functions of the inverse to $\triangle|_{\cH^\perp}$.
We then obtain a Green's function for the de Rham complex by 
\[
G =\d^*\triangle^{-1}: \Omega^1(S^1) \to \Omega^0(S^1).
\]
Note that this operator varies continuously as we vary the underlying metric.

A key role is then played by the cochain homotopy equivalence of Hodge theory:
\[
\begin{tikzpicture}[baseline=(L.base),anchor=base,->,auto,swap]
     \path node (L) {$\CC[\epsilon]$} ++(2,0) node (M) {$\Omega^*(S^1)$} 
     (M.mid) +(0,.075) coordinate (raise) +(0,-.075) coordinate (lower);
     \draw (L.east |- lower) to node {$\scriptstyle \iota$} (M.west |- lower);
     \draw (M.west |- raise) to node {$\scriptstyle \pi$} (L.east |- raise);
     \draw (M.south east) ..controls +(1,-.5) and +(1,.5) .. node {$\scriptstyle G$} (M.north east);
\end{tikzpicture}
\]
Here $\CC[\epsilon]$ is to be understood as the de Rham cohomology of $S^1$, 
where $\epsilon$ is the metric volume form.
This equivalence allows us to transfer constructions on all fields to constructions on the much smaller $(\Pi V \oplus \Pi V^*)[\epsilon]$.

For instance, let us apply these tools to the classical observables (i.e., work modulo $\hbar$),
which are simpler.
In that case, 
\[
\Obs^{\cl}_{\rm ff}(S^1) = \Sym(\Omega^*(S^1) \otimes \Pi (V \oplus V^*)[1])
\]
and so the differential preserves symmetric powers since it is just the symmetrization of the exterior derivative. 
The cochain homotopy equivalence of Hodge theory
extends to a cochain homotopy equivalence
\begin{equation}\label{eqn: classical obs hkr}
\begin{tikzpicture}[baseline=(L.base),anchor=base,->,auto,swap]
     \path node (L) {$\Sym(\Pi (V \oplus V^*)[\epsilon][1])$} ++(4,0) node (M) {$\Obs_\ff^{\cl}(S^1) $} 
     (M.mid) +(0,.075) coordinate (raise) +(0,-.075) coordinate (lower);
     \draw (L.east |- lower) to node {$\scriptstyle \iota^\cl$} (M.west |- lower);
     \draw (M.west |- raise) to node {$\scriptstyle \pi^\cl$} (L.east |- raise);
     \draw (M.south east) ..controls +(1,-.5) and +(1,.5) .. node {$\scriptstyle G^\cl$} (M.north east);
\end{tikzpicture}
\end{equation}
where we add the superscript $\cl$ to $\pi$, $\iota$, and $G$ for the induced maps.
We can postcompose with the Hochschild--Kostant--Rosenberg map to get a quasi-isomorphism
\begin{equation}\label{eqn: classical obs hkr qiso}
\Obs_{\rm ff}^{\cl}(S^1) \xto{\simeq} {\rm Hoch}_*(\Sym(\Pi (V \oplus V^*)).
\end{equation}
The quantum case is a deformation of this classical situation.

Indeed, quantization of $\Obs^\cl$ simply means that we add the BV Laplacian to the differential.
(For a detailed discussion of these issues---including an interpretation of Wick's lemma in this setting---see Chapter 2 of~\cite{GwThesis}.)
In this setting we can apply the homological perturbation lemma to deform (\ref{eqn: classical obs hkr}):
\begin{equation}\label{eqn: quantum obs hkr}
\begin{tikzpicture}[baseline=(L.base),anchor=base,->,auto,swap]
     \path node (L) {$(\Sym(\Pi (V \oplus V^*)[\epsilon][1])[[\hbar]], \d^q)$} ++(4,0) node (M) {$\Obs_\ff^{\q}(S^1) $} 
     (M.mid) +(0,.075) coordinate (raise) +(0,-.075) coordinate (lower);
     \draw (L.east |- lower) to node {$\scriptstyle \iota^\q$} (M.west |- lower);
     \draw (M.west |- raise) to node {$\scriptstyle \pi^\q$} (L.east |- raise);
     \draw (M.south east) ..controls +(1,-.5) and +(1,.5) .. node {$\scriptstyle G^\q$} (M.north east);
\end{tikzpicture}
\end{equation}
It is helpful to rephrase both sides of this deformation retraction in more convenient terms.

The observables of any free BV theory can be viewed as the Lie algebra {\em chains} of a shifted Heisenberg Lie algebra constructed from the linear observables,
as discussed in the proof of Lemma~\ref{lmm: free fermion observables}.
In our case, we have
\[
\Obs_\ff^{\q}(S^1) = \C^\Lie_*(\Omega^*(S^1) \otimes \Pi (V \oplus V^*) \ltimes \CC \hbar[-1])[\hbar^{-1}],
\]
where the shifted pairing on observables determines a Lie bracket 
$[\alpha,\beta] = \hbar \langle\alpha,\beta\rangle.$
As the circle is formal (in the sense of rational homotopy theory),
we obtain an equivalence of dg Lie algebras
\[
\Omega^*(S^1) \otimes \Pi (V \oplus V^*) \ltimes \CC \hbar[-1] 
\simeq (\Pi V \oplus \Pi V^*)[\epsilon]\ltimes \CC \hbar[-1].
\]
In consequence, we have a natural quasi-isomorphism
\begin{equation}\label{eqn: quantum obs hkr qiso}
\Obs^\q_\ff(S^1) \xto{\simeq} \C_*^\Lie((\Pi V \oplus \Pi V^*)[\epsilon]\ltimes \CC \hbar[-1]).
\end{equation}
The right hand side is manifestly a deformation of~(\ref{eqn: classical obs hkr qiso}).

Now, there is also a natural deformation of the HKR map for a deformation quantization.
In particular, the HKR map deforms to a quasi-isomorphism
\[
\C_*^\Lie((\Pi V \oplus \Pi V^*)[\epsilon]\ltimes \CC \hbar[-1])[\hbar^{-1}] \xto{\simeq} {\rm Hoch}_*(\Cl_\hbar(V \oplus V^*))).
\] 
Composing with the map (\ref{eqn: quantum obs hkr}), we obtain the claim.
\end{proof}

\subsubsection{Berezin integration in the BV framework}

The right hand side of (\ref{eqn: quantum obs hkr}) may not look appealing as written, 
but it is precisely the usual BV complex encoding finite-dimensional fermionic integration,
as we will now show.
This explicit identification will be useful later in recognizing character formulas in terms of Feynman diagrams.

Fix a basis $\{x_1,\ldots,x_d\}$ of $\Pi V$ and let $\{\bar{x}_i\}$ denote the dual basis for $\Pi V^*$.
Write $\xi_i$ for $\bar{x}_i \epsilon$ and $\bar{\xi}_i$ for $x_i \epsilon$.
Then the underlying graded algebra is
\[
\C_*^\Lie((\Pi V \oplus \Pi V^*)[\epsilon]\ltimes \CC \hbar[-1]) \cong \CC[x_i,\bar{x}_i,\xi_i,\bar{\xi}_i,\hbar]
\]
where, due to the shift in taking Chevalley chains, we view $x_i$ and $\bar{x}_i$ as having degree -1 and odd parity while $\xi_i$ and $\bar{\xi}_i$ have degree 0 and are odd parity.
The differential comes from the bracket on the shifted Heisenberg Lie algebra, and on these generators  it has the form
\[
\hbar \Delta = \hbar \sum_{i =1}^d \frac{\partial}{\partial x_i}\frac{\partial}{\partial \xi_i} +\frac{\partial}{\partial \bar{x}_i}\frac{\partial}{\partial \bar{\xi}_i}.
\]
In this format the complex is manifestly the BV complex for finite-dimensional fermionic integration.
Its cohomology is concentrated in degree 0, and it is spanned by the ``top fermion'' $\xi_1 \cdots \xi_d \bar{\xi}_1 \cdots \bar{\xi}_d$.
In particular, after inverting $\hbar$ (or setting it to 1), we find that there is a natural quasi-isomorphism
\[
\begin{array}{cccc}
\int_\Ber: & \CC[x_i,\bar{x}_i,\xi_i,\bar{\xi}_i,\hbar^{\pm 1}] &\to& \CC[\hbar^{\pm 1}]\\
&  \hbar^k \prod_i x_i^{m_i} \bar{x}_i^{\bar{m}_i} \prod_i \xi^{n_i} \bar{\xi}_i^{\bar{n}_i} & \mapsto & \begin{cases} 
\hbar^k, &  \forall i,\,m_i = 0 = \bar{m}_i, n_i = 1 = \bar{n}_i \\
0,& \text{else}
\end{cases}
\end{array}
\]
That is, this quasi-isomorphism is precisely the Berezin integral.

Hence we have a composite quasi-isomorphism
\[
\Obs^\q_\ff(S^1) \xto{\simeq} \C_*^\Lie((\Pi V \oplus \Pi V^*)[\epsilon]\ltimes \CC \hbar[-1]) \xto{\int_\Ber}  \CC[\hbar^{\pm 1}]
\]
that computes the expected value of an observable for the free fermion on a circle.
We now turn to interpreting this map in terms of algebras of operators and their traces.

\subsubsection{Hochschild homology of the Clifford algebra}

This version of Berezin integration fits nicely with the Hochschild homology description of quantum observables.
The relationship is via some standard algebra.

\begin{lmm}
\label{lmm: cliff morita trivial}
The Clifford algebra $\Cl(V \oplus V^*)$ is isomorphic to the super matrix algebra $\End(\Sym(\Pi V))$, and hence it is Morita-trivial.
\end{lmm}

This isomorphism has another common name. 
Recall that there is a canonical inclusion ${\mathfrak so}(V\oplus V^*) \hookrightarrow \Cl(V \oplus V^*)$ of Lie algebras.
(The image is the component $\Lambda^2 (V \oplus V^*)$.)
Hence, this isomorphism identifies $\Sym(\Pi V)$ as a representation of ${\mathfrak so}(V\oplus V^*)$:
it is the {\em spinor} representation~$\SS_V$. 

In light of Section~\ref{sec: tangles}, we see this lemma implies 
that the spinor representation $\SS_V$ is a perfect module in $\ZZ/2$-graded vector spaces, 
and it determines a factorization algebra on oriented 1-manifolds with boundary,
assigning $\SS_V$ to the positively-oriented point, its dual $\SS_V^*$ to the negatively-oriented point,
and the Clifford algebra $\Cl(V \oplus V^*)$ to open intervals.

\begin{rmk}
We have seen that the Clifford algebra can be constructed as a factorization algebra on $\RR$ using the BV formalism. 
It is also possible to realize the spinor representation by analogous methods
(i.e., animate the spinors as boundary conditions for the free fermion),
but as we will not need such an explicit factorization model in this paper,
we do not develop it here.
\end{rmk}

\begin{proof}
When $V \cong \CC$, there is an isomorphism $\Cl(V \oplus V^*) \cong \End(\CC^{1|1})$ of superalgebras, 
which is realized by identifying $\CC^{1|1}$ with $\Sym(\Pi V)$, the natural Fock module for the Clifford algebra {\em aka} the spinor representation of ${\mathfrak so}(V\oplus V^*)$.
We will exhibit this isomorphism explicitly in a moment.
Note now that the 1-dimensional example determines the general case:
if $V \cong \CC^d$, then 
\[
\Cl(V \oplus V^*) \cong \End(\CC^{1|1})^{\otimes d} \cong \End(\Sym(\Pi V)) .
\]
Hence, we have given the module that realizes the Morita equivalence with~$\CC$.

For the sake of completeness, we spell out the 1-dimensional case.
Let $x$ be a basis element for $\Pi V$ so that $\Sym(\Pi V) = \CC[x]$, since $x^2 = 0$.
Then $\Cl(V \oplus V^*)$ has basis $\{1,x,\bar{x}, x \bar{x}\}$; 
and the multiplication is $x \cdot x= 0 = \bar{x} \cdot \bar{x}$ and $x \cdot \bar{x} + \bar{x} \cdot x = 1$.
We can view $\Sym(\Pi V)$ as the quotient module $\Cl(V \oplus V^*)/(\bar{x})$.
Let $p$ denote the action of the Clifford algebra on this module.
Hence, as matrices, we find
\[
p(x) =
\begin{pmatrix}
0 & 0 \\
1 & 0
\end{pmatrix},\quad
p(\bar{x}) =
\begin{pmatrix}
0 & 1 \\
0 & 0
\end{pmatrix},\quad
p(x\bar{x}) =
\begin{pmatrix}
0 & 0 \\
0 & 1
\end{pmatrix},\quad
p(\bar{x}x) =
\begin{pmatrix}
1 & 0 \\
0 & 0
\end{pmatrix},
\]
so we see that $p$ is an isomorphism.
\end{proof}

A Morita-trivial algebra always has trivial Hochschild homology,
because a bimodule realizing the trivialization explicitly provides a quasi-isomorphism between the two Hochschild chains.
Such standard results also hold in the our setting of $\ZZ/2$-graded modules.
(See Bass \cite{Bass} and Kassel \cite{KasZ2} for nice discussions.)
Here, the Morita trivialization by the spinor representation $\SS_V$ means that the Hochschild homology is spanned by taking trace over $\SS_V$,
which coincides with the Berezin integral.
As an aside we explicate this assertion with formulas.

\begin{crl}
The zeroth Hochschild homology ${\rm HH}_0(\Cl(V \oplus V^*))$ is 1-dimensional
and spanned by any element $x_1 \cdots x_d \bar{x}_1 \cdots \bar{x}_d$, 
where the $\{x_i\}$ are a basis for $\Pi V$ and the $\{\bar{x}_i\}$ are a basis for~$\Pi V^*$.
\end{crl}

\begin{proof}
Consider $V = \CC$, as the general case follows. 
Then $1 = [x,\bar{x}]$, using the commutator that keeps track of signs. 
(More traditionally, one might write $\{x,\bar{x}\}$ for $x\bar{x} + \bar{x}x$.)
Hence the identity 1 is exact in the Hochschild complex.
Similarly, $x = [x\bar{x},x]$ and $\bar{x} = [\bar{x}x,\bar{x}]$.
However, $x\bar{x}$ is not a commutator, so it spans the zeroth Hochschild homology.
\end{proof}

Under the isomorphism of Lemma \ref{lmm: cliff morita trivial}, 
we can identify the supertrace on matrices with a linear functional on the Hochschild homology.
For instance, in the case $\dim(V) = 1$, observe that 
the matrix 
\[
A=\begin{pmatrix}
a & b \\
c & d
\end{pmatrix}
\]
corresponds to the element 
\[
\alpha=a \bar{x} x + b \bar{x} + c x + d x \bar{x}.
\]
The supertrace is thus $a-d$, while the component of $\alpha$ along $x \bar{x}$ is $d-a$
because
\[
a \bar{x} x = a (1- x \bar{x}).
\]
If we used the ``top fermion'' component $\bar{x} x$ instead,
we would agree with the supertrace on the nose.

In other words, computing the universal trace of an element of $\Cl(V \oplus V^*)$ amounts to taking the coefficient of the ``top fermion,'' just as with Berezin integration.
Hence, after inverting $\hbar$, this lemma combines with our main lemma to produce a canonical map
\[
\H^0(\Obs^\q_{\rm ff}(S^1)) \xto{\cong} \CC[\hbar,\hbar^{-1}].
\]
This map is precisely the usual fermionic Wick expansion for the free fermion, 
via the relationship with the BV complex that we have just explained.

\subsubsection{A bulk-boundary free fermion}
\label{sec: bb ff}

So far we have examined the free fermion on manifolds without boundary and hence without boundary conditions.
We have shown how its BV quantization encodes the Clifford algebra and its Hochschild homology.
It is straightforward to construct the free fermion on a one-manifold with boundary and impose boundary conditions such as
\[
\psi|_{\partial L} = 0.
\]
More generally, for each boundary point $x \in \partial L$,
fix a Lagrangian subspace
\[
W_x \subset \Pi (V \oplus V^*)
\]
and consider the space of fields to be 
\[
\{ \Psi = (\psi, \psib) \,:\, \Psi|_x \in W_x\}.
\]
Then \cite{GRW} demonstrates a BV quantization that yields a constructible factorization algebra of observables on $L$ such that
\begin{itemize}
\item on the interior $L \setminus \partial L$, the observables agree with $\Obs^\q_\ff \simeq \Cl(V \oplus V^*)$ from the preceding subsections, and
\item on any small open set $U_x$ containing a boundary point $x \in \partial L$, the observables recover the Fock module $\Sym(\Pi(V \oplus V^*)/W_x)$ of the Clifford algebra,
which is isomorphic to the spinor representation.
\end{itemize}
This result is Proposition~5.1 of~\cite{GRW}.
That statement is formulated for a symplectic vector space but it applies {\it verbatim} to a symplectic odd vector space.

In short, we have the following statement.

\begin{lmm}
For a finite dimensional vector space~$V$,
the quantum observables of a free fermion valued in $\Pi(V \oplus V^*)$ with Lagrangian boundary conditions $\{W_x\}$ form an $\cE_{0 \subset 1}$ algebra equivalent to the Clifford algebra $\Cl(V \oplus V^*)$ acting on the spinor representation~$\SS_V$.
\end{lmm}

\subsection{The charged fermion}

Let $\g$ be a finite-dimensional Lie algebra, 
and let $\rho: \g \to \End(V)$ be a representation of $\g$ on a finite-dimensional vector space $V$.
(We will typically just talk about $V$ and suppress the explicit dependence on $\rho$.)
There is a dual action of $\g$ on $V^*$.
Thus, we obtain an action of $\g$ on the free fermion associated to $V$.
Going further, we can view elements of $\Omega^*(L) \otimes \g$ as ``background fields'' that act on the fermionic fields.

\begin{dfn}
The {\em charged fermion} on a one-manifold $L$ without boundary associated to a finite-dimensional representation $\rho: \g \to \End(V)$ has fields
\[
\psi \in \Omega^*(L) \otimes \Pi V,\; \psib \in \Omega^*(L) \otimes \Pi V^*, \quad\text{and}\quad A \in \Omega^*(L) \otimes \g[1]
\]
and action functional
\[
S_{\rm cf}(\psi,\psib, A) = \int_L (\psib, (\d + \rho(A))\psi) 
\]
where $(-,-)$ denotes the evaluation pairing between $\Pi V$ and $\Pi V^*$.
The equations of motion are thus 
\[
(\d + \rho(A))\psi = 0 = (\d + \rho^*(A))\psib.
\]
That is, a solution is a horizontal section for the connection $\d + \rho(A)$ and a horizontal section for the dual connection.
\end{dfn}

Note that
\[
S_{\rm cf}(\psi,\psib,A) = S_{\rm ff}(\psi,\psib) + I_0(\psi,\psib, A),
\]
where we call  
\[
I_0(\psi,\psib, A) = \int_L (\psib, \rho(A)\psi),
\]
the {\em interaction term}.

\subsubsection{The quantization of the charged free fermion}
\label{sec on quantizing charged FF}

We will now quantize the fermionic fields while leaving the connection 1-form fields as classical.
This charged theory is 1-dimensional so one can simply construct the BV quantization, 
as there are no analytic issues like divergences to worry about.
The underlying free theory is the free fermion already discussed.
We choose a Riemannian metric on the manifold $L$, which determines a gauge-fixing operator $\d^*$ on the fermionic fields.
When $L = \RR$, the propagator arises from the Heaviside step function:
\[
P^f(t,t') = \heav(t-t') = \begin{cases}1, & t > t'\\ -1, & t < t'\end{cases}.
\]
When $L= S^1= [0,1]/(0 \sim 1)$, the propagator is
\[
P^f(t,t') = \begin{cases}t-t'+\frac{1}{2}, & t < t'\\ t-t'-\frac{1}{2}, & t > t'\end{cases}.
\]
These data determine the weights assigned to the edges in any Feynman diagram.

The vertices are determined by the interaction term $I_0$.
It is a cubic function in total but only quadratic in the fermionic fields, which are the fields we quantize.
The only connected graphs arising from this interaction are thus trees,
whose first Betti number $b_1$ is zero, or wheels, whose first Betti number $b_1$ is one.
(A wheel means a graph with one loop---the wheel---but where that loop may have edges sprouting off of it.)

\begin{rmk}
The sum over trees corresponds to transferring the $L_\infty$ structure along the chain homotopy equivalence given by propagator 
(essentially, by Hodge theory).
\end{rmk} 

\begin{lmm}
\label{lmm: ffoneloopquant}
The charged fermion quantizes at one loop with quantized interaction $I^\q = I_{\rm tree} + \hbar I_{\rm wheel}$ 
where 
\[
I_{\rm tree}(\psi,\psib, A) = \sum_{\{\gamma \, : \, b_0 = 1, b_1=0\}} w_\gamma(P^f,I_0),
\]
which is quadratic in the fermions and has arbitrary polynomial degree in the $A$ field,
and
\[
I_{\rm wheel}(\psi,\psib, A) = \sum_{\{\gamma \, : \, b_0=b_1=1\}} w_\gamma(P^f,I_0),
\]
which is a function only of the $A$-field.
\end{lmm}

\begin{proof}
As mentioned already, there are no issues with regularization, 
so the expression for $I_{\rm wheel}$ is well-defined.
Since it is only a function of $A$, it is automatically annihilated by $\Delta$, 
which is nontrivial only on functions of the fermionic fields.
Likewise, $\{I_{\rm wheel},-\} = 0$ as the bracket only pairs with functions of fermionic fields.
Since $I_0$ is a quadratic function of the fermions,
the element $\Delta I_0$ is also a function only of the $A$ fields,
as the fermionic terms are contracted away.
(This observation applies to the scale $L$ versions of $I_0$, appearing by tree-level RG flow.)
Thus the regularized action satisfies the quantum master equation, 
as the expression 
\[
QI_0 + \{I_0,I_0\} + \hbar \{I_0, I_{\rm wheel}\} + \hbar \Delta I_0 + \hbar \Delta I_{\rm wheel} = \hbar \Delta I_0,
\]
which is only a function of the $A$ fields and hence only depends on the background fields.
In particular, $\{\Delta I_0,-\} = 0$ and hence $Q + \{I_0,-\} + \hbar \{I_{\rm wheel},-\} + \hbar \Delta$ determines a differential.
\end{proof}

These expressions in Lemma~\ref{lmm: ffoneloopquant} become more meaningful
when $L = S^1$ and we restrict to the harmonic fields,
i.e., $\psi$ and $\psib$ constant and $A$ fields of the form $X \d t$, with~$X\in\g$.
In other words, they become functions on the finite dimensional vector spaces $\g$, $V$, and~$V^*$.

\begin{lmm}\label{lmm: one loop term}
Let $\psi = v \in V$, $\psib = \bar{v} \in V^*$, and $A = X \d t$ be harmonic fields on $S^1$.
Evaluated on these fields, the quantized interaction has the form
\[
I_{\rm tree}(v,\bar{v},X \d t) = I_0(\psi,\psib,X \d t) = \langle \bar{v},\rho(X) v\rangle \int_{S^1} \d t
\]
and
\[
I_{\rm wheel}(X \d t)= -\log \det\left(e^{-\rho(X)/2} \Td(\rho(X))\right),
\]
where $\Td(t) = t/(1-e^{-t})$.
Here $\det$ means the determinant of an endomorphism on~$V$.
\end{lmm}

In Proposition \ref{prp: character} we explain how this one loop term corresponds to the character of a representation of~$\g$.

\begin{proof}
The tree term is straightforward.
For any constant fermion field $\psi =b$, the action by a constant 1-form $X \d t$ is another constant fermion field $\rho(X) v$.
But the propagator vanishes on constant fields,
so for any tree with an internal edge, the value on harmonic fields vanishes.
Hence the only contribution is from the trivalent vertex itself,
which is given by~$I_0$.

It remains to compute the one loop term on the circle,
which is straightforward.
(These computations here are the fermionic analogues of computations in \cite{GG},
so we refer the interested reader there for a more extensive treatment.)
First, note that the wheel with $k$ legs has Lie-theoretic factor given by
\[
\frac{1}{k} \tr(\rho(X)^k),
\]
with $X \in \g$.
Second, the analytic factor is 
\[
-\frac{2}{(2\pi)^k} \zeta(k),
\]
where $\zeta$ denotes the Riemann zeta function.
(See Lemma 8.4 of \cite{GG}.
The extra minus sign is the usual minus sign for a loop of fermionic fields.)
Hence, summing over all the wheels, we find
\[
I_{\rm wheel}(X \d t) = -\sum_{k\geq1} \frac{2(2k-1)!}{(2\pi)^k}\zeta(2k)\tr(\rho(X)^{2k}) 
= -\log\det\left(e^{-\rho(X)/2} \Td(\rho(X))\right),
\]
where $\Td(t) = t/(1-e^{-t})$.
(For a discussion of these power series manipulations, 
look at the proof---and surrounding discussion---of Theorem 8.6 in~\cite{GG}.)
\end{proof}

\subsubsection{The observables of the charged fermion}

We find the following result, which is the charged analogue of the observables for the free fermion.

\begin{lmm}
\label{lmm: g-eq clifford alg}
For $L = \RR$, the factorization algebra of classical observables is weakly equivalent to the factorization algebra associated to the superalgebra $\C^*(\g,\Sym\, \Pi (V \oplus V^*))$.
The factorization algebra of quantum observables $\Obs^\q_{\rm cf}$ is weakly equivalent to the factorization algebra associated to the superalgebra $\C^*(\g,\Cl_\hbar(V \oplus V^*))$.
\end{lmm}

In particular, the zeroth cohomology is the $\g$-invariant subalgebra:
\begin{enumerate}
\item The classical observables $\H^0(\Obs^{\cl})$ correspond to the superalgebra $(\Sym\, \Pi (V \oplus V^*))^\g$, and 
\item The quantum observables $\H^0(\Obs^{\q})$ correspond to the superalgebra $\Cl_\hbar(V \oplus V^*)^\g$.
\end{enumerate}

\begin{proof}
Both factorization algebras are locally constant and hence determine $\cE_1$-algebras.
The classical observables are equivalent to $\C^*(\g,\Sym\, \Pi (V \oplus V^*))$ by the Poincar\'e lemma.
The claim about the quantum observables can then be seen as taking the enveloping factorization algebra of a Heisenberg Lie algebra but over the base ring $\C^*(\g)$,
as in Section 4.3 of~\cite{CG1}.
\end{proof}

There is a natural generalization of Lemma~\ref{lmm: cliff morita trivial}.

\begin{dfn}
For $\rho: \g \to \End(V)$ a finite-dimensional representation, 
let $\SS_\rho$ denote the spinor representation $\SS_V = \Sym(\Pi V)$ of $\Cl(V \oplus V^*)$ equipped with the action of $\g$ induced by the canonical inclusion
\[
\g \xto{\rho \times \rho^*} \mathfrak{so}(V \oplus V^*) \hookrightarrow \Cl(V \oplus V^*)
\]
of Lie algebras.
We call $\SS_\rho$ the {\em $\rho$-charged spinors}.
\end{dfn}

It is straightforward algebra to verify that the argument for Lemma~\ref{lmm: cliff morita trivial} carries over to show the following.

\begin{lmm}
The Clifford algebra $\C^*(\g,\Cl(V \oplus V^*))$ in $\C^*(\g)$-modules is isomorphic to the super matrix algebra $\C^*(\g,\End(\SS_V))$, and hence it is Morita-trivial.
\end{lmm}

In consequence, we obtain a useful description of global observables on the circle. 
Again, the Morita triviality provides a canonical identification of Hochschild chains.

\begin{lmm}
\label{lmm: iso to hoch}
There is a natural isomorphism
\[
\H^*(\Obs^\q_{\rm cf}(S^1)) \cong {\rm HH}_*(\C^*(\g, \Cl_\hbar(V \oplus V^*))) \cong {\rm HH}_*(\C^*(\g))[\hbar^{\pm 1}]
\]
up to a choice of Riemannian metric on~$S^1$.
Because ${\rm HH}_0((\C^*(\g)) \cong \csym(\g^*)^\g$, every degree zero cocycle in the observables determines a class function with values in Laurent polynomials in~$\hbar$.
\end{lmm}

\begin{proof}
The first isomorphism is a general fact about locally constant factorization algebras on the circle in conjunction with the preceding lemma.
The second isomorphism is a direct computation: the dg algebra $\C^*(\g, \Cl_\hbar(V \oplus V^*))$ is Morita-trivial for the same reasons that the Clifford algebra is.
(Indeed, it's just a Clifford algebra in the appropriate category of $\C^*(\g)$-modules.)

Thanks to Lemma~\ref{lmm: one loop term}, however,
we can provide a more explicit description.
Consider the trivially charged fermion, where $\rho = 0$.
Then the global observables are the tensor product
\[
\C^*(\Omega^*(S^1) \otimes \g) \otimes \Obs^\q_{\rm ff}(S^1).
\]
(To be more precise, we should use the completed projective tensor product, 
but this issue does not affect the cohomology in this de Rham setting.)
The quantized interaction determines an isomorphism between the quantized observables for the charged fermion (using $\rho$) and the trivially charged fermion:
\[
\begin{array}{ccc}
\Obs^\q_{\rm cf}(S^1) & \xto{\cong} & \C^*(\Omega^*(S^1) \otimes \g) \otimes \Obs^\q_{\rm ff}(S^1) \\
f & \mapsto & e^{I^\q/\hbar} \,f 
\end{array}
\]
{\it A priori} that exponential term may seem problematic, as it could lead to infinitely many powers of $\hbar^{-1}$,
but it does not, as we are working with fermions.
Due to antisymmetry this exponential is a polynomial (which we compute explicitly in the proposition below).
Note as well that $\Delta I_0 = 0$ when $\g$ is semisimple, since there are no nonzero invariant linear functions on~$\g$.

We showed earlier that $\Obs^\q_{\rm ff}(S^1)$ admits a natural quasi-isomorphism onto $\CC[\hbar,\hbar^{-1}]$, 
so we can postcompose that map with multiplication against the exponentiated interaction to obtain a quasi-isomorphism
\[
\Obs^\q_{\rm cf}(S^1) \to \C^*(\Omega^*(S^1) \otimes \g).
\]
This induces the desired isomorphism on cohomology.
\end{proof}

In parallel with the free fermion, we can give an explicit description of the quasi-isomorphism,
and we will obtain a nice relationship with the character of the representation~$V$.

\subsubsection{The character as a partition function}

We assemble our arguments and constructions to obtain the following.

\begin{prp}
\label{prp: character}
Let $\g$ be a semisimple Lie algebra and $V$ a finite-dimensional $\g$-representation.
Under the quasi-isomorphism of Lemma \ref{lmm: iso to hoch},
the image of $1 \in \H^0\Obs^\q_{\rm cf}(S^1)$ is the class function
\[
\hbar^{-2\dim(V)}{\rm ch}_{\SS_\rho}(X) \in \csym(\g^*)^\g[\hbar^{\pm 1}],
\] 
which is, up to a power of $\hbar$, the character of the spinor representation $\SS_\rho$ of~$\g$.
\end{prp}

This result says that the charged fermion encodes the character of the Fock module $\SS_\rho$ of the Clifford algebra $\Cl(V \oplus V^*)$ as a $\g$-representation.
It is a purely algebraic consequence of the Morita trivialization given by the spinors.

It is compelling, however, that the Feynman diagrammatic computation matches, 
as we now show.
We will use the explicit isomorphism via the identification between the charged and trivially charged fermion. 
(See the proof of Lemma~\ref{lmm: iso to hoch}.)

We know that
\[
\exp(I^\q/\hbar) = \exp(I_{\rm tree}/\hbar) \exp(I_{\rm wheel}),
\]
but the wheel factor only depends on $X \d t$.
Hence, since we will take a Berezin integral of the whole term,
we need to identify the ``top fermion'' component of the tree factor.

By Lemma \ref{lmm: one loop term}, 
\[
\exp(I_{\rm tree}(v,\bar{v},X \d t)/\hbar) = \exp\left(\frac{1}{\hbar} \langle \bar{v}, \rho(X) v\rangle\right),
\]
whose ``top fermion'' component is $\hbar^{-2\dim V}\det(\rho(X))$.
We also know that
\[
\exp(I_{\rm wheel}(X \d t)) 
= \frac{1}{\det\left(e^{-\rho(X)/2}\Td(\rho(X))\right)}
= \frac{\det \left(e^{\rho(X)/2} - e^{-\rho(X)/2} \right)}{\det(\rho(X))}.
\]
Thus, 
\[
\exp(I^\q/\hbar)  = \hbar^{-2\dim V} \det \left(e^{\rho(X)/2} - e^{-\rho(X)/2} \right),
\]
as the factors of $\det \rho(X)$ cancel.

Here the hypothesis of semisimplicity enters:
there are no nonzero invariant linear functions of $\g$,
so 
\[
\det (e^{-\rho(X)/2}) = e^{-\tr \rho(X)/2}
\]
is the constant function~$1$.
Thus, the image of $1$ is precisely $\hbar^{-2\dim V}\det(1 - e^{-\rho(X)})$.

To recognize this function as the character,
we use the weight space decomposition of $V$ for the $\rho$-action: 
\[
V = \bigoplus_{\lambda \in \Lambda_\rho} \CC \lambda.
\] 
As a function on the maximal torus,
we see that the character of $\Sym(\Pi\rho)$~is
\[
\prod_{\lambda \in \Lambda_\rho} (1-e^{\lambda}).
\]
(Note that the negative sign arises because $\Pi V$ has odd parity.)
The representation $\SS_\rho$, however, is {\em not} $\Sym(\Pi\rho)$;
it is a tensor product of $\Sym(\Pi\rho)$ with a 1-dimensional representation,
determined by the ``vacuum vector'' (i.e., the component $\Sym^0(\Pi\rho)$),
on which the torus acts by~$\prod_{\lambda \in \Lambda_\rho} e^{\lambda/2}$.

The character of $\SS_\rho$ is thus
\[
\prod_{\lambda \in \Lambda_\rho} e^{\lambda/2} \cdot \prod_{\lambda \in \Lambda_\rho} (1-e^{\lambda}) = \prod_{\lambda \in \Lambda_\rho} (e^{\lambda/2}-e^{-\lambda}/2). 
\]
When restricted to the maximal torus, we see
\[
\det(e^{\rho(X)/2} - e^{-\rho(X)/2}) = \prod_{\lambda \in \Lambda_\rho} (e^{\lambda/2}-e^{-\lambda}/2). 
\]
The character formula agrees on the nose with the partition function, so we have obtained the claim.

%

\subsubsection{A remark on the bulk-boundary charged fermion}

Above we showed that the observables of the charged free fermion are equivalent to the dg associative algebra $\C^*(\g, \Cl_\hbar(V \oplus V^*))$,
and hence inherit the $\rho$-charged spinors as a module.
In other words,
we can extend $\Obs^\q_{\rm cf}$ to an $\cE_{0 \subset 1}$ algebra.

It is possible to construct this $\cE_{0 \subset 1}$ algebra directly via BV quantization of the charged fermion with boundary conditions,
in parallel with Section~\ref{sec: bb ff}.
The main change from the Feynman diagrammatics just described is that one must modify the propagator to take into the boundary condition,
and one must check this still yields a factorization algebra.
In his thesis, Eugene Rabinovich \cite{RabThesis} develops analogues for bulk-boundary systems of the structural theorems of \cite{CosBook} and \cite{CG2},
in particular, that BV quantization yields factorization algebras.

\section{A 1-dimensional fermionic defect for Chern--Simons theory}
\label{sec: coupled}

In this section we finally construct a Reshetikhin--Turaev theory 
by exhibiting a framed $\cE_{1 \subset 3}$-algebra 
whose $\cE_3$-algebra is the observables $\cA^\lambda$ of quantized Chern--Simons theory
and whose $\cE_1$-algebra is the observables of a quantized, $\g$-charged fermion.
In physical language, we couple a 3-dimensional gauge theory to a fermionic theory along a 1-dimensional defect.
In this way we demonstrate how perturbative Chern--Simons theory recovers the quantum group approach, so long as one works with $q$ formally expanded around the identity.

We begin this section by introducing the classical theory,
and then we turn to examining a simpler, intermediary case to our main goal:
we quantize the fermion but leave the bulk gauge theory unquantized,
which we call the half-classical theory.
In light of the preceding section, 
we see that we recover the Reshetikhin--Turaev theory for the $\cE_3$-algebra $\cA = \C^*(\g)$ and the perfect module $\cV = \C^*(\g, \Pi V)$.
Thus we ensure that our fully quantum theory will yield a deformation of the correct Reshetikhin--Turaev theory.

Finally, we explain how to produce the fully quantum theory.
Our main result is the following.

\begin{prp} 
\label{prp: defect main}
Let $\g$ be a semisimple Lie algebra $\g$, 
let $\rho: \g \to \End(V)$ be a finite-dimensional representation,
and let $\SS_\rho$ denote the associated spinor representation of $\g$ on $\Sym(\Pi V)$.
For each BV quantization $\lambda$ of Chern--Simons theory,
there is a contractible space of deformations of the algebra $\C^*(\g,\End(\SS_\rho))$ and module $\C^*_{\Lie}(\g, \SS_\rho)$ to an $\cE_{0 \subset 1}$-algebra $\cA^\lambda_\rho$ in left $\cA^\lambda$-modules.
It arises as the quantization of the charged fermion coupled to Chern--Simons theory.

Moreover, the quantization of $\C^*_{\Lie}(\g, \SS_\rho)$ determines an $(n,1)$-Morita equivalence between $(\cA^\lambda, \cA^\lambda)$ and $(\cA^\lambda, \cA^\lambda_\rho)$.
In consequence, each choice of knot $L : S^1 \hookrightarrow \RR^3$ determines a distinguished element in $\cA^\lambda$ by the canonical isomorphism to $\cA^\lambda$ from the factorization homology of the constructible factorization algebra with $\cA^\lambda$ on the knot complement and the fermionic observables $\cA^\lambda_\rho$ along the knot.
\end{prp}

The first part of this proposition is Proposition~\ref{prp: existence of defect quants} and Corollary~\ref{crl: defect fact alg}, 
while the second part is the consequence of Section~\ref{sec: tangles}.
Recall the discussion around Lemma~\ref{lmm: perfect is n1 morita} for the definition and properties of $(n,1)$-Morita maps.

\subsection{The classical theory for the defect}

Let $M$ be an oriented 3-manifold (the ``bulk'') with an embedding $i: L \hookrightarrow M$ of an oriented 1-manifold (the ``defect'').
We now consider the minimal coupling between the fermions and the gauge fields:
\begin{equation}
\label{coupled action}
S_{\rm coup}(A,\psi,\psib) = CS(A) + S_{\rm ff}(\psi,\psib) + \int_L ( \psib, \rho(i^*A) \psi ).
\end{equation}
Note that this action functional combines Chern--Simons theory in the bulk 
with a charged fermion along the defect.
The equations of motion have two parts.
Along the embedded 1-manifold, we have
\[
\d \psi + \rho(i^*A) \psi = 0 \quad\text{and}\quad \d \psib + \bar{\rho}(i^*A) \psib = 0,
\]
which say that $\psi$ and $\psib$ are horizontal for the connection induced from $M$.
On $M$ itself we get a modified Maurer--Cartan equation
\[
\d A + \frac{1}{2}[A,A] = j(\psi,\psib),
\]
where $j$ is the current generated by the fermions and is defined by the requirement that
\[
\langle \tilde{A},j(\psi,\psib) \rangle = (\psib, \rho(i^*\tilde{A}) \psi)
\]
for every gauge field $\tilde{A}$. (We add the tilde merely to emphasize that $\tilde{A}$ is arbitrary and not necessarily a connection form satisfying the equations of motion.)
Note that the current $j$ is supported along~$L$ and hence is manifestly a distribution (or a current in de Rham's sense).

\subsection{Quantizing the fermion but leaving the gauge fields classical}

The main reason for doing this case is that it will let us see in concrete terms
that we recover a Reshetikhin--Turaev tangle theory for $\C^*(\g)$ as an $\cE_3$-algebra.
In other words, we will see how our prescription behaves in the case of classical Chern--Simons theory.

We have already done the essential work by quantizing the charged fermion on a 1-manifold.
Let us call this coupled system the {\em half-classical theory},
as we treat the fermions quantum mechanically but the gauge fields as classical. 
We denote its factorization algebra of observables by~$\Obs^{1/2}$.
Our work in Sections~\ref{sec: CS} and~\ref{sec: ff}, notably Lemma~\ref{lmm: g-eq clifford alg}, shows the following.

\begin{lmm}
For an embedded 1-manifold $L$, possibly with boundary, in $\RR^3$, 
the observables $\Obs^{1/2}$ are equivalent to 
the constructible factorization algebra arising by putting
\begin{itemize}
\item $\C^*(\g)$ on the bulk~$\RR^3 \setminus L$,
\item $\C^*(\g,\Cl_\hbar(V \oplus V^*))$ on $L \setminus \partial L$, and
\item $\C^*(\g,\SS_\rho)$ on each boundary point~$x \in \partial L$.
\end{itemize}
\end{lmm}

This result confirms concretely the viewpoint on defects sketched in Section~\ref{sec: wilson loop and fermions} of the introduction.
For instance, by Proposition~\ref{prp: character},
given a knot in $\RR^3$, 
this lemma shows that the image of the identity in the Clifford algebra (which is endomorphisms of the spinor bundle) recovers the character of spinor representation.

\begin{rmk}
Note as well that no framing of the 1-manifold is needed.
This feature is equivalent to the fact that the construction is manifestly isotopy-invariant.
\end{rmk}

\subsection{The deformation complex for a theory with defects}
\label{subsec: def complex for defects}

In Section \ref{subsec: def cplx for CS} we described the deformation complex for Chern--Simons theory,
which allowed us to see that there were no obstructions to quantization and to identify a choice of level as the only meaningful freedom amongst such quantizations.
This cohomological analysis allows us to see that the desired quantization exists and what freedom is allowed as we construct it---without doing any diagrammatic computations.
We undertake something analogous for the coupled system.

\subsubsection{Long lines and loops}

Let $M$ be a 3-manifold and let $i : L \hookrightarrow M$ be a proper embedding of a 1-manifold without boundary into~$M$.
We want to consider functionals of the form
\[
F = F_{\rm bulk}(A) + F_{\rm defect}(A,\psi,\psib),
\]
where $F_{\rm bulk}$ is a local functional on the 3-manifold $M$ that only depends on the bulk field $A$ 
and where $F_{\rm defect}$ is a local functional on the 1-manifold $L$ that depends on the fields $\psi$ and $\psib$ and also on $A$ {\em but only through its behavior along $L$}.
That is, we allow $F_{\rm defect}$ to depend only on the pullback $i^*A$ of $A$ to $L$.
(This restriction is unnecessary: see the discussion in Section~\ref{rmk on normal directions}.)
As we also wish to keep Chern--Simons theory (and its quantizations) as the bulk theory,
we are only interested in the space of defect functionals.
Hence, we are reduced to studying the deformation complex $\Def_{\rm cf}$ of the charged fermion on~$L$.

Equivalently, we phrase the problem as follows.
Consider the formal moduli space $\Def_{\rm coupled}$ of $\cE_{1 \subset 3}$-algebra structures on the observables of classical Chern--Simons theory coupled to a charged fermion along a properly embedded line.
There is a forgetful map
\[
\Def_{\rm coupled} \to \Def_{CS}
\]
to the formal moduli space of $\cE_3$ algebra structures on observables for classical Chern--Simons theory (i.e., $\C^*(\fg)$).
We are interested in the fiber of this map.

As we are only interested in what happens locally, 
we take $L$ to be a copy of $\RR$ embedded into $M = \RR^3$ such that the ends of $L$ are at infinity.
(Equivalently, think of a 3-ball skewered on a line.)

\begin{lmm}
\label{lmm: coupled def cplx}
For $\g$ a finite-dimensional Lie algebra and $V$ a finite-dimensional representation,
the zeroth cohomology group of the charged fermion deformation complex is
\[
H^0(\Def_{\rm cf}) = H^1(\g,\Sym^{>0}(\Pi V \oplus \Pi V^*)_\rho ),
\]
and the first cohomology group is
\[
H^1(\Def_{\rm cf}) = H^2(\g,\Sym^{>0}(\Pi V \oplus \Pi V^*)_\rho).
\]
\end{lmm}


These descriptions admit simplifications for semisimple Lie algebras,
by standard vanishing results in Lie algebra cohomology.

\begin{crl}
For $\g$ semisimple,
\begin{align*}
H^1(\Def_{\rm cf}) 
&= H^2(\g,\Sym^{>0}(\Pi V \oplus \Pi V^*)_\rho)\\ 
&\cong \H^2(\g) \otimes \left(\Sym^{>0}(\Pi V \oplus \Pi V^*)_\rho\right)^\g \\
&= 0
\end{align*}
and
\begin{align*}
H^0(\Def_{\rm cf}) 
&= H^1(\g,\Sym^{>0}(\Pi V \oplus \Pi V^*)_\rho ) \\
&\cong \H^1(\g) \otimes \left(\Sym^{>0}(\Pi V \oplus \Pi V^*)_\rho \right)^\g \\
&= 0
\end{align*}
by Whitehead's lemmas.
\end{crl}

Thus, there are no interesting deformations of the defect functionals:
the only interesting deformation of the coupled theory is by the level,
which only affects the bulk theory.
Deformations of the coupling can be trivialized by automorphisms of the fields.

\begin{proof}[Proof of Lemma \ref{lmm: coupled def cplx}]
This argument is parallel to that for pure Chern--Simons theory,
discussed in Section~\ref{subsec: def cplx for CS}.
The charged fermion corresponds to a sheaf on~$L$ of dg Lie algebras 
\[
\cL_{\rm cf} = \Omega^*_L \otimes (\g \ltimes \left(\Pi(V \oplus V^*)[-1]\right)),
\]
which is quasi-isomorphic to the locally constant sheaf $\g \ltimes (\Pi(V \oplus V^*)[-1])$.
Hence, as dg Lie algebras in the category of $D$-modules, we have a quasi-isomorphism
\[
\begin{tikzcd}
C^\infty_L \otimes (\g \ltimes (\Pi(V \oplus V^*)[-1])) \arrow[r, hook, "\simeq"] & \cJ \cL_{\rm cf},
\end{tikzcd}
\]
where $\cJ\cL_{\rm cf}$ denotes the jets of this sheaf along~$L$.
Hence the full deformation complex is quasi-isomorphic to 
\[
\Omega^*_L[1] \otimes \C^{\geq 1}(\g \ltimes (\Pi(V \oplus V^*)[-1])),
\]
the {\em reduced} Lie algebra cochains.
We only care about deformations of the defect functionals, 
and hence require our functionals to depend on the fields $\psi$ and $\psib$.
Observe that
\[
\C^{\geq 1}(\g \ltimes (\Pi(V \oplus V^*)[-1])) = \bigoplus_{m+n > 0} \C^m_{\Lie}(\g, \Sym^n(\Pi (V \oplus V^*)))
\]
and the defect functionals thus involve the subcomplex
\[
\bigoplus_{n > 0} \C^*_{\Lie}(\g, \Sym^n(\Pi (V \oplus V^*))) = \C^*(\g, \Sym^{> 0}(\Pi(V \oplus V^*))).
\]
Thus the deformation complex of defect functionals is quasi-isomorphic to
\[
\Omega^*_L[1] \otimes \C^*(\g, \Sym^{> 0}(\Pi(V \oplus V^*)[-1]))
\]
and so its cohomology is
\[
\H^*(\g, \Sym^{>0}(\Pi (V \oplus V^*))[1]
\]
for $L \cong \RR$.
\end{proof}

\subsubsection{On normal directions}
\label{rmk on normal directions}

One might wish to allow the coupling term to depend on the gauge field in the normal directions.
For instance, we might want to include an interaction term like 
\[
\int_L (\psib, i^*( \cL_Z A) \psi) 
\]
where $Z$ is a vector field on $\RR^3$ and $\cL_Z$ denotes the Lie derivative.
In other words, we want to write down local functionals along $L$ that depend on the gauge field's Taylor expansion in the directions normal to~$L$.

We examine this situation where we identify $L$ with the $z$-axis in $\RR^3$ with coordinates $(x,y,z)$.
Consider then the completion of the dg Lie algebra $\g^M$ (recall Definition~\ref{dfn:g^M}) along~$L$:
\[
\g^\wedge_L = \Omega^*(L) \otimes \CC[x,y,\d x, \d y] \otimes \g
\]
with differential 
\[
(d_L + \partial_x \d x + \partial_y \d y) \otimes \id_\g.
\]
Note that the middle term in the tensor product encodes the de Rham complex on the 2-dimensional formal disk, 
as we are completing in the normal directions.
This dg Lie algebra knows about the gauge field along $L$ as well as variations in the normal directions.

We can then modify the proof of Lemma \ref{lmm: coupled def cplx} by replacing
\[
\cL_{\rm cf} = \Omega^*_L \otimes (\g \ltimes \Pi(V \oplus V^*)[-1])
\]
with
\[
\cL^\wedge_{\rm cf} = \Omega^*_L \otimes \left( (\g \otimes \CC[x,y,\d x, \d y]) \ltimes \Pi(V \oplus V^*)[-1]\right).
\]
We could then compute the local Lie algebra cochains along $L$ for this sheaf of dg Lie algebras,
which would encode coupling terms that depend on normal directions.

Note, however, that there is a canonical inclusion
\[
\cL_{\rm cf} \hookrightarrow \cL^\wedge_{\rm cf}
\]
of sheaves of dg Lie algebras,
extending the inclusion 
\[
\CC \hookrightarrow ( \CC[x,y,\d x, \d y] , \d_{dR}).
\]
This inclusion is a quasi-isomorphism of differential complexes.
Hence the local Lie algebra cochains are quasi-isomorphic as well,
so our work above suffices.

\subsubsection{One-manifolds with boundary}

In Section~2.3 of~\cite{RabThesis},
Rabinovich develops the deformation complex for a classical BV theory with boundary conditions,
such as the charged free fermion.
The results are direct analogs of those in~\cite{CosBook}.
We now apply that formalism to obtain an analog of Lemma~\ref{lmm: coupled def cplx}.

We are now interested in proper embeddings into $\RR^3$ of 1-manifolds with boundary.
As we are only interested in what happens locally, we take $L$ to be a copy of the half-line~$\HH^1 = [0,\infty)$ linearly embedded with its boundary at the origin.
(Picture it as a ball impaled on a stick.)
We are interested in deformations of the defect functionals $F_{\text{defect}}$ but with a boundary condition imposed on the fermionic fields.
For simplicity, we impose the condition~$\psi|_{\partial L} = 0$.
In consequence, we are studying the deformation complex $\Def^\partial_{{\rm cf}}$ of a charged fermion on $\HH^1$ with this boundary condition.

Equivalently, we phrase the problem as follows.
Consider the formal moduli space $\Def_{\rm coupled}^\partial$ of $\cE_{0 \subset 1 \subset 3}$-algebra structures on the observables of classical Chern--Simons theory coupled to a charged fermion with a standard boundary condition.
There is a forgetful map
\[
\Def_{\rm coupled}^\partial \to \Def_{CS}
\]
to the formal moduli space of $\cE_3$ algebra structures on observables for classical Chern--Simons theory (i.e., $\C^*(\fg)$).
We are interested in the fiber of this map.

\begin{lmm}
\label{lmm: bdry def cplx}
Let $\g$ be a finite-dimensional Lie algebra and $(V,\rho)$ a finite-dimensional representation.
Let
\[
I = \Sym(\Pi V) \otimes \Sym^{>0}(\Pi V^*)
\]
the ideal generated by $\Pi V^*$ inside the algebra
\[
\Sym(\Pi V \oplus \Pi V^*),
\]
equipped with the $\g$-action inherited from~$\rho$.
The zeroth cohomology group of the deformation complex for the charged fermion is isomorphic to
\[
H^0(\Def^\partial_{{\rm cf}}) = H^1(\g,I_\rho),
\]
and the first cohomology group is
\[
H^1(\Def^\partial_{{\rm cf}}) = H^2(\g,I_\rho).
\]
\end{lmm}

We then immediately deduce the following,
which guarantees there are no obstructions to quantizing Chern--Simons theory on $\RR^3$ coupled to a charged fermion on an embedded half-line~$\HH^1$ with standard boundary conditions.

\begin{crl}
For $\g$ semisimple,
\[
H^1(\Def^\partial_{{\rm cf}}) 
\cong \H^2(\g) \otimes I_\rho^\g = 0
\]
and
\[
H^0(\Def^\partial_{\rm cf}) 
\cong \H^1(\g) \otimes I_\rho^\g = 0
\]
by Whitehead's lemmas.
\end{crl}

The free fermion is a special case of the topological BF theory analyzed by Rabinovich at the end of Section~2.3.
To obtain this lemma, we simply include the background fields coming from the ambient gauge theory.

Before undertaking the proof, 
we thus review those results.

The free fermion theory for the vector space $V$ is an example of the topological BF theory where the Lie algebra is
\[
\g = (\Pi V)[-1],
\]
i.e., a $\ZZ \times \ZZ/2$-graded Lie algebra that is purely odd and concentrated in cohomological degree~1.
We unwind this assertion in detail, 
so we can invoke Rabinovich's computations.

The fields of BF theory on the real line are
\[
\Omega^*(\RR) \otimes \fg[1] \oplus \Omega^*(\RR) \otimes \fg^*[-1]
\]
following Example~2.2.11 of~\cite{RabThesis}.
Explicitly, that means the fields are
\[
\Omega^*(\RR) \otimes \Pi V \oplus \Omega^*(\RR) \otimes \Pi V^* \cong 
\Omega^*(\RR) \otimes (\Pi V \oplus \Pi V^*),
\]
which are precisely the free fermion fields.
Rabinovich calls fields in the first summand $A$-fields and fields in the second summand $B$-fields,
since the action functional is
\[
S_{BF}(A,B) = \int B \wedge \d A.
\]
In our case, $A$ is the $\psi$ field, and $B$ is the $\psib$,
so we recover the free fermion action
\[
\int \psib \,\d \psi.
\]
In Example~2.2.23 Rabinovich introduces two boundary conditions on the half-line~$\HH^1 = [0,\infty)$: 
\begin{itemize}
\item the A-boundary condition, where $B|_0 = 0$ so that only the $A$-field is interesting on the boundary, and
\item the B-boundary condition, where $A|_0 = 0$.
\end{itemize}
The answers below will be quasi-isomorphic,
and any choice of Lagrangian subspace $W \subset \Pi(V \oplus V^*)$ would yield an equivalent boundary condition.

Equation~(2.3.89) tells us that the deformation complex for the $A$-boundary condition is modeled by equation~(2.3.86) with $M = \HH^1$ and $\partial M = 0$.
In particular, we find $\Def^{\partial,A}_{\rm ff}$ is modeled by the 2-term cochain complex
\begin{equation}
\label{eqn: ff def cplx}
\Sym^{>0}(\Pi V \oplus \Pi V^*) \to \Sym^{>0}(\Pi V^*)
\end{equation}
where the second term sits in cohomological degree~0.
The map is the ring map sending $\Pi V$ to zero (i.e., the quotient by the ideal generated by $\Pi V$), 
so it is surjective.
Thus this complex has trivial cohomology in every degree except~$-1$.
Its degree~$-1$ cohomology is
\begin{equation}
\label{eqn: ff def H}
I_A = \Sym(\Pi V) \otimes \Sym^{>0}(\Pi V^*),
\end{equation}
the ideal generated by $\Pi V^*$ but placed in this degree.
Note that $I_A = I$, where $I$ is defined in the statement of the lemma.

The $B$-boundary condition has deformation complex $\Def^{\partial,B}_{\rm ff}$
modeled by the cochain complex
\[
\Sym^{>0}(\Pi V \oplus \Pi V^*) \to \Sym^{>0}(\Pi V)
\]
where the second term sits in cohomological degree~0,
by Equation~(2.3.90).
It likewise has cohomology vanishing outside degree~$-1$, and its nontrivial cohomology group is 
\[
I_B = \Sym^{>0}(\Pi V) \otimes \Sym(\Pi V^*)
\]
the ideal generated by~$\Pi V$.

\begin{proof}
The deformation complexes of local functionals are typically very large (as they involve local differential forms, including Lagrangian top forms),
but their cohomology can be quite small, as seen with Chern--Simons theory.
For the topological BF theories,
Rabinovich constructed explicit quasi-isomorphisms from his ``small models'' to the complexes of local functionals in his Equations~(2.3.89) and~(2.3.91).
For the free fermion, we just saw that the A- and B-conditions lead to isomorphic small models,
so we will work with the A-condition.
Recall Equation~\eqref{eqn: ff def cplx}:
\[
\Def^{\partial,A}_{\rm ff} = (\cdots \to 0 \to \Sym^{>0}(\Pi V \oplus \Pi V^*) \to \Sym^{>0}(\Pi V^*) \to 0 \to \cdots)
\]
concentrated in degrees~$-1$ and~0.
Recall that its cohomology is $I_A$, as defined in Equation~\eqref{eqn: ff def H}.

Let $\g$ denote the Lie algebra in which the background gauge fields are valued (i.e., it does not mean $\Pi V[-1]$ below).
For the charged free fermion on~$\HH^1$,
the corresponding deformation complex~$\Def^{\partial,A}_{\rm cf}$ is modeled by
\[
\C^*(\g, \Def^{\partial,A}_{\rm ff}).
\]
We justify this assertion below, 
after proving the claim in the lemma.

We can view $\Def^{\partial,A}_{\rm cf}$ as the totalization of a double complex, 
with the vertical direction arising from $\g$ and the horizontal direction from $\Def^{\partial,A}_{\rm ff}$.
This double complex is nonzero only in two vertical columns,
siting in the 3rd quadrant.
Consider the spectral sequence that uses the horizontal differential of $\Def^{\partial,A}_{\rm ff}$ first and then the vertical differential for Lie algebra cohomology.
The $E_2$ page is 
\[
\H^*(\g, I_A[1]) = \H^*(\g, I_\rho)[1].
\]
The spectral sequence manifestly collapses on the $E_2$ page,
so
\[
\H^k(\Def^{\partial,A}_{\rm cf}) = \H^{k+1}(\g, I_\rho)
\]
as claimed.

We now justify our model for the deformation complex;
our argument is parallel to that for Lemma~\ref{lmm: coupled def cplx}.
The fields of the charged fermion are modeled by a direct sum of the $\g$-valued de Rham complex on $\RR^3$ and the free fermion fields on $\HH^1$ with A-boundary condition.
These form a dg Lie algebra, 
where the fermions are modules for $\Omega^*(\RR^3) \otimes \g$ by restricting forms to $\HH^1$ and using $\g$ action on $\Pi V \oplus \Pi V^*$.
The Poincar\'e lemma provides a quasi-isomorphism
\[
\g \to \Omega^*(\RR^3) \otimes \g
\]
of dg Lie algebras, 
and thus we work instead with a $\C^*(\g)$-family of fermionic field theories on $\HH^1$
(i.e., we use $\C^*(\g)$ as the base dg algebra for all computations).
We then base-change the quasi-isomorphism of Rabinovich's Equation~(2.3.89) to obtain our model of the deformation complex.
\end{proof}

\subsection{Existence of quantization}

This computation of the deformation complex is quite helpful, 
as it has the following striking consequence.

\begin{prp} 
\label{prp: existence of defect quants}
For each BV quantization of Chern--Simons theory and for each choice of finite-dimensional representation $V$ of a semisimple Lie algebra $\g$,
there is a contractible space of quantizations of the coupled system.
\end{prp}

This claim is not an immediately corollary of our computation above 
because we cannot invoke directly the results of~\cite{CosBook} or~\cite{RabThesis},
which do not treat theories with defects.
Nonetheless, those arguments will be modified to work in our context.\footnote{Our argument is out of textual order, as it depends on Section~\ref{subsec: defect diagrams}. 
That discussion is quite simple to summarize---renormalization is done by the configuration space method---but the construction is rather long-winded, on first acquaintance. 
Hence we wish to emphasize the structure of our approach by discussing the deformation-theoretic aspects before the analytic aspects.}

\begin{proof}
The BV quantization of a classical theory has essentially two stages:
\begin{itemize}
\item constructing the perturbative expansion with Feynman diagrams and removing divergences via regularization and renormalization, and 
\item solving the quantum master equation (QME).
\end{itemize}
This same two-step process applies to our coupled theory.

In setting up the perturbative expansion, 
we combine the ingredients from the bulk and defect theories in a direct way.
The space of fields is a product 
\[
\cE_{\rm coupled} = \cE_{CS} \times \cE_{\rm ferm}
\]
of the bulk fields and the fermionic fields (with their standard boundary condition imposed).
The underlying free theory is also a product: 
an abelian Chern--Simons theory in the bulk $\RR^3$ and a free fermion along the properly embedded 1-manifold~$L$.
The natural gauge fix simply uses the gauge fix already provided separately on each component.
Hence, we obtain a propagator for the coupled system, 
which is the product (in the categorical sense) of the propagator for each component.
We thus define the renormalization group (RG) flow as the ``product'' of the RG flows for the two free theories. 
(See Chapter~4 of \cite{RabThesis} for a thorough discussion of topological BF theory on a half-line,
which includes the one-dimensional fermion as a special case.)

In Section~\ref{subsec: defect diagrams} we show that the configuration space method for renormalization applies to this coupled theory,
and so it produces a scale-dependent family of functionals that satisfy the RG flow.\footnote{This claim is very closely related to the theory of Bott--Taubes integrals or Jacobi diagrams, as is seen in the Appendix. We only use these tools to renormalize; we do not analyze anomalous faces to solve the QME.}
See Definition~\ref{dfn: ren defect action} for this renormalized action.
(In the language of \cite{CosBook}, it provides a pre-theory.)
The configuration space method shows  that no counterterms are needed:
there are no ultraviolet singularities for this coupled theory,
just as Chern--Simons theory does not have any ultraviolet singularities.
Thus we are able to take the $ t\to 0$ limit of the RG flow, in the style of~\cite{CosBook}.\footnote{This observation applies not only to this theory, but to any topological field theory of Chern--Simons or BF type that has a topological defect. 
For instance, we could construct topological surface defects for Chern--Simons theory by such methods.}\footnote{Note that we do {\em not} assert a version of Theorem~A of \cite{CosBook} for {\em any} theory coupled to a defect theory;
we only assert that {\em this} coupled theory determines a pre-theory.}

Having handled the perturbative expansion, we turn to the QME.
Recall that the failure to satisfy the QME is itself represented by a local functional,
a longstanding idea in the BV formalism.
A precise incarnation appears as Corollary~11.1.2 of Chapter 5 of~\cite{CosBook}.  
In more detail, given an effective field theory $\{S[t]\}_{t > 0}$ consisting of a family of effective action functionals,
the putative scale $t$ quantum differential 
\[
\d_t = \{S[t], -\}_t + \hbar \Delta_t
\]
may fail to square to zero,
but these obstruction terms $\{Ob[t]\}_{t > 0}$, satisfying
\[
\d_t^2 = \{Ob[t], - \}_t,
\]
satisfy the RG flow. 
Their $t \to 0$ limit of $Ob[t]$ has an asymptotic expansion as a local functional,
just like the effective actions $\{S[t]\}$ themselves.
The proof is a direct computation from the very definition of an effective field theory.

For our coupled theory, these definitions and arguments still make sense.
Our RG flow provides exactly the ingredients needed to run the argument of Lemma 11.1.1 in Chapter 5 of~\cite{CosBook} (and its analogue, Proposition~4.6.3 of~\cite{RabThesis}).
The obstruction to satisfying the QME at a given scale $t$ satisfies the classical (i.e., tree-level) RG flow,
and it is annihilated by bracketing with the classical action functional at scale~$t$.
In the $t \to 0$ limit, the obstruction becomes local.
As we know that pure Chern--Simons theory satisfies the QME,
the obstruction must involve vertices labeled by the coupling term and hence must be supported on the defect~$L$.
It is thus a degree one cocycle in~$\Def_{\rm cf}$, 
the deformation complex describing couplings between the fermionic and the bulk theory.

We now finish the argument by invoking our cohomological analysis of the deformation complex.
We know that the first cohomology group is trivial, so the obstruction is trivializable.
Hence a quantization exists.
Moreover, we know that the zeroth cohomology group for the coupling term vanishes,
so any choice of trivialization of the obstruction is cohomologous to any other.
Thus we obtain the contractibility of the space of quantizations.
\end{proof}


By obstruction-theoretic arguments, 
we now know that there is an essentially unique way to quantize this coupled system.
Since this quantization also induces, by the main theorem of~\cite{CG2},
a constructible factorization algebra, 
we have the following consequence.

\begin{crl}
\label{crl: defect fact alg}
For each BV quantization $\lambda$ of Chern--Simons theory and for each choice of finite-dimensional representation $V$ of a semisimple Lie algebra $\g$,
there is a contractible space of deformations of the framed $\cE_{0 \subset 1 \subset 3}$-algebra $\cA^\cl_V$
to a framed $\cE_{0 \subset 1 \subset 3}$-algebra~$\cA^\lambda_V$.
\end{crl}

Theorem~\ref{thm: central} of this paper follows from this result,
by the arguments in the introduction.

\subsubsection{Remarks on integrals of Bott-Taubes type}

In a sense, this result also gives a structural explanation for a fact discovered independently by Altsch\"uler--Freidel \cite{AltFreVCS} and Thurston \cite{Thu}.
Bott--Taubes \cite{BotTau} had seen that the Feynman diagrams that arise by naively quantizing the classical Wilson loop observable for a knot $K$ were {\em not} isotopy-invariant,
but for the first few graphs, this failure was modest, essentially just the usual issue of self-linking.
What the subsequent work found was that one could correct {\em all} the diagrams by a simple factor depending only on a choice of framing (and hence resolving the self-linking issue).
In the BV formalism, we are seeing that the classical Wilson loop observable (a cocycle in the classical observables) lifts to a quantum cocycle.
The obstruction theory gives some insight here: 
the Bott--Taubes integrals are essentially the construction of a theory with a defect along the knot,
and we have seen, at least for our choice of such theory, that there is essentially a unique quantization.

The reader might ask where the framing dependence comes in.
It is not explicitly visible in the obstruction theory, 
but it is implicit in the general formalism, as follows.
We know that quantization of the coupled theory produces a framed $\cE_{1 \subset 3}$-algebra,
and these explicitly involve a choice of framing of the embedded 1-manifold.
It is a further condition on the relationship between the $3$- and $1$-disk algebras to make it framing-independent.\footnote{For an $\cE_3$-algebra $\cA$ and a $\cE_1$-algebra $\cB$ to form a framed $\cE_{1 \subset 3}$-algebra, we need a $\cE_2$-algebra map $\int_{S^1} \cA \to {\rm HH}^*(\cB)$.
To eliminate the dependence on framing of embedded 1-manifolds, we need this $\cE_2$-algebra map to be a (homotopy) $O(2)$-fixed point.
This is further data, and not always possible.
Hence the {\em generic} situation is that there should be framing dependence.}

\appendix
\section{Explicit quantizations via the configuration space method}

This appendix is devoted to reviewing how to construct the perturbative field theories discussed in the paper.
In the first section we consider pure Chern--Simons theory (i.e., without the fermionic defect),
and in the next section we explain how to incorporate a certain class of defects.
The techniques go by the name of the {\it configuration space method} as the key idea is that the singularities that appear in the Feynman diagrams can be resolved by a point-splitting regularization,
and organizing this process systematically involves compactifying configuration spaces intelligently.

\subsection{The BV quantization of pure Chern--Simons theory}
\label{pure CS quantization}

In this section we review the perturbative quantization of Chern--Simons theory,
following Axelrod--Singer \cite{AxeSinI,AxeSinII}, Kontsevich \cite{KonECM}, and Bar-Natan \cite{BarCS},
but as phrased in the Batalin--Vilkovisky formalism.
Before jumping into our theory of interest,
we briefly discuss the perturbative BV quantization procedure in general outline.
(Sawon \cite{SawCS} provides a gentle expository introduction to this material.)

\subsubsection{Perturbative quantization}

Perturbative quantization aims to approximate a putative functional integral
\[
\int_{\phi \in \cF} O(\phi) e^{-S(\phi)/\hbar} \cD \phi,
\]
where $\cF$ is a space of fields equipped with a measure $e^{-S(\phi)/\hbar} \cD \phi$,
and $O$ is a function on fields (i.e., an observable) integrable against this measure.
This functional integral is rarely well-defined, 
but the procedure for computing the putative approximation does determine a well-behaved mathematical construction.
We will simply work with that process as a stand-alone notion.

It will be helpful to recall the model case, 
as an excuse to introduce notation before we introduce the aspects of functional analysis.

Let $V$ be a finite-dimensional vector space;
we will write $v$ for an element of $V$.
Let $S$ be a formal power series in the dual space $V^\vee$.
In other words, we are thinking of $S$ as a function on $V$ viewed as a formal manifold.
We impose the following conditions on $S$:
the origin is an isolated critical point of $S$ and $S(0) = 0$.
Then
\[
S(v) = Q(v) + I(v),
\]
with $Q$ a nondegenerate quadratic function 
and the ``interaction'' $I$ a power series that involves cubic and higher order terms.
Hence we come to consider the formal expression of interest:
\[
\int_{v \in V} e^{-S(v)/\hbar} \d v = \int_{v \in V} e^{-Q(v)/\hbar} e^{-I(v)/\hbar} \d v.
\]
One should view the integral on the right hand side as computing the expected value of $e^{-I(v)/\hbar}$ against the unnormalized Gaussian measure $e^{-Q(v)/\hbar}\d v$.

There is a combinatorial expression for this integral,
due to the fact that the moments of a Gaussian measure are easily described in terms of derivatives of the ``inverse'' $Q^{-1}$.
(As $Q \in \Sym^2(V^\vee)$ is nondegenerate, 
it determines an isomorphism $V \cong V^\vee$ and hence an inverse element $Q^{-1} \in \Sym^2(V)$.)
Hence one takes the power series expansion of $e^{-I(v)/\hbar}$
and computes the moments of each term.
Although the full sum is {\em a priori} infinite and ill-defined,
it is a well-defined formal power series in $\hbar$,
since each $\hbar^k$ only involves finitely many moments.
In the end one obtains an equality
\[
\int_{v \in V} e^{-Q(v)/\hbar} e^{-I(v)/\hbar} \d v = \sum_{\gamma \in {\rm Graphs}} \frac{\hbar^{b_1(\gamma})}{|\Aut(\gamma)|} w_\gamma(Q^{-1}, I),
\]
where $\gamma$ denotes a finite graph,
$b_1(\gamma)$ denotes its loop number (or first Betti number),
and $w_\gamma(Q^{-1},I)$ denotes a polynomial function on $V$ constructed from the graph, the matrix $Q^{-1}$, and the interaction term~$I$.

Extensive details can be found in Chapter 2 of \cite{CosBook},
but in short we have the following:
\begin{itemize}
\item $\gamma$ is a {\em stable} graph, meaning that it has finitely many vertices (each of valence at least three), finitely many edges, and finitely many legs or tails (which have an end lying on a vertex and ``free end'');
\item an automorphism of a stable graph permutes the sets of vertices, edges, and legs but preserves their relationships (e.g., if it moves an edge, it moves the edge's boundary vertices in the same way);
\item for a graph $\gamma$ with $k$ external legs, 
the weight $w_\gamma(Q^{-1},I)$ is an element of $\Sym^k(V^\vee)$ built by labeling each edge with a {\em propagator} $Q^{-1} \in \Sym^2(V)$ and labeling each $n$-valent vertex with the degree $n$ component $I_n \in \Sym^n(V^\vee)$ of the interaction $I$ and then contracting these tensors according to the shape of the graph.
\end{itemize}
We use the notation 
\[
W(P,I) = \sum_{\gamma \in {\rm ConnGraphs}} \frac{\hbar^{b_1(\gamma})}{|\Aut(\gamma)|} w_\gamma(P, I)
\]
for the corresponding sum over {\em connected} stable graphs 
and $P \in \Sym^2(V)$ an arbitrary element.
We thus know that for $P = Q^{-1}$, the expression $\exp W(P,I)$ recovers the formal integral of $\exp(-S/\hbar)$ over~$V$.

We now use this expansion in the context of Chern--Simons theory.

\subsubsection{The combinatorics and Lie theory for Chern--Simons theory}

In the case of Chern--Simons theory, 
the quadratic term is $Q(A) = (1/2)\langle A, \d A \rangle$ 
and the higher order term is simply the cubic expression $I(A) = (1/6)\langle A,[A,A]\rangle$.
We defer to the next subsection the issue of understanding $Q^{-1}$,
which requires us to understand what $\d^{-1}$ means, 
an issue analytic in nature.
Here we simply focus on the graphs that appear.

The first observation is that only trivalent graphs are relevant,
as we only have a cubic term.

Second, since our space of fields is a tensor product $\Omega^*(\RR^3) \otimes \g$,
all the atomic components of our diagrammatics, namely the vertices and edges, 
factor as tensor products as well.
Thus, for any graph $\gamma$, we have
\[
w_\gamma(P,I) = w_\gamma^{an} w_\gamma^{\g},
\]
where the Lie-theoretic weight is independent of the analysis
and the analytic factor is the same for any Lie algebra~$\g$.

The Lie-theoretic factors and their relationships impose strong and useful relations,
such as the well-known AS and IHX relations (arising from the antisymmetry and Jacobi relation of the Lie brackets). 
These rules allow one to discover nontrivial relations between the Lie-theoretic weights of graphs.

\begin{rmk}
One can abstract this situation further and consider {\em weight systems} in place of Lie algebras with a nondegenerate pairing,
but we will not need that level of generality.
\end{rmk}

\subsubsection{The analysis for Chern--Simons theory}

The other tensor factor of our fields in $\Omega^*(\RR^3)$,
and for any Feynman diagram, 
the analytic factor is supposed to be determined by knowing the moments for the Gaussian measure 
\[
e^{-\int_M \alpha \wedge \d\alpha} \cD \alpha,
\]
where $\alpha$ denotes an element of $\Omega^*(\RR^3)$.
A formal application of the formula for moments suggests we need to know $\d^{-1}$,
which would decorate the edges in a graph expansion.
This element is well-defined using the theory of distributions, 
but in trying to then compute the weight of a diagram, 
one sees that one is trying to contract a distribution $\d^{-1}$ (from an edge) with the distributions
given by the cubic vertices.
Such a contraction of distributions is typically ill-defined.

A key innovation was the use of (partially) compactified configuration spaces to regularize these divergences appearing in the integrals of the diagrammatic expansion.
A physicist might call this a point-splitting regularization.

The basic idea is most easily seen in understanding $\d^{-1}$ itself.
Note that in the diagrams, the edge should correspond to an integral kernel for $\d^{-1}$,
and hence a de Rham current $P$ on $\RR^3 \times \RR^3$.
This current is smooth away from the diagonal $\Delta: \RR^3 \to \RR^3 \times \RR^3$,
so it determines a smooth differential form $P'$ on the configuration space $\Conf_2(\RR^3)$,
which is simply the complement of the diagonal.
(We mean here the configuration space of ordered points.)
The singular behavior of the current is confined to the diagonal.

Here the novel step happens.
There is a natural partial compactification $\bConf_2(\RR^3)$ of $\Conf_2(\RR^3)$ by the real blow-up along the diagonal of $\RR^3 \times \RR^3$;
in other words, we replace the diagonal by the sphere bundle for the normal bundle of the diagonal.
There is a natural map $\bar{\iota}: \bConf_2(\RR^3) \to (\RR^3)^2$ that extends the inclusion $\iota$ of $\Conf_2(\RR^3)$.
The nonobvious, crucial fact is that $P'$ extends to a smooth differential form $\bar{P}$ on $\bConf_2(\RR^3)$.
It has the explicit form
\[
\bar{P} = C \, \bar{\pi}^* \d vol_{S^2} ,
\]
where $C$ is a constant and $\bar{\pi}: \bConf_2(\RR^3) \to S^2$ is the extension of the Gauss map $(x,y) \mapsto (x-y)/|x-y|$ to the compactification.
Thus, if one wants to understand $\d^{-1}$, 
then one can use $\bar{P}$ as the integral kernel in the following way.
Given $\alpha \in \Omega^*(\RR^3)$, view it as a differential form on $(\RR^3)^2$ that is constant in the first factor of $\RR^3$, then pull it back to $\bar{\iota}^* \alpha$, then wedge with $\bar{P}$ and integrate out the along the second factor of $\RR^3$.
That is,
\[
\d^{-1} \alpha (x) = \int_{(\bar{\iota} \circ \pi_1)^{-1}(x)} \bar{P}(x,y) \wedge \alpha(y).
\]
The integrals here are all of smooth forms, so one need not worry about divergences.

A similar procedure allows one to understand the weight of an arbitrary graph~$\gamma$.
There is a Fulton--MacPherson-type partial compactification of $\Conf_{V(\gamma)}(\RR^3)$,
where $V(\gamma)$ denotes the number of vertices.
We follow closely the exposition of~\cite{KMV}.

\begin{dfn}
For a submanifold $Y \hookrightarrow X$, the {\em blowup} $\Bl(X,Y)$ denotes the result of replacing $Y$ with the sphere bundle of its normal bundle. 
It is diffeomorphic to removing an open tubular neighborhood of $Y$ inside~$X$.
\end{dfn}

In any product $M^k$, let $\Delta_k$ denote the thin diagonal $\{(x,\ldots,x) \in M^k\}$.
There is a natural inclusion
\[
\Conf_p(M) \hookrightarrow M^p \times \prod_{S \subset \{1,\ldots,p\}, |S| \geq 2} \Bl(M^S,\Delta_S)
\]
since any $p$-tuple of distinct points determines an $S$-tuple of distinct points for any subset~$S$.

\begin{dfn}
When $M$ is compact, let $\bConf_p(M)$ denote the closure of the image of this map.\footnote{\cite{KMV} use the notation $C(p,M)$ for $\Conf_p(M)$ and $C[p,M]$ for $\bConf_p(M)$.}
\end{dfn}

By construction, there is a canonical map $\bConf_p(M) \to M^p$, 
and so one can talk about the underlying $k$th point $\ev_k(x)$ of some point $x \in \bConf_p(M)$ by mapping to $M^p$ and then taking the $k$th coordinate.

When $M = \RR^3$, view it as the complement of the point $\infty$ in~$S^3$.
Consider then the map
\[
\Conf_p(\RR^3) \hookrightarrow \Conf_{p+1}(S^3)
\]
sending $p$ points $(x_1,\ldots,x_p)$ in $\RR^3$, 
to the $p+1$ points $(x_1,\ldots,x_p, \infty)$ in~$S^3$.
This feature suggests how to (partially) compactify~$\Conf_p(\RR^3)$.

\begin{dfn} 
Let $\bConf_p(\RR^3)$ denote the pullback of the diagram
\[
\begin{tikzcd}
& \bConf_{p+1}(S^3) \arrow[d, "\ev_{p+1}"]\\
\{\infty\} \arrow[r] & S^3 
\end{tikzcd}
\]
i.e., the subspace of $\bConf_{p+1}(S^3)$ whose final point is~$\infty$.
\end{dfn}

This partial compactification is a manifold with corners.
It is helpful to know that the codimension one faces of $\bConf_p(\RR^3)$ consist of subsets of points that simultaneously come together or escape to infinity. 
This dynamical language describes how to build a path from a point in $\Conf_p(\RR^3)$ to a point on the boundary.
Hence one can see that a point on the boundary consists of a $p$-tuple of points in $\RR^3$, with some repetitions, and the data of how to separate those repeated points to first order (i.e., tangent directions for each repeated point).
The corners arise from having a set of points collide and then later having more points join them, and so on, with each subsequent collision increasing the codimension.

We now use these spaces to define the analytic weight.
For each edge $e \in \gamma$, there is a natural map 
\[
\pi_e: \bConf_{V(\gamma)}(\RR^3) \to \bConf_2(\RR^3)
\]
by taking the points parametrized by the vertices at the boundary of the edge.
There is also a natural map
\[
\Phi: \bConf_{V(\gamma)}(\RR^3) \to (\RR^3)^k
\]
by taking the points parametrized by the external legs of the graphs, with $k$ the number of such legs.
For $\alpha \in \Omega^*_c((\RR^3)^k)$, we define
\[
w^{an}_\gamma(\alpha) = \int_{\bConf_{V(\gamma)}(\RR^3)} \Phi^*\alpha \wedge \bigwedge_{e \in E(\gamma)} \pi_e^* \bar{P}.
\]
This integral is manifestly well-defined.

In consequence, the perturbative quantization of Chern--Simons theory requires no counterterms.
The functional
\begin{equation}
\label{eqn: ren CS}
I^{\rm ren}_{CS} = W(\bar{P},I_{CS}) = \sum_\gamma \frac{1}{|\Aut(\gamma)|} w_\gamma^{an} w_\gamma^\g
\end{equation}
encodes the interaction term after integrating out the nonzero modes of the theory.
We will call it the {\em renormalized action functional} for Chern--Simons theory.

\begin{rmk}
The framework of \cite{CosBook} works for more general elliptic complexes
and relies on the existence of parametrices. 
More concretely, one can often use heat kernel methods.
As shown in Section 15.2 of \cite{CosCS}, 
the construction given here is the limit of the heat kernel construction as the lower length scale goes to zero.
\end{rmk}

\subsubsection{The quantum master equation}
\label{subsec: CS QME}

In a situation like Chern--Simons theory, 
one is not integrating just over a vector space viewed as a formal manifold---as we discussed when introducing the diagrammatic expansion---but over a formal stack.
Thus our diagrammatic expansion must be suitably compatible with the stacky structure.
In the BRST/BV formalism,
this compatibility is encoded 
as the {\em quantum master equation} (or QME).
A systematic discussion well-suited to our purposes can be found in Chapter 5 of \cite{CosBook},
but we review certain key points here.

In brief, the shifted Poisson structure on the classical field theory determines a distinguished second-order differential operator of cohomological degree one known as the {\em BV Laplacian}~$\Delta$. 

\begin{thm}[Prop. 15.2.6, \cite{CosCS}]
The renormalized action (\ref{eqn: ren CS}) satisfies the quantum master equation:
\[
(\d_{dR} + \{I^{\rm ren}_{CS}, -\} +\hbar \Delta)^2 = 0.
\]
\end{thm}

The argument intertwines the structure of the configuration spaces $\bConf_n(\RR^3)$ with the quantum master equation in an elegant way.
As we will mimic this argument for the defect theory,
we sketch it here.

The first step is to consider the equivalent assertion that
\[
\d_{dR} I^{\rm ren}_{CS} + \frac{1}{2} \{I^{\rm ren}_{CS},I^{\rm ren}_{CS}\} + \hbar \Delta I^{\rm ren}_{CS} 
\]
is a constant,
which is the form we work with.

The main idea is then to unpack the first term $\d_{dR} I^{\rm ren}_{CS}$ and see why it cancels off the bracket and BV Laplacian terms.
One must keep track of the Lie-theoretic weights, of course, 
but the key step is analytic---essentially just Stokes' theorem and graph combinatorics---so we focus on it here.

For any graph $\gamma$ with $k$ legs and for any  $\alpha \in \Omega^*_c((\RR^3)^k)$, 
we have by definition that
\[
w^{an}_\gamma(\d_{dR} \alpha) 
= \int_{\bConf_{V(\gamma)}(\RR^3)} (\d_{dR}\Phi^*\alpha) \wedge \bigwedge_{e \in E(\gamma)} \pi_e^* \bar{P}
\]
Using that $\d_{dR}$ is a derivation,
we rewrite this integral as a sum of two terms
\[
\int_{\bConf_{V(\gamma)}(\RR^3)} \d_{dR}\left(\Phi^*\alpha \wedge \bigwedge_{e \in E(\gamma)} \pi_e^* \bar{P}\right) 
- (\pm)\int_{\bConf_{V(\gamma)}(\RR^3)} \Phi^*\alpha \wedge \left(\d_{dR} \bigwedge_{e \in E(\gamma)} \pi_e^* \bar{P}\right)
\]
and then we rewrite this sum as
\begin{equation}
\label{eqn: CS QME expansion}
\int_{\partial\bConf_{V(\gamma)}(\RR^3)} \Phi^*\alpha \wedge \bigwedge_{e \in E(\gamma)} \pi_e^* \bar{P}
- \sum_{e \in E(\gamma)} (\pm) \int_{\bConf_{V(\gamma)}(\RR^3)} \Phi^*\alpha \wedge \pi_e^*\bar{K} \wedge \bigwedge_{e' \neq e} \pi_{e'}^* \bar{P},
\end{equation}
where the first term follows from Stokes' theorem and the second term uses that $\d_{dR} \bar{P} = \bar{K}$.

Now the second term, namely the sum over internal edges, can be broken up into separating edges (i.e., an edge whose removal breaks the graph into disconnected pieces) or nonseparating edges.
For a separating edge $e$, one can see that there is a term in $\{I^{\rm ren}_{CS}, I^{\rm ren}_{CS}\}$
that cancels the corresponding summand, 
since the bracket is given by connecting two interaction terms with $\bar{K}$.
For a nonseparating edge $e$, one can see that there is a term in $\Delta I^{\rm ren}_{CS}$
that cancels the corresponding summand,
since the BV Laplacian connects two legs of a single interaction term with~$\bar{K}$.

Hence it remains to show that the first term, namely the integral over $\partial \bConf(V(\gamma))$, vanishes.
But Kontsevich showed this, using symmetry arguments, in \cite{KonECM};
see Lemma 2.2 and the subsequent discussion.

\subsubsection{}

We have presented a BV quantization of Chern--Simons theory, 
but in the introduction, we asserted that there were many possible such quantizations,
parametrized by a choice of $\hbar$-dependent level.
Hence we need to see where they would appear in the process we just described.

Thankfully this is straightforward:
our classical action functional involves a choice of invariant symmetric bilinear form $\langle-,-\rangle_\g$ on $\g$,
and that choice determines the Lie-theoretic parts of the propagator and the 3-valent vertex.
We could instead pick any sequence of such pairings $(\langle-,-\rangle_{(n)})_{n > 0}$ and use
\[
\langle-,-\rangle_\g + \hbar \langle-,-\rangle_{(1)} + \cdots + \hbar^n \langle-,-\rangle_{(n)} + \cdots
\]
in place of $\langle-,-\rangle_\g$.
All our arguments for the BV quantization work {\it verbatim}, 
and so we see that such a choice of {\em level} determines a BV quantization of the classical Chern--Simons theory.\footnote{We do not require any of the $\langle-,-\rangle_{(n)}$ to be nondegenerate, 
just the leading term $\langle-,-\rangle_\g$.}

What is not obvious from this construction is whether these BV quantizations are distinct up to isomorphism.
One could imagine there is some $\hbar$-dependent automorphism of the space of fields that identifies the quantizations associated to two different levels.
Here one can take advantage of a bit of formalism developed in~\cite{CosBook},
namely the deformation complex of a classical BV theory.
As discussed in Section~\ref{subsec: def cplx for CS}, it allow us to see that these quantizations are distinct and justifies Theorem~\ref{thm: levels for CS}.

\subsection{The diagrammatics of the fully quantum defect theory}
\label{subsec: defect diagrams}

We now turn to perturbatively quantizing the whole coupled system (see the action \eqref{coupled action}),
i.e., we examine quantum Chern--Simons theory coupled to a quantum fermion.
Parallel to our development of pure Chern--Simons theory,
our first step is to show that no counterterms are needed.
That is, we show the Feynman diagrams admit no divergences.
We then examine the QME, which does not hold on the nose,
and recognize the ``framing anomaly.''

\subsubsection{Bicolored trivalent graphs}

The coupled theory has two interaction terms,
the Chern--Simons interaction and the coupling between the gauge field and the fermions.
We view the coupling as having in inward and outward directed leg for the $\psi$ and $\psib$, respectively.
Each field also contributes a propagator:
$P^A$ for the gauge fields, which we might draw with a wiggly line, 
and $P^f$ for the fermion fields, which we might draw as a directed edge.
Note that due to the directed nature of the fermionic propagators and the coupling interaction,
they always lie on a directed path or cycle.
Thus any graph can be constructed in two stages:
first specify the purely fermionic piece, which will be a disjoint union of oriented lines and wheels,
and then attach the fermionic piece to a Chern--Simons graph.

Note that the weight of a graph still factors into an analytic and a Lie-theoretic contribution.
Now the Lie-theoretic weight $w^{\g,\rho}$ depends on the Lie algebra $\g$ 
and the choice of finite-dimensional representation $\rho: \g \to \End(V)$.

\subsubsection{The configuration spaces}

To show that no regularization is needed for the analytical aspects,
we again use the configuration space method, 
modified to take account of the embedded 1-manifold.
Again, we follow closely the exposition of~\cite{KMV}.

Let $\iota: L \hookrightarrow \RR^3$ denote a proper embedding of a 1- (possibly with boundary). 
Let $(\ell; n)$ denote an element of $\NN \times \NN$ that we view as saying we want to talk about ``$\ell$ points on $L$ and $n$ points in $\RR^3$.''

\begin{dfn}
Let $\bConf_{\ell;n}(\iota: L \to \RR^3)$ denote the partial compactification of the space of configurations of $\ell$ points on $L$ and $n$ points in $\RR^3$ 
given by the pullback of the diagram 
\[
\begin{tikzcd}
& \bConf_{\ell+n}(\RR^3) \arrow[d, "{\rm pr}^\ell"]\\
\bConf_\ell(L) \arrow[r,"\iota^\ell"] & \bConf_{\ell}(\RR^3) 
\end{tikzcd}
\]
where ${\rm pr}^\ell$ denotes the projection onto the first $\ell$ points and where $\iota^\ell$ sends $(y_1,\ldots,y_\ell) \in \Conf_\ell(L)$ to $(\iota(y_1),\ldots,\iota(y_\ell)) \in \Conf_\ell(\RR^3)$,
and extends naturally to the compactifications.
\end{dfn}

The boundary of $\bConf_{\ell;n}(\iota: L \to \RR^3)$ consists of points that have collided but know how to separate infinitesimally to produce $\ell$ disjoint points along $L$ and $n$ disjoint points in the bulk.

This construction varies nicely over the space of proper embeddings $\Emb(L, \RR^3)$, 
and hence it naturally determines a smooth fiber bundle 
\[
\bpi_{\ell;n}: \bConf_{\ell;n}(L,\RR^3) \to \Emb(L,\RR^3)
\]
whose fiber over $\iota$ is $\bConf_{\ell;n}(\iota: L \to \RR^3)$.
This claim is an exceedingly minor generalization of Proposition 4.6 of \cite{KMV}.\footnote{\cite{KMV} use the notation $C[\vec{i}+s; \cL^n_m, \Gamma] \to \cL^n_m$ for this kind of space, 
where $\cL^n_m$ denotes a space of proper embeddings of $m$ disjoint copies of $\RR$ into $\RR^n$,
$\vec{i} = (i_1,\ldots,i_m)$ indicates we want $i_j$ points on the $j$th strand, and $s$ indicates we want $s$ points in $\RR^n$.}

\subsubsection{The configuration space integrals}
\label{subsec: defect integrals}

Let $\gamma$ denote a trivalent graph for the coupled theory.
We need a lot of notation:
\begin{itemize}
\item Let $E^f(\gamma)$ denote the set of fermionic edges, which are labeled by $P^f$.
\item Let $E^A(\gamma)$ denote the set of gauge edges, which are labeled by $P^A$.
\item Let $V^A_t(\gamma)$ denote the set of trivalent vertices with three $A$ legs.
\item Let $V^f_t(\gamma)$ denote the set of trivalent vertices with one $A$ leg and hence also a $\psi$ and $\psib$ leg.
\item Let $V^A_u(\gamma)$ denote the set of external $A$ legs (or univalent $A$-vertices).
\item Let $V^f_u(\gamma)$ denote the set of external fermionic legs (or univalent $\psi$- or $\psib$-vertices).
\end{itemize}
Set 
\[
\ell = |V^f_t(\gamma)|+ |V^f_u(\gamma)| \quad\text{ and }\quad n = |V^A_t(\gamma)|+|V^A_u(\gamma)|.
\]
We use $\bConf_{\gamma}(\iota: L \to \RR^3)$ to denote $\bConf_{\ell;n}(\iota: L \to \RR^3)$.

For such a trivalent graph $\gamma$, 
we get maps
\[
\phi^A_\gamma: \bConf_{\gamma}(\iota: L \to \RR^3) \to (S^2)^{|E^A(\gamma)|}
\]
and
\[
\phi^f_\gamma: \bConf_{\gamma}(\iota: L \to \RR^3) \to (S^0)^{|E^f(\gamma)|}.
\]
There is then a pullback form 
\[
\Omega_\gamma = \phi^{f*}_\gamma \omega^f \wedge \phi^{A*}_\gamma \omega^A,
\]
where
\[
\omega^f = \prod_{|E^f(\gamma)|} P^f
\]
and 
\[
\omega^A = \bigwedge_{|E^A(\gamma)|} \d vol_{S^2}.
\]
(Recall that the propagator for the fermion depends on the 1-manifold: see Section~\ref{sec on quantizing charged FF}.)
There is also a natural map 
\[
\Phi_\gamma: \bConf_\gamma(\iota) \to (S^3)^{|V_u^A|} \times  L^{|V_u^f|}
\]
by remembering the points of the external legs ({\em aka} univalent vertices).
(Note that a point can run to $\infty$ where $S^3 = \RR^3\cup\infty$.)

Finally, for any $\alpha \in \Omega^*_c((\RR^3)^{|V_u^A|} \times  L^{|V_u^f|})$, 
we define
\[
w^{an}_\gamma(\alpha) 
= \int_{\bConf_{\gamma}(\iota)} \Phi_\gamma^*\alpha \wedge \Omega_\gamma.
\]
This expression determines a local functional of polynomial degree $m + n + \bar{n}$ on the total space of fields,
which we will also denote~$w^{an}_\gamma$.

\begin{dfn}
\label{dfn: ren defect action}
The {\em renormalized action functional of the defect theory} is
\[
I_{\rm coup}^{\rm ren} = \sum_{\gamma} \frac{\hbar^{b_1(\gamma)}}{|\Aut(\gamma)|} w_\gamma^{an} w_\gamma^{\g,\rho},
\]
and it encodes integrating out the non-zero modes of the gauge and fermion fields.
\end{dfn}

We remark that this expression is similar to the usual configuration space integrals,
except that we allow external legs and 
thus produce a functional on the space of compactly-supported fields.

\subsubsection{The quantum master equation}

As we will now explain,
this renormalized action does {\em not} satisfy the quantum master equation.
The obstruction is exact, however, 
as it must be, by our computations of the deformation complex.
One can thus correct the renormalized action to obtain a solution to the QME.

We will not provide a full characterization of the obstruction or how to trivialize it.
Instead we will describe the first nontrivial term in the obstruction and explain why it is precisely the usual self-linking issue first identified in~\cite{BotTau}.

\begin{rmk}
We expect---but do not verify here---that the literature on configuration space integrals
(notably \cite{BotTau,Thu, AltFreVCS, CatCotLon1, KMV})
allows one to exhibit a full quantization of the coupled system.
The main technical difference between our situation and theirs is that we allow a more general class of graphs (notably those with external edges).
One must revisit the usual cancellation arguments and verify that 
\begin{itemize}
\item only integrals over anomalous faces matter and
\item these can be cancelled by a simple pre-factor taking into account the self-linking of each component of the embedded 1-manifold.
\end{itemize}
It would be appealing to synthesize our structural results with the rich, computational insights about configuration space integrals developed by these authors.
\end{rmk}

Analogously to Section \ref{subsec: CS QME}, 
we find that the obstruction to satisfying the QME depends on integrals over the boundary faces of the compactified configuration spaces.

\begin{lmm}\label{lmm: QME fails}
The renormalized action functional need not satisfy the quantum master equation.
The obstruction is  
\[
\sum_{\gamma} \frac{\hbar^{b_1(\gamma)}}{|\Aut(\gamma)|} w_\gamma^{\g,\rho} \int_{\partial \bConf_{\gamma}(\iota: L \to \RR^3)} \Omega_\gamma
\]
and it arises by an integral over the boundary faces of the configuration spaces.
\end{lmm}

\begin{proof}
We wish to compute the obstruction
\[
\d_{dR} I^{\rm ren}_{\rm coup} + \frac{1}{2} \{I^{\rm ren}_{\rm coup},I^{\rm ren}_{\rm coup}\} + \hbar \Delta I^{\rm ren}_{\rm coup} .
\]
Our approach is parallel to the treatment of pure Chern--Simons theory.
As seen in the computation leading to Equation~\eqref{eqn: CS QME expansion},
Stokes' theorem lets us understand $\d_{dR}I_{\rm coup}^{\rm ren}$ as a sum of two terms,
one given by integrating over the boundary $\partial\bConf_{\gamma}(\iota: L \to \RR^3)$ 
and the other by replacing each propagator by the integral kernel for the BV bracket.

More explicitly, let $\gamma$ be a graph for this defect theory. 
Then for any 
\[
\alpha \in \Omega^*_c((\RR^3)^{|V_u^A|} \times  L^{|V_u^f|}),
\]
Stokes' theorem gives us
\[
w^{an}_\gamma(\d_{dR} \alpha)  = \int_{\partial\bConf_{\gamma}(\iota)} \Phi^*\alpha \wedge \Omega_\gamma
- \sum_{e \in E(\gamma)} (\pm) \int_{\bConf_{\gamma}(\iota)} \Phi^*\alpha \wedge \pi_e^*K \wedge \bigwedge_{e' \neq e} \pi_{e'}^* P,
\]
where to an edge $e \in E^A$, we use $K^A$ to replace $P^A$, 
and on an edge $e \in E^f$, we use $K^f$ to replace $P^f$.

The other summand in the obstruction
\begin{equation}\label{eqn: second part of QME}
\frac{1}{2}\{I^{\rm ren}_{\rm coup},I^{\rm ren}_{\rm coup}\} + \hbar \Delta I^{\rm ren}_{\rm coup}.
\end{equation}
kills the term given by the sum over edges.
Hence only the integral over the boundary survives to provide an obstruction.
\end{proof}

\subsubsection{Self-linking}

The obstruction term does not vanish, in fact, 
due to a diagram that encodes the self-linking integral for~$L$.
Consider the one-loop diagram with two external legs: an incoming and outgoing fermion.
Label one edge by the gauge propagator and the other edge by the fermion heat kernel.
This diagram appears when one applies BV Laplacian for the fermions to the 2-vertex tree arising by RG flow.\footnote{There is another diagram with these external legs: one labels an edge by the gauge heat kernel and the other edge by the fermion propagator. As the gauge heat kernel is a 3-form, however, this diagram vanishes for degree reasons.}

For simplicity and concreteness, consider a knot $L: S^1 \hookrightarrow \RR^3$.
At a finite length scale~$t$, our integral is over the configuration space of two points in the knot $L$,
and the integrand is given by pulling back the volume form of $S^2$ along the Gauss map and also multiplying by the heat kernel $K_{\rm ff}[t]$ on $S^1$.
The $t \to \infty$ limit of $K_{\rm ff}[t]$ is a constant function, 
since $K_{\rm ff}[\infty]$ is the integral kernel for projection onto the zero modes of the Laplacian.
Thus the integrand becomes the pullback of the volume form of $S^2$ along the Gauss map,
and hence it computes the self-linking.

It is a standard fact that a choice of knot framing allows one to ``fix'' this integral by adding the total torsion of the knot.
The sum is then the linking number of the knot and its displacement by the framing.

Hence {\em if we choose a framing $\tau$ with zero linking}, 
then we can add a linear term 
\[
I_\tau(A) = \int_L \cL_\tau(L^* A)
\]
to the action functional.
(This term just integrates the pullback of $A$ along $L$ after infinitesimally displacing by~$\tau$.)
The BV bracket with this minimal coupling term determines a diagram
where one vertex is labeled by the minimal coupling $\int_L ( \psib, \rho(i^*A) \psi )$ and the other vertex is labeled by~$I_\tau$.
The associated integral is then the total torsion, which cancels the self-linking.

\bibliographystyle{amsalpha}
\bibliography{2025refs}

\end{document}